\documentclass[a4paper,12pt,reqno]{amsart}
\usepackage[english]{babel}
\usepackage{amsmath}
\usepackage{amscd}
\usepackage{amsgen}
\usepackage[final]{epsfig}
\usepackage{latexsym}
\usepackage{amsfonts}
\usepackage{amssymb}
\usepackage{amsthm}
\usepackage{url}
\usepackage{slashed}
\usepackage{stmaryrd}
\usepackage[all]{xy}

\allowdisplaybreaks

\catcode`\@=11
\@namedef{subjclassname@2010}{%
\textup{2010} Mathematics Subject Classification} \catcode`\@=12

\newtheorem{thm}{Theorem}[section]
\newtheorem{lemm}{Lemma}[section]

\newtheorem{corr}{Corollary}[section]

\newtheorem{ex}{Example}[section]
\newtheorem{prop}{Proposition}[section]
\newtheorem{rem}{Remark}[section]

\def\M{{\mathcal M}}

\def\f{{\frac{n}{2}}}
\def\fm{{\frac{n\!-\!1}{2}}}
\def\fp{{\frac{n\!+\!1}{2}}}
\def\r{{\mathbb R}}        %real numbers
\def\H{{\mathcal H}}

\def\D{{\mathcal D}}
\def\c{{\mathbb C}}        %complex numbers
\def\N{{\mathbb N}}        %natural numbers
\def\J{{\sf J}}
\def\T{{\mathcal T}}       %operator T
\def\OB{{\mathcal O}}      %obstruction
\def\S{{\mathbb S}}        %round sphere

\def\Rho{{\sf P}}          %Rho-tensor
\def\B{{\mathcal B}}       %Bach-tensor
\def\st{\stackrel{\text{def}}{=}}

\newcommand{\Hess}{\operatorname{Hess}}
\newcommand{\tr}{\operatorname{tr}}
\newcommand{\Ric}{\operatorname{Ric}}
\newcommand{\scal}{\operatorname{scal}}
\newcommand{\grad}{\operatorname{grad}}
\newcommand{\Res}{\operatorname{Res}}
\newcommand{\id}{\operatorname{id}}

\numberwithin{equation}{section}

\title{Explicit formulas for GJMS-operators and $Q$-curvatures}

\author{Andreas Juhl}

\address{Humboldt-Universit\"at, Institut f\"ur Mathematik,
Unter den Linden, D-10099 Berlin}

\email{ajuhl@math.hu-berlin.de}

\address{University Uppsala, Department of Mathematics,
P.O. Box 480, S-75106 Uppsala}

\email{andreasj@math.uu.se}

\begin{document}

\begin{abstract} We describe GJMS-operators as linear combinations of
compositions of natural second-order differential operators. These
are defined in terms of Poincar\'e-Einstein metrics and renormalized
volume coefficients. As special cases, we derive explicit formulas
for conformally covariant third and fourth powers of the Laplacian.
Moreover, we prove related formulas for all Branson's
$Q$-curvatures. The results settle and refine conjectural statements
in earlier works. The proofs rest on the theory of residue families
introduced in \cite{juhl-book}.
\end{abstract}


\maketitle

\centerline \today

\footnotetext{The work was supported by grant 621--2003--5240 of the
Swedish Research Council VR and SFB 647 ``Space-Time-Matter'' at
Humboldt-University Berlin.}

\tableofcontents

%\newpage

%\renewcommand{\thefootnote}{1}

%%%%%%%%%%%%%%%%%%%%%%%%%%%%%%%%%%%%%%%%%%%%%%%%%%%%%%%%%%%%%%%%%%%%%%%%%%%
\section{Introduction and formulation of the main results}\label{intro-sum}

In the present work, we discuss the structure of certain geometric
differential operators of high order on Riemannian manifolds $M$
which are known as GJMS-operators. These operators are high-order
generalizations of the Yamabe-operator\footnote{Here and throughout
we use the convention that $-\Delta = \delta d$ is non-negative.}
\begin{equation}\label{Y}
P_2(g) \st \Delta_g \!-\! \left(\f\!-\!1\right) \J_g, \quad \J_g =
\frac{\scal(g)}{2(n\!-\!1)}.
\end{equation}
All GJMS-operators $P_{2N}(g)$ are natural differential operators of
the form
\begin{equation}\label{power}
\Delta_g^N + \mbox{lower-order terms}
\end{equation}
which are conformally covariant in the sense that
\begin{equation}\label{covar}
e^{(\f+N)\varphi} P_{2N}(e^{2\varphi}g)(u) = P_{2N}(g)
(e^{(\f-N)\varphi} u), \; u \in C^\infty(M)
\end{equation}
for all $\varphi \in C^\infty(M)$. The operators $P_{2N}(g)$ were
constructed by Graham, Jenne, Mason and Sparling in the seminal work
\cite{GJMS}. They are induced, in a certain sense, by the powers of
the Laplacian of the Fefferman-Graham ambient metric \cite{FG-final}
on a space of {\em two} higher dimensions.

Operators which satisfy the conditions \eqref{power} and
\eqref{covar} are often referred to as conformally covariant powers
of the Laplacian. Since these two conditions do not uniquely
determine such operators, the GJMS-operators are to be considered as
specific constructions of conformally covariant powers of the
Laplacian.

The second operator in the sequence $P_2, P_4, P_6, \dots$ of
GJMS-operators is known as the Paneitz-operator \cite{pan}. It is
given by
\begin{equation}\label{Pan}
P_4 = \Delta^2 + \delta ((n\!-\!2) \J - 4 \Rho) d +
\left(\f-2\right)\left(\f \J^2 - 2 |\Rho|^2 - \Delta \J\right),
\end{equation}
where $\Rho$ denotes the Schouten tensor, i.e., $(n-2) \Rho = \Ric -
\J g$. In \eqref{Pan}, the tensor $\Rho$ is regarded as an
endomorphism on one-forms using $g$ and $|\Rho|^2 = \Rho_{ij}
\Rho^{ij}$.

On manifolds of odd dimension $n$, there are GJMS-operators $P_{2N}$
of any order $2N \ge 2$. But on manifolds of even dimension $n$, the
construction in \cite{GJMS} is obstructed at order $n$, i.e., does
not give a conformally covariant power of Laplacian of order $>n$
for general metrics. In fact, by \cite{G-non} (for $n=4$ and $N=3$)
and \cite{GH} (in general) no such operator exists for general
metrics. However, for locally conformally flat metrics and
conformally Einstein metrics, the construction in \cite{GJMS} yields
conformally covariant powers of the Laplacian of any even order. In
particular, we have an infinite sequence of such operators for all
round spheres $\S^n$.

The GJMS-operators play a central role in conformal differential
geometry and geometric analysis. More information on these aspects
can be found in the recent monographs \cite{DGH-book},
\cite{juhl-book}, \cite{BJ} and the references therein.

The complexity of lower-order terms in \eqref{power} grows
exponentially with $N$. Therefore, finding an explicit description
of these terms seems to be an almost impossible task. An additional
problem is that the lower-order terms can be written in many
different ways. Nevertheless, in the present work we shall solve
this problem by providing a description of the GJMS-operators (and
of the related Branson's $Q$-curvatures) which is both almost
explicit and esthetically appealing. The main insight is that the
enormous complexity of these operators can be described in terms of
a sequence of natural {\em second-order} operators.

We continue with the formulation of the main results.

For this purpose, we need to introduce some more notation.

First, we shall use the following combinatorial conventions. A
sequence $I=(I_1,\dots,I_r)$ of integers $I_j \ge 1$ will be
regarded as a composition of the natural number
$|I|=I_1+I_2+\cdots+I_r$. In other words, compositions are
partitions in which the order of the summands is considered. $|I|$
will be called the size of $I$. To any composition
$I=(I_1,\dots,I_r)$, we associate the numbers
\begin{equation*}
m_I = - (-1)^r |I|! \, (|I|\!-\!1)! \prod_{j=1}^r \frac{1}{I_j! \,
(I_j\!-\!1)!} \prod_{j=1}^{r-1} \frac{1}{I_j \!+\! I_{j+1}}
\end{equation*}
and
\begin{equation*}
n_I = \prod_{j=1}^r \binom{\sum_{k \le j} I_k -1}{I_j-1}
\binom{\sum_{k \ge j} I_k -1}{I_j-1}.
\end{equation*}
Note that $m_{(N)} = n_{(N)} = 1$ for all $N \ge 1$ and
$n_{(1,\dots,1)} = 1$. Moreover, for any composition $I$, we define
the operator
$$
P_{2I} = P_{2I_1} \circ \cdots \circ P_{2I_r}.
$$
Then we set
\begin{equation}\label{M-defin}
\M_{2N} = \sum_{|I|=N} m_I P_{2I} = P_{2N} + \mbox{compositions of
lower-order GJMS-operators}.
\end{equation}
These sums contain $2^{N-1}$ terms. Similarly, we define
$$
\M_{2I} = \M_{2I_1} \circ \cdots \circ \M_{2I_r}.
$$

Second, the present theory rests on the notion of
Poincar\'e-Einstein metrics
$$
g_+ = r^{-2}(dr^2 + g_r), \; g_0 = g
$$
in the sense of Fefferman and Graham (see \cite{cartan} and
\cite{FG-final}). Here $g_r$ is a family of metrics on $M$ so that
$g_+$ is Einstein on the space $X = M \times (0,\varepsilon)$ in the
sense that
\begin{equation}\label{Ein}
\Ric(g_+) + n g_+ = 0.
\end{equation}
In particular, $g_+$ has negative scalar curvature $-n(n+1)$. This
concept will be recalled in Section \ref{RF}. The volume form of a
Poincar\'e-Einstein metric $g_+$ relative to $g$ takes the form
$$
dvol(g_+) = r^{-n-1} v(r) dvol(g) dr
$$
with
$$
v(r) = 1 + v_2 r^2 + v_4 r^4 + \cdots + v_n r^n + \cdots;
$$
here $v_n =0$ if $n$ is odd. In odd dimensions $n$, all coefficients
$v_2, v_4, \dots$ are scalar-valued curvature invariants of $g$.
Similarly, in even dimensions $n$, the coefficients $v_2, \dots,
v_n$ are uniquely determined by $g$. The functionals $g \mapsto
v_{2j}(g)$ are known as renormalized volume coefficients (see
\cite{G-vol}, \cite{G-ext}) or holographic coefficients (see
\cite{juhl-book}). In even dimensions $n$, the quantity $v_n$ is the
(infinitesimal) conformal anomaly of the renormalized volume of
$g_+$ (see \cite{G-vol}).\footnote{Sometimes $v_n$ is called the
{\em holographic} anomaly.} Now we set
$$
w(r) = \sqrt{v(r)} = 1 + w_2 r^2 + w_4 r^4 + \cdots + w_n r^n +
\cdots
$$
with $w_n = 0$ for odd $n$. The coefficients $w_2,w_4,w_6,\dots$ are
polynomials in $v_2,v_4,v_6,\dots$. In particular,
$$
2 w_2 = v_2, \quad 8 w_4 = 4 v_4 - v_2^2 \quad \mbox{and} \quad 16
w_6 = 8 v_6 - 4 v_4 v_2 + v_2^3.
$$

Finally, we introduce the generating function
\begin{equation}\label{def-H}
\H(r) \st \sum_{N \ge 1} \M_{2N} \frac{1}{(N\!-\!1)!^2}
\left(\frac{r^2}{4}\right)^{N-1}
\end{equation}
of the sequence $\M_2,\M_4,\dots$. Here we use the following
convention. For even $n$, the sum in \eqref{def-H} is defined to run
only up to $r^{n-2}$. For odd $n$, GJMS-operators are well-defined
for any order $2N \ge 2$, and the sum in \eqref{def-H} is regarded
as a formal infinite power series.

The first main result gives formulas for all GJMS-operators.

\begin{thm}\label{main-P} All GJMS-operators on Riemannian manifolds
$(M,g)$ of dimension $n \ge 3$ can be written in the form
\begin{equation}\label{main-deco}
P_{2N} = \sum_{|I|=N} n_I \M_{2I} = \M_2^N + \mbox{compositions with
fewer factors}.
\end{equation}
Moreover, all operators $\M_{2N}(g)$ are (at most) second-order and
are given by the formula
\begin{equation}\label{GF-form}
\H(g)(r) = - \delta (g_r^{-1} d) - r^{-2} \left( \Delta_{g_+}(\log
w) - |d \log w|^2_{g_+} \right).
\end{equation}
Here $\delta$ is defined with respect to $g$, and we regard $g_r$ as
an endomorphism on one-forms using $g$. Finally, $w$ is regarded as
a function on $X$.
\end{thm}

The relation \eqref{GF-form} is to be understood as an identity of
formal power series asserting that the coefficients on both sides
coincide.

Theorem \ref{main-P} states that all GJMS-operators of $(M,g)$ are
generated by the Taylor coefficients of the one-parameter family $r
\mapsto \D(g)(r)$ of second-order differential operators on $M$
defined by the right-hand side of \eqref{GF-form}. In view of
$\D(g)(0) = P_2$, this family can be viewed as a (geometrically
defined) perturbation of the Yamabe operator.\footnote{It is also
tempting to consider it as an operator $C^\infty(M) \to C^\infty (M
\times [0,\varepsilon))$.}

The formulation of Theorem \ref{main-P} does not require explicit
formulas for the Poincar\'e-Einstein metric $g_+$, i.e., does not
require to solve the Einstein equation \eqref{Ein}. But, of course,
any such knowledge makes \eqref{GF-form} more explicit.

In particular, for a {\em locally conformally flat} metric $g$, we
have
$$
(g_r)_{ij}=g_{ij}-r^2 \Rho_{ij}+\frac{r^4}{4} \Rho_{ik}\Rho^k_j,
$$
and the construction of GJMS-operators is {\em not} obstructed for
orders exceeding the dimension of the underlying manifold
\cite{FG-final}. Since $g_r$ takes the form $(1-r^2/2\Rho)^2$, when
regarded as a linear operator on one-forms using $g$, we find
$$
g_r^{-1} = (1-r^2/2 \Rho)^{-2} = \sum_{N \ge 1} N \Rho^{N-1}
\left(\frac{r^2}{2}\right)^{N-1}
$$
and
$$
w(r) = \sqrt{\det(1-r^2/2 \Rho)}.
$$
In this special case, the assertions in Theorem \ref{main-P} extend
to {\em all} orders. \eqref{GF-form} shows that the operator
$\M_{2N}$ has the principal part
\begin{equation}\label{principal-flat}
-(N\!-\!1)! N! 2^{N-1} \delta(\Rho^{N-1}d)
\end{equation}
and that the zeroth-order term of $\M_{2N}$ equals the sum
\begin{equation}\label{CT-flat}
-(N\!-\!1)!^2 2^{N-1} \left( \left(\f\!-\!N\right) \tr (\Rho^N) +
\frac{1}{2} \sum_{a=1}^{N-1} \tr (\Rho^a) \tr (\Rho^{N-a})\right)
\end{equation}
(modulo contributions involving derivatives). In combination with
\eqref{main-deco}, these results yield explicit formulas for all
GJMS-operators in terms of $g$ and $\Rho$. In the special case of
the round sphere $\S^n$, we have
$$
g_r^{-1} = (1\!-\!r^2/4)^{-2} \id \quad \mbox{and} \quad w(r) =
(1\!-\!r^2/4)^{\f},
$$
and \eqref{GF-form} simplifies to
$$
\H(r) = (1\!-\!r^2/4)^{-2} \left( \Delta - \f \left(\f\!-\!1
\right)\right) = \left( \sum_{N \ge 1} N (r^2/4)^{N-1} \right) P_2.
$$
In other words, the second part of Theorem \ref{main-P} states that
$$
\M_{2N} = N!(N\!-\!1)! P_2  \quad \mbox{for $N \ge 1$.}
$$
These summation formulas were proved in \cite{juhl-power} by using
the product formula \eqref{Einstein} for the GJMS-operators of
$\S^n$ (see \cite{branson}, \cite{G-stereo}).

For general metrics, the Taylor coefficients of $w(r)$ can be
expressed in terms of renormalized volume coefficients which, in
turn, can be written in terms of the Schouten tensor $\Rho$ and
Graham's extended obstruction tensors $\Omega^{(k)}$ (see Section 2
of \cite{G-ext}). Note that $\Omega^{(1)} = \B/(4-n)$, where $\B$ is
the Bach tensor. In low-order cases, the description of $P_{2N}$ in
Theorem \ref{main-P} can be made more explicit. The corresponding
results for $P_6$ and $P_8$ are stated in Section \ref{low}.

Theorem \ref{main-P} describes any GJMS-operator in terms of
second-order operators. Alternatively, the second part of Theorem
\ref{main-P} can be rephrased by stating that any GJMS-operator
$P_{2N}$ is the sum of the linear combination
$$
-\sum_{|I| = N, \, I \ne (N)} m_I P_{2I}
$$
of compositions of lower-order GJMS-operators and the second-order
operator given by the coefficient of
$$
\frac{1}{(N\!-\!1)!^2} \left(\frac{r^2}{4}\right)^{N-1}
$$
in the Taylor series of $\D(g)(r)$. We shall refer to such
identities as to {\em recursive} formulas for GJMS-operators.

Finally, we emphasize the following fact. The definition
\eqref{M-defin} shows that, for any $N \ge 1$, the operators
$\M_{2N}$ are of order $\le 2N$. The result that the operators
$\M_{2N}$ are actually only of second order can be regarded as a
remarkable {\em cancellation} effect.

The second main result of the present paper is a related universal
formula for all Branson's $Q$-curvatures. We recall that, for even
$n$ and $2N < n$, the GJMS-operator $P_{2N}$ gives rise to a scalar
curvature quantity $Q_{2N}$ via the formula
\begin{equation}\label{Q-def}
P_{2N}(g)(1) = (-1)^N \left(\f-N\right) Q_{2N}(g).
\end{equation}
The critical $Q$-curvatures $Q_n$ (for even $n$) can be defined
through the non-critical $Q$-curvatures $Q_{2N}$ ($2N < n$) by a
limiting process. For odd $n$, the quantities $Q_{2N}$ are defined
for all $N \ge 1$ by \eqref{Q-def}. For more information and a
discussion of the significance of $Q$-curvature we refer to
\cite{branson}, \cite{br-last}, \cite{origin}, \cite{DGH-book},
\cite{juhl-book} and \cite{BJ}.

\begin{thm}\label{main-Q} For even $n$ and $2N \le n$ and for odd $n$
and all $N \ge 1$,
\begin{equation}\label{magic-Q}
(-1)^N Q_{2N} = \sum_{|I|+a=N} n_{(I,a)} a! (a\!-\!1)! 2^{2a}
\M_{2I}(w_{2a}).
\end{equation}
\end{thm}

Here we use the convention that the sum on the right-hand side of
\eqref{magic-Q} contains a contribution by $w_{2N}$ (for the trivial
composition $I$). Note that the number of terms in that sum is
$2^{N-1}$, i.e., it grows exponentially with $N$.

The formula \eqref{magic-Q} follows by combining the first part of
Theorem \ref{main-P} with structural results for the zeroth-order
terms $\M_{2N}(1)$ of the operators $\M_{2N}$ (Theorem \ref{basic}
and Theorem \ref{basic-2}). In turn, the proofs of these results
utilize the universal recursive formula
\begin{equation}\label{URQ}
\sum_{|I|+a=N} m_{(I,a)} (-1)^a P_{2I}(Q_{2a}) = N! (N\!-\!1)!
2^{2N} w_{2N}
\end{equation}
for $Q$-curvature proved in \cite{Q-recursive}.

For locally conformally flat metrics, Theorem \ref{main-Q} extends
to all $Q$-curvatures $Q_{2N}$.

Of course, Theorem \ref{main-Q} shows its full power only when it is
combined with the explicit formula for the operators $\M_{2N}$ in
the second part of Theorem \ref{main-P}. It is interesting to
compare \eqref{magic-Q} and \eqref{URQ}. First of all, we note that
both formulas are {\em universal} in the dimension $n$. Next, the
formula \eqref{magic-Q} actually should be regarded as a resolution
of the recursive structure of $Q$-curvatures described in
\eqref{URQ}. If we split off the contributions of the trivial
composition, then \eqref{magic-Q} and \eqref{URQ} take the
respective forms
$$
(-1)^N Q_{2N} = \sum_{|I|+a=N  \atop  a \le N-1} n_{(I,a)}
a!(a\!-\!1)!2^{2a} \M_{2I}(w_{2a}) + N!(N\!-\!1)! 2^{2N} w_{2N}
$$
and
$$
(-1)^N Q_{2N} + \sum_{|I|+a=N \atop a \le N-1} m_{(I,a)} (-1)^a
P_{2I}(Q_{2a}) = N!(N\!-\!1)! 2^{2N} w_{2N}.
$$
In both formulas, we have collected terms of the same nature on the
same sides. The latter identities clearly show that \eqref{magic-Q}
and \eqref{URQ} express the differences
$$
(-1)^N Q_{2N} - N!(N\!-\!1)! 2^{2N} w_{2N}
$$
in two conceptually different ways.\footnote{These formulas simplify
even a bit further if the coefficients $w_{2N}$ are defined as the
coefficients of $w(r)$ as a power series in $r/2$. However, we
prefer to use the same notation as in \cite{Q-recursive}.} For more
details concerning the low-order cases $N \le 4$ we refer to Section
\ref{Q-curv-ex}.

Note that for $a=1$ and the composition $I=(1,\dots,1)$ of size
$N\!-\!1$ the right-hand side of \eqref{magic-Q} contains the
contribution $-\Delta^{N-1} \J$. This reproduces the well-known fact
\cite{branson} that the contribution to $Q_{2N}$ which contains the
maximal number of derivatives is $(-1)^{N-1} \Delta^{N-1} \J$.

Finally, the relations \eqref{magic-Q} and \eqref{URQ} should be
compared with the holographic formulas of \cite{holo} and
\cite{holo-II} for the differences
$$
(-1)^N Q_{2N} - N!(N\!-\!1)! 2^{2N-1} v_{2N}.
$$
In the critical case $2N=n$, the holographic formula shows that the
latter difference is an exact divergence. Hence for closed manifolds
the global conformal invariants
\begin{equation*}
\int_M Q_n dvol \quad \mbox{and} \quad \int_M v_n dvol
\end{equation*}
are proportional \cite{GZ}. This fact plays a basic role, for
instance, in \cite{GrH}. A major difference between the holographic
formulas and the formulas \eqref{magic-Q}, \eqref{URQ} is that the
metric variations of the Laplacian which contribute to the former
relations (through the asymptotic expansions of eigenfunctions) do
{\em not} appear in the latter ones.

Theorem \ref{main-P} and Theorem \ref{main-Q} yield {\em explicit}
formulas for GJMS-operators and $Q$-curvatures in the sense that
both objects are described in terms of (the Taylor coefficients of)
associated Poincar\'e-Einstein metrics. In order to derive fully
explicit expressions for GJMS-operators and $Q$-curvatures in terms
of the metrics, it suffices to combine them with corresponding
formulas for the Poincar\'e-Einstein metrics. The natural building
blocks of such formulas are Graham's extended obstruction tensors
\cite{G-ext}.

Finally, we note that in \cite{G-P} Gover and Peterson developed an
alternative approach to GJMS-operators and $Q$-curvatures using
tractor calculus. A remarkable feature of their method is that it
yields explicit formulas in terms of the metric without solving the
Einstein equation \eqref{Ein}.

The structure of the paper is as follows. Section \ref{IF} contains
the proof of the first assertion of Theorem \ref{main-P}. In the
following three sections we recall the notion of residue families
introduced in \cite{juhl-book}, prove their factorization relations
in full generality (Theorem \ref{Factor-A} and Theorem
\ref{Factor-B}), and derive from these a representation formula for
these families in terms of GJMS-operators (Theorem \ref{D-rep}). It
is this representation formula which explains the role of the
operators $\M_{2N}$. Only by reasons of degree, Theorem \ref{D-rep}
implies two basic properties of the sums $\M_{2N}$ which we refer to
as the {\em restriction property} (Theorem \ref{rest-prop}) and the
{\em commutator relations} (Theorem \ref{base-CR}). The commutator
relations are the main device to prove that the operators $\M_{2N}$
are second-order and to derive the formula for their principal parts
(Theorem \ref{base-c}). An important tool in the proof of this
formula is Theorem \ref{double}. The latter result might be also of
independent interest. In Section \ref{structure}, we determine the
the zeroth-order terms of the operators $\M_{2N}$. We find (Theorem
\ref{basic} and Theorem \ref{basic-2}) that these curvature
quantities satisfy recursive relations. These relations are closely
related to the recursive structure of $Q$-curvatures. The results
lead to the proof of the second part of Theorem \ref{main-P}, when
combined with the description of the principal part of $\M_{2N}$ in
Theorem \ref{base-c}. Since these arguments do not suffice to cover
also the super-critical GJMS-operators for locally conformally flat
metrics, Theorem \ref{mu-second} provides an alternative proof of
Theorem \ref{basic-2} (for such metrics) which does extend to the
super-critical cases. It does not refer to $Q$-curvature. Instead,
it rests on an evaluation of the restriction property and the
commutator relations, i.e., is similar in spirit to the proof of the
formula for the leading terms of $\M_{2N}$. Finally, Theorem
\ref{CT-f} derives the contributions described in \eqref{CT-flat}.
Section \ref{Q-curv-ex} is devoted to the proof of Theorem
\ref{main-Q}. As illustrations of the general theory, and for the
convenience of the reader interested mainly in low-order special
cases, we make explicit in Section \ref{low}, the results in the
special cases of the conformally covariant third and fourth power of
the Laplacian. In Section \ref{open}, we comment on a number of
scattered issues, sketch some further results and point out a few
open problems. The paper finishes with an appendix which contains
direct proofs of some low-order special cases of Theorem
\ref{double} and a description of a resulting algorithm which
reproduces Graham's formulas for Poincar\'e-Einstein metrics in
terms of the Schouten tensor and the extended obstruction tensors.

Some early steps of the work were influenced by discussions with
Vladimir Sou\u{c}ek and Petr Somberg during a visit at Charles
University (Prague) in spring 2010. Parts of the results were
presented at the workshop on {\em Parabolic Geometries and Related
Topics} of the University of Tokyo in November 2010, at Toronto
University in April 2011 and at the workshop {\em Cartan
Connections, Geometry of Homogeneous Spaces, and Dynamics} at the
Erwin Schr\"{o}dinger Institute (ESI) in Vienna in July 2011. We
gratefully acknowledge the support by these institutions. The proof
of the inversion formula in Section \ref{IF} is due to Christian
Krattenthaler (Vienna). It is a pleasure to thank him for allowing
to include his proof here. Many stages of the work involved computer
experiments. In particular, we derived structural insight from
numerical data and used the computer to test predictions. The
corresponding calculations were performed using Mathematica with the
NCAlgebra-package\footnote{\url{http://math.ucsd.edu/~ncalg/}}. We
thank Carten Falk (Humboldt Universit\"{a}t, Berlin) for his invaluable
efforts to create the numerous programs. Finally, we are grateful to
Robin Graham for a discussion of subtle questions concerning
super-critical GJMS-operators.

%%%%%%%%%%%%%%%%%%%%%%%%%%%%%%%%%%%%%%%%%%%%%%%%%%%%%%%%%%%%%%%%%%%%%%%
\section{The inversion formula}\label{IF}

We first recall some basic definitions from Section \ref{intro-sum}.
For any composition $I=(I_1,\dots,I_r)$, we set
\begin{equation}\label{m-form}
m_I = -(-1)^r |I|!\,(|I|\!-\!1)! \prod_{j=1}^r \frac{1}{I_j! \,
(I_j\!-\!1)!} \prod_{j=1}^{r-1} \frac{1}{I_j \!+\! I_{j+1}}
\end{equation}
and
\begin{equation}\label{n-form}
n_I = \prod_{j=1}^r \binom{\sum_{k \le j} I_k -1}{I_j-1}
\binom{\sum_{k \ge j} I_k -1}{I_j-1} \in \N.
\end{equation}
Then $m_{(N)} = n_{(N)} = 1$ for all $N \ge 1$ and
$$
m_{(I_1,I_2)} = - \binom{N-1}{I_1} \binom{N-1}{I_2} = -
n_{(I_1,I_2)}
$$
for all compositions $I=(I_1,I_2)$ of size $N \ge 2$. In the
following, we shall also work with an alternative expression for
$n_I$ with reduced fractions. In fact, an easy calculation shows
that
\begin{equation}\label{eq:n_I}
n_I = (|I|\!-\!1)!^2 \prod_{j=1}^{r} \frac{1}{(I_j\!-\!1)!^2}
\prod_{j=1}^{r-1} \frac{1}{\left(\sum_{k=1}^j I_k \right)
\left(\sum_{k = j+1}^r I_k\right)}.
\end{equation}
Note also that
\begin{equation}\label{inverse-c}
m_{I^{-1}} = m_I \quad \mbox{and} \quad n_{I^{-1}} = n_I,
\end{equation}
where $I^{-1} = (I_r,\cdots,I_1)$ is the inverse composition of $I =
(I_1,\cdots,I_r)$.

Moreover, we define
\begin{equation}
P_{2I} = P_{2I_1} \circ \cdots \circ P_{2I_r}
\end{equation}
and set
\begin{equation}\label{M-def}
\M_{2N} = \sum_{|I|=N} m_I P_{2I}.
\end{equation}
The latter sum contains $2^{N-1}$ terms. Note that $\M_2 = P_2$.
Since all GJMS-operators are formally self-adjoint \cite{GZ},
Eq.~\eqref{inverse-c} implies that the operators $\M_{2N}$ are
formally self-adjoint, too. Finally, we set
\begin{equation}
\M_{2I} = \M_{2I_1} \circ \cdots \circ \M_{2I_r}.
\end{equation}
By \eqref{inverse-c}, we also have $\M_{2I} = \M_{2I}^*$.

The following result states in which sense the operators $\M_{2N}$
can be regarded as the building blocks of the GJMS-operators.

\begin{thm}{\bf (Inversion Formula)}\label{duality} For $N \ge 1$,
we have
\begin{equation}\label{dual}
P_{2N} = \sum_{|I|=N} n_I \M_{2I}.
\end{equation}
\end{thm}

We refer to Theorem \ref{duality} as an inversion formula since it
can be regarded as a description of the inverse of the
transformation
\begin{equation}\label{integral}
\left\{ P_2,P_4,P_6,\cdots \right\} \to \left\{\M_2,\M_4,\M_6,\cdots
\right\}
\end{equation}
defined by \eqref{M-def}. The inversion formula is an algebraic fact
which relates two sets of operators. Therefore, it is justified to
omit here the assumptions which guarantee the existence of the
GJMS-operators.

Note that the sums on the right-hand sides of \eqref{dual} contain
$2^{N-1}$ terms, respectively.

Note also that, in view of \eqref{eq:n_I}, the inversion formula can
be written in the alternative form
\begin{equation}
\frac{P_{2N}}{(N\!-\!1)!^2} = \sum_{|I|=N} \frac{1}{q_I}
\frac{\M_{2I}}{\prod_{j=1}^r (I_j\!-\!1)!^2}
\end{equation}
with
$$
q_I = \prod_{j=1}^{r-1} \left(\sum_{k=1}^j I_k \right) \left(\sum_{k
= j+1}^r I_k\right).
$$

\begin{proof}{(\cite{K})} We substitute \eqref{M-def} in the right-hand
side of \eqref{dual} to obtain
\begin{equation} \label{eq:2}
\sum _{\vert I \vert=N}^{} n_I \sum_{\vert J_1 \vert = I_1, \dots,
\vert J_r \vert=I_r}^{} \bigg( \prod_{k=1}^{r} m_{J_k} \bigg)
P_{2J_1} \circ P_{2J_2} \circ \cdots \circ P_{2J_r}.
\end{equation}
Let $K=(K_1,K_2,\dots,K_s)$ be a fixed composition of $N$. Then the
coefficient of $P_{2K}$ in \eqref{eq:2} is given by
\begin{multline}\label{n-m}
\sum_{A\subseteq[s-1]}^{} n_{(K_1 + \dots + K_{a_1},
K_{a_1+1}+\dots+K_{a_2},\dots ,K_{a_{\ell-1}+1}+\dots+K_{a_\ell},
K_{a_{\ell}+1}+\dots+K_{s})} \\
\cdot m_{(K_1,\dots,K_{a_1})} m_{(K_{a_1+1},\dots,K_{a_2})}\cdots
m_{(K_{a_{\ell-1}+1},\dots,K_{a_\ell})}
m_{(K_{a_{\ell}+1},\dots,K_{s})}.
\end{multline}
The sum in \eqref{n-m} runs over all subsets $A =
\{a_1,a_2,\dots,a_\ell\} \subseteq \{1,2,\dots,s-1\}$ (including the
empty set) whereby we use the convention that
$a_1<a_2<\dots<a_\ell$. Using \eqref{m-form} and \eqref{eq:n_I},
after some simplification this expression becomes
\begin{multline}\label{eq:3}
(-1)^{s+1} (N\!-\!1)!^2 \prod_{j=1}^{s} \frac{1}{K_j! (K_j\!-\!1)!}
\prod_{j=1}^{s-1} \frac{1}{(K_j + K_{j+1})} \\
\times \sum_{A \subseteq[s-1]}^{} (-1)^{\vert A\vert}
(K_1+\dots+K_{a_1}) (K_{a_1+1}+\dots+K_{a_2}) \cdots (K_{a_{\ell}+1}+\dots+K_{s}) \\
\cdot \frac {\prod_{a \in A}^{} (K_a+K_{a+1})}{\prod_{a\in A}^{}
(K_1+\dots+K_a)(K_{a+1}+\dots+K_s)},
\end{multline}
where $|A|$ stands for the cardinality of the set $A$, as usual. We
have to show that the sum in \eqref{eq:3} vanishes for $s>1$. That
it equals $1$ for $s=1$ is obvious. The following lemma implies our
claim (by setting $X=Y=0$).

\begin{lemm}\label{lem:1} For $s>1$, we have
\begin{multline} \label{eq:4}
\sum_{A\subseteq[s-1]}^{} (-1)^{\vert A\vert} (K_1+\dots+K_{a_1})
(K_{a_1+1}+\dots+K_{a_2}) \cdots (K_{a_{\ell}+1}+\dots+K_{s}+X) \\
\cdot \frac{\prod_{a\in A}^{} (K_a+K_{a+1}+Y\cdot \chi(a=s-1))} {
\prod_{a\in A}^{} (K_1+\dots+K_a)(K_{a+1}+\dots+K_s)}\\
= - \frac{X (K_1+\dots+K_{s-1}) +Y (K_s+X)}{K_2+\dots+K_s},
\end{multline}
where $\chi(\mathcal S)=1$ if $\mathcal S$ is true and
$\chi(\mathcal S)=0$ otherwise.
\end{lemm}

\begin{proof} We prove the claim by induction on $s$. For $s \ge 1$,
let $f(s;X,Y)$ denote the left-hand side of \eqref{eq:4}. In
particular, we have
\begin{equation}\label{eq:5}
f(1;X,Y)=K_1+X.
\end{equation}

The claim \eqref{eq:4} can be readily verified for $s=2$. Now we
suppose that \eqref{eq:4} has been proved up to $s$.

We consider $f(s+1;X,Y)$. We split the defining sum according to the
largest element of the indexing set $A$, $t$ say, so that
$A=A'\cup\{t\}$. In this manner, we obtain
\begin{multline*}
f(s+1;X,Y)=(K_1+\dots+K_{s+1}+X) \\ - \sum_{t=1}^{s}
\frac{(K_{t+1}+\dots+K_{s+1}+X)(K_t+K_{t+1}+Y\cdot\chi(t=s))}
{(K_1+\dots+K_t)(K_{t+1}+\dots+K_{s+1})} \kern4cm \\
\cdot \sum_{A'\subseteq[t-1]}^{}(-1)^{\vert A'\vert}
(K_1+\dots+K_{a'_1}) \cdots (K_{a'_{\ell-1}+1}+\dots+K_{t})\\
\cdot \frac{\prod_{a\in A'}^{} (K_a+K_{a+1})} {\prod_{a \in A'}
^{}(K_1+\dots+K_a)(K_{a+1}+\dots+K_{s+1})}.
\end{multline*}
We may rewrite this as
\begin{multline*}
f(s+1;X,Y)=(K_1+\dots+K_{s+1}+X) \\
-\sum_{t=1}^{s-1} \frac {(K_{t+1}+\dots+K_{s+1}+X)(K_t+K_{t+1})}
{(K_1+\dots+K_t)(K_{t+1}+\dots+K_{s+1})} \kern4cm \\
\kern2cm \cdot f(t;-K_{t+1}-\dots-K_{s+1},-K_{t+1}-\dots-K_{s+1})
\big\vert_{K_t\to K_t+\dots+K_{s+1}} \\
- \frac {(K_{s+1}+X)(K_s+K_{s+1}+Y)}{(K_1+\dots+K_s)K_{s+1}}
f(s;-K_{s+1},-K_{s+1}) \big\vert_{K_s\to K_s+K_{s+1}}.
\end{multline*}
Now we may use the induction hypothesis, so that we arrive at
\begin{multline*}
f(s+1;X,Y)=(K_1+\dots+K_{s+1}+X) \\
- \sum_{t=1}^{s-1} \frac {(K_{t+1}+\dots+K_{s+1}+X)(K_t+K_{t+1})}
{(K_1 + \dots + K_t)(K_{t+1}+\dots+K_{s+1})} \kern4cm \\
\kern2cm \cdot \frac{(K_{t+1}+\dots+K_{s+1})(K_1+\dots+K_{t-1}+K_t)}
{K_2+\dots+K_{s+1}} \\
- \frac{(K_{s+1}+X)(K_s+K_{s+1}+Y)} {(K_1+\dots+K_s)K_{s+1}} \frac
{K_{s+1}(K_1+\dots+K_{s-1}+K_{s})} {K_2+\dots+K_{s+1}} \\
= (K_1+\dots+K_{s+1}+X) - \sum_{t=1}^{s-1} \frac
{(K_{t+1}+\dots+K_{s+1}+X)(K_t+K_{t+1})} {K_2+\dots+K_{s+1}}\\ -
\frac {(K_{s+1}+X)(K_s+K_{s+1}+Y)} {K_2+\dots+K_{s+1}}.
\end{multline*}
By a routine calculation this simplifies to
$$
f(s+1;X,Y) = - \frac{X(K_1+ \dots + K_{s}) + Y(K_{s+1} + X)}{K_2+
\dots + K_{s+1}}.
$$
This is exactly the right-hand side of \eqref{eq:4} with $s$
replaced by $s+1$. \end{proof}

This completes the proof of the Theorem \ref{duality}. \end{proof}

We finish the present section by making Theorem \ref{duality}
explicit in the four lowest-order special cases. The following
formulas are written in a way which makes the self-adjointness of
the sums obvious.

\begin{ex}\label{P4-form} The definition $\M_4 = P_4 - P_2^2$ is equivalent
to $P_4 = \M_4 + \M_2^2$.
\end{ex}

\begin{ex}\label{P6-form} $P_6 = \M_6 + 2 (\M_2 \M_4 + \M_4 \M_2) + \M_2^3$.
\end{ex}

\begin{ex}\label{P8-form} $P_8$ can be written in the form of the sum
\begin{equation*}
\M_8 + 3 (\M_2 \M_6 + \M_6 \M_2) + 9 \M_4^2  + 3 (\M_2^2 \M_4 + \M_4
\M_2^2) + 4 \M_2 \M_4 \M_2  + \M_2^4
\end{equation*}
of $8$ terms.
\end{ex}

\begin{ex}\label{P10-form} $P_{10}$ can be written in the form of the sum
\begin{multline*}
\M_{10} + 4 (\M_2 \M_8 + \M_8 \M_2) + 24 (\M_4 \M_6 + \M_6 \M_4) \\
+ 6 (\M_2^2 \M_6 + \M_6 \M_2^2) + 24 (\M_2 \M_4^2 + \M_4^2 \M_2) + 9
\M_2 \M_6 \M_2 + 16 \M_4 \M_2 \M_4 \\ + 4 (\M_2^3 \M_4 + \M_4
\M_2^3) + 6 (\M_2^2 \M_4 \M_2 + \M_2 \M_4 \M_2^2) + \M_2^5
\end{multline*}
of $16$ terms.
\end{ex}

These formulas invert the definitions of the low-order special cases
$\M_{2N}$ ($N \le 5$) displayed in Appendix \ref{averages}.

%%%%%%%%%%%%%%%%%%%%%%%%%%%%%%%%%%%%%%%%%%%%%%%%%%%%%%%%%%%%%%%%%%%%%%%
\section{Residue families}\label{RF}

In the present section, we recall from \cite{juhl-book} the notion
of residue families
$$
D_{2N}^{res}(g;\lambda): C^\infty(M \times [0,\varepsilon)) \to
C^\infty(M)
$$
(see also \cite{BJ} for an introduction) and describe their main
properties. In particular, we describe their recursive structure
which finds its natural expression in the form of two systems of
factorization identities.

The definition of $D_{2N}^{res}(g;\lambda)$ involves two
ingredients:
\begin{itemize}
\item the renormalized volume coefficients $v_2(g),\cdots,v_{2N}(g)$
and
\item the formal asymptotic expansions of eigenfunctions of the Laplacian
of a Poincar\'e-Einstein metric $g_+$ relative to $g$.
\end{itemize}
We first recall the notion of Poincar\'e-Einstein metrics in the
sense of Fefferman and Graham \cite{cartan}, \cite{FG-final}. Let
$M$ be a manifold of dimension $n$ with a given metric $g$. On the
space $M \times (0,\varepsilon)$ we consider metrics of the form
\begin{equation}\label{PE-metric}
g_+ = r^{-2} (dr^2 + g_r),
\end{equation}
where $g_r$ is a one-parameter family of metrics on $M$ so that $g_0
= g$. Moreover, we require that for {\em odd} $n$ the tensor
$$
\Ric(g_+) + n g_+
$$
vanishes to infinite order along $M$, and that for {\em even} $n \ge
4$
\begin{equation}\label{einstein}
\Ric(g_+) + n g_+ = O(r^{n-2})
\end{equation}
together with a vanishing trace condition for $\Ric(g_+) + ng_+$ to
the order $r^{n-2}$. If $g_r$ is required to be even (in $r$) and
$n$ is odd, these conditions uniquely determine the family $g_r$ for
a general metric $g$. Similarly, for even $n$, the conditions
uniquely determine the coefficients $g_{(2)}, \dots, g_{(n-2)}$,
$\tilde{g}_{(n)}$ and the trace of $g_{(n)}$ in the even power
series
$$
g_r = g + r^2 g_{(2)} + \cdots + r^{n-2} g_{(n-2)} + r^n (g_{(n)} +
\log r \tilde{g}_{(n)}) + \cdots.
$$
Moreover, we have $\tr_g (\tilde{g}_{(n)}) = 0$. Since for even $n$
(and general metrics) the constructions in the present paper will
only depend on the terms $g_{(2)},\dots,g_{(n-2)}$ and $\tr_g
(g_{(n)})$ (which are uniquely determined by $g$), it will be
convenient to {\em define} $g_r$ in this case by the finite sum
$$
g + r^2 g_{(2)} + \cdots + r^{n-2} g_{(n-2)} + r^n g_{(n)}.
$$
For a Poincar\'e-Einstein metric $g_+$ of $g$, we consider the
quotient
$$
v(r) = \frac{vol(g_r)}{vol(g)} \in C^\infty(M)
$$
of volume forms on $M$. Then
$$
dvol (g_+) = r^{-n-1} v(r) dvol(g) dr.
$$
In particular, for even $n$, we have the even expansion
\begin{equation}\label{v-coeff}
v(r) = 1 + r^2 v_2 + r^4 v_4 + \cdots + r^n v_n + \cdots.
\end{equation}
The coefficients $v_{2j} \in C^\infty(M)$ in this expansion are
called renormalized volume (or holographic) coefficients. The
coefficients $v_{2j}$ are intriguing scalar-valued Riemannian
curvature invariants (see \cite{G-ext}, \cite{juhl-book} and the
references therein). We also define
\begin{equation}\label{w-coeff}
w(r) = \sqrt{v(r)} = 1 + r^2 w_2 + r^4 w_4 + \cdots + r^n w_n +
\cdots.
\end{equation}

Now let $i: M \hookrightarrow M \times [0,\varepsilon)$ denote the
embedding $i(m)=(m,0)$. We consider formal solutions
\begin{equation}\label{EF}
u (\cdot,r) \sim \sum_{j \ge 0} r^{\lambda+2j}
\T_{2j}(g;\lambda)(f)(\cdot), \; \T_0 (f) = f \in C^\infty(M)
\end{equation}
of the eigen-equation
$$
-\Delta_{g_+} u = \lambda(n\!-\!\lambda) u, \; \lambda \in \c.
$$
The coefficients in the expansion \eqref{EF} are given by rational
families $\T_{2j}(g;\lambda)$ (in $\lambda$) of differential
operators acting on the ``boundary value'' $f$ of $u$. For example,
the first two families are
$$
\T_2(\lambda) = \frac{1}{2(n\!-\!2\!-\!2\lambda)} (\Delta \!-\!
\lambda \J)
$$
and
\begin{multline*}
\T_4(\lambda) =
\frac{1}{8(n\!-\!2\!-\!2\lambda)(n\!-\!4\!-\!2\lambda)} \big[
(\Delta \!-\! (\lambda\!+\!2)\J) (\Delta \!-\! \lambda \J) \\ +
\lambda(2\lambda\!-\!n\!+\!2) |\Rho|^2 + 2(2\lambda\!-\!n\!+\!2)
\delta (\Rho d) + (2\lambda\!-\!n\!+\!2) (d\J,d) \big].
\end{multline*}
Now for general $g$, even $n$ and $2N \le n$, we define
\begin{equation}\label{res-fam}
D_{2N}^{res}(g;\lambda) = 2^{2N} N! \left[
\left(-\f\!-\!\lambda\!+\!2N\!-\!1\right) \cdots
\left(-\f\!-\!\lambda\!+\!N\right) \right]
\delta_{2N}(g;\lambda\!+\!n\!-\!2N)
\end{equation}
with
\begin{equation}\label{res-fam-2}
\delta_{2N}(g;\lambda) = \sum_{j=0}^N \frac{1}{(2N\!-\!2j)!} \left[
\T_{2j}^*(g;\lambda) v_0 + \cdots + \T_0^*(g;\lambda) v_{2j} \right]
i^* (\partial/\partial r)^{2N-2j}.
\end{equation}
Here the holographic coefficients act as multiplication operators,
and $\T_{2j}^*(g;\lambda)$ denotes the formal adjoint of the
differential operator $\T_{2j}(g;\lambda)$ on $C^\infty(M)$ with
respect to the metric $g$. Note that $D_0^{res}(g;\lambda) = i^*$.
The main role of the polynomial pre-factor in \eqref{res-fam} is to
remove the poles (see also \eqref{detail-left}). Similarly, we
define $D_{2N}^{res}(g;\lambda)$, $N \ge 1$ for general $g$ and odd
$n$. It is obvious that the definition of $D_{2N}^{res}(\lambda)$
only requires to solve the eigen-equation approximately up to the
order $r^{\lambda+2N}$.

We also recall that, for a locally conformally flat metric $g$, the
metric
$$
r^{-2} (dr^2 + g - r^2 \Rho + r^4/4 \Rho^2)
$$
is a Poincar\'e-Einstein metric relative to $g$ (see Chapter 7 of
\cite{FG-final} or Section 6.14 of \cite{juhl-book}). In that case,
residue families are well-defined for any $N \ge 1$.

The residue families $D_{2N}^{res}(g;\lambda)$ admit an
interpretation as obstructions to the extension of eigenfunctions of
$\Delta_{g_+}$ through the boundary $r=0$. These obstructions appear
as residues of associated meromorphic families of distributions. It
is the latter relation which motivates the name. For later use, we
need to recall the precise statement. Let $u \in C^\infty(X)$, $X =
(0,\varepsilon) \times M$ be an eigenfunction
$$
-\Delta_{g_+} u = \mu(n\!-\!\mu) u
$$
with boundary value $f \in C^\infty(M)$. Let $\varphi \in
C_0^\infty(\bar{X})$ be a test function. We consider the integral
$$
\int_X r^\lambda u \varphi d vol(r^2 g_+).
$$
It is holomorphic for $\lambda$ with sufficiently large real part
and admits a meromorphic continuation with a simple pole at
$\lambda=-\mu\!-\!1\!-\!2N$. Then
\begin{equation}\label{residue}
\Res_{\lambda = -\mu-1-2N} \left( \int_X r^\lambda u \varphi
dvol(r^2 g_+) \right) = \int_M f
\delta_{2N}(\lambda\!+\!n\!-\!2N)(\varphi) dvol(g).
\end{equation}
For full details we refer to \cite{juhl-book} and \cite{BJ}.

Now residue families have the following basic properties.
\begin{itemize}\label{properties}
\item $D_{2N}^{res}(g;\lambda)$ is a polynomial in $\lambda$ of
degree $N$.
\item $D_{2N}^{res}(g;-n/2+N) = P_{2N}(g) i^*$.
\item $D_{2N}^{res}(g;\lambda)$ satisfies a certain conformal covariance
law under $g \mapsto e^{2\varphi} g$.
\end{itemize}

The second of these properties is contained in a system of $N$
factorization identities which are satisfied by the family
$D_{2N}^{res}(g;\lambda)$. There are actually two such systems. The
{\em first} system is given by the following result.

\begin{thm}\label{Factor-A} Assume that $2N \le n$ for even $n$ and
$N \ge 1$ for odd $n$. Then the factorization relations
\begin{equation}\label{Fact-A}
D_{2N}^{res}\left(g;-\f\!+\!2N\!-\!j\right) = P_{2j}(g) \circ
D_{2N-2j}^{res}\left(g;-\f\!+\!2N\!-\!j\right)
\end{equation}
for $j=1,\dots,N$ hold true for any metric $g$.
\end{thm}

Next, we denote by
\begin{equation}\label{compact}
\bar{P}_{2N}(g) = P_{2N}(dr^2\!+\!g_r)
\end{equation}
the GJMS-operator of order $2N$ for the conformal compactification
$\bar{g} = dr^2\!+\!g_r$ of $g_+$ on $X = M \times [0,\varepsilon)$.
In these terms, the {\em second} system of factorization identities
is given by the following result.

\begin{thm}\label{Factor-B} Assume that $2N \le n$ for even $n$
and $N \ge 1$ for odd $n$. Then the factorization relations
\begin{equation}\label{Fact-B}
D_{2N}^{res}\left(g;-\frac{n\!+\!1}{2}\!+\!j\right) =
D_{2N-2j}^{res}\left(g;-\frac{n\!+\!1}{2}\!-\!j\right) \circ
\bar{P}_{2j}(g)
\end{equation}
for $j=1,\dots,N$ hold true for any metric $g$.
\end{thm}

We recall that, for general metrics on even-dimensional manifolds,
the existence of associated Poincar\'e-Einstein metrics is
obstructed. As a consequence, GJMS-operators for general metrics are
only defined for orders which do not exceed the dimension of the
underlying space. This is the reason for the condition $2N \le n$ in
both theorems. However, in odd dimensions there are no obstructions.

For odd $n$, the factorizations \eqref{Fact-B} involve the
GJMS-operators of all orders $2N \ge 2$ on the space $X = M \times
[0,\varepsilon)$ of {\em even} dimension $n+1$ for the metric
$\bar{g} = dr^2 + g_r$. Nevertheless, these operators are
well-defined. Indeed, Theorem \ref{double} yields a formula for a
Poincar\'e-Einstein metric for $\bar{g}$ to any order. Applying the
construction of \cite{GJMS} defines the desired operators
$\bar{P}_{2N}$ for all $N \ge 1$. These are conjugate to the
GJMS-operators for the Einstein metric $r^{-2} \bar{g}$.

Note also that for locally conformally flat metrics, Theorems
\ref{Factor-A}--\eqref{Factor-B} extend to all $N \ge 1$.

The systems \eqref{Fact-A} and \eqref{Fact-B} contain the relations
$$
D_{2N}^{res}\left(g,-\f\!+\!N\right) = P_{2N}(g) i^* \quad
\mbox{and} \quad D_{2N}^{res}\left(g;-\frac{n\!+\!1}{2}\!+\!N\right)
= i^* \bar{P}_{2N}(g)
$$
as the respective special cases $j=N$.

In \cite{juhl-book}, we established the factorizations in the system
\eqref{Fact-A} for all metrics (if $2N \le n$ for even $n$) and the
factorizations in the system \eqref{Fact-B} for locally conformally
flat metrics. In \cite{juhl-book}, we also confirmed the
factorizations in the system \eqref{Fact-B} for $N=1,2$ and the
factorization $j=1$ in the system \eqref{Fact-B} for $N=3$ by direct
calculations. For locally conformally flat metrics, {\em both}
systems of identities actually follow from their versions in the
flat case by the conformal covariance of the families (see Theorem
6.6.3 in \cite{juhl-book} and Theorem 1.5.3 in \cite{BJ}).
Alternatively, the system \eqref{Fact-A} is a consequence of the
identification of GJMS-operators in the asymptotic expansions of
eigenfunctions of the Laplacian of Poincar\'e-Einstein metrics
\cite{GZ}. The latter argument extends to general metrics. We
continue with the description of the details of that proof.

\begin{proof}[{\bf Proof of Theorem \ref{Factor-A}}] We use the fact
that
\begin{equation}\label{solution}
\T_{2k}(\lambda) = \frac{1}{2^{2k} k! (\f\!-\!\lambda\!-\!1) \cdots
(\f\!-\!\lambda\!-\!k)} P_{2k}(\lambda)
\end{equation}
with a holomorphic family $P_{2k}(\lambda)$ of the form $\Delta^k +
\cdots$ (see \cite{GZ}, \cite{juhl-book}, \cite{BJ}). Hence we can
write the family $D_{2N}^{res}(\lambda)$ in the form
\begin{multline}\label{detail-left}
\sum_{l=0}^{N-k} \sum_{k=0}^N \frac{2^{2N}N!}{2^{2k}k!}
\frac{1}{(2N\!-\!2k\!-\!2l)!} \\
\times \left(-\f\!-\!\lambda\!+\!2N\!-\!k\!-\!1\right)
\cdots\left(-\f\!-\!\lambda\!+\!N\right)
P_{2k}^*(\lambda\!+\!n\!-\!2N) v_{2l} i^*
\left(\frac{\partial^2}{\partial r^2}\right)^{N-k-l}.
\end{multline}
Since the families $P_{2k}(\lambda)$ are polynomials of degree $k$
in $\lambda$ (which easily can be proved by induction),
\eqref{detail-left} shows that $D_{2N}^{res}(\lambda)$ is a
polynomial of degree $N$ in $\lambda$. Similarly, for the family
$D_{2N-2j}^{res}(\lambda)$, we find the formula
\begin{multline}\label{detail-right}
\sum_{l=0}^{N-j-k} \sum_{k=0}^{N-j}
\frac{2^{2N-2j}(N\!-\!j)!}{2^{2k}k!}
\frac{1}{(2N\!-\!2j\!-\!2k\!-\!2l)!}
\\ \times \left(-\f\!-\!\lambda\!+\!2N\!-\!2j\!-\!k\!-\!1\right)
\cdots \left(-\f\!-\!\lambda\!+\!N\!-\!j\right)
\\ \times P_{2k}^*(\lambda\!+\!n\!-\!2(N\!-\!j)) v_{2l} i^*
\left(\frac{\partial^2}{\partial r^2}\right)^{N-j-k-l}.
\end{multline}
But for the value $\lambda=-\f\!+\!2N\!-\!j$, the non-trivial
contributions in \eqref{detail-left} satisfy $N \ge k \ge j$. Hence
by an index shift the sum simplifies to
\begin{multline}\label{d-r-simple}
\sum_{l=0}^{N-k-j} \sum_{k=0}^{N-j}
\frac{2^{2N}N!}{2^{2k+2j}(k\!+\!j)!}
\frac{1}{(2N\!-\!2k\!-\!2j\!-\!2l)!} \frac{(N\!-\!j)!}{k!} (-1)^{N-k-j} \\
\times P_{2k+2j}^*\left(\f\!-\!j\right) v_{2l} i^*
\left(\frac{\partial^2}{\partial r^2}\right)^{N-k-j-l}.
\end{multline}
On the other hand, for $\lambda=-\f\!+\!2N\!-\!j$, the sum
\eqref{detail-right} reduces to
\begin{multline}\label{d-l-simple}
\sum_{l=0}^{N-j-k} \sum_{k=0}^{N-j}
\frac{2^{2N-2j}(N\!-\!j)!}{2^{2k}k!}
\frac{1}{(2N\!-\!2j\!-\!2k\!-\!2l)!} \frac{N!}{(k\!+\!j)!}
(-1)^{N-k-j} \\ \times P_{2k}^*\left(\f\!+\!j\right) v_{2l} i^*
\left(\frac{\partial^2}{\partial r^2}\right)^{N-j-k-l}.
\end{multline}
Now we use the fact that the families $P_{2N}(\lambda)$ obey the
factorization relations
\begin{equation}\label{solution-factor}
P_{2N}\left(\f\!-\!k\right) = P_{2N-2k}\left(\f\!+\!k\right) P_{2k}
\end{equation}
for $N \ge k$ (\cite{juhl-book}, Theorem 6.11.18). For $k=N$, these
relations state that
\begin{equation}\label{spectral-GJMS}
P_{2N}\left(\f\!-\!N\right) = P_{2N}.
\end{equation}
Eq.~\eqref{spectral-GJMS} is the fundamental connection between
GJMS-operators and formal asymptotic expansions of eigenfunctions of
$\Delta_{g_+}$ (see \cite{GZ}). In order to prove
\eqref{solution-factor}, we recall the structure of the asymptotic
expansions of generalized eigenfunctions of $-\Delta_{g_+}$ for
generic eigenvalues. These have the form
$$
\sum_{j \ge 0} \T_{2j}(n\!-\!\lambda)(f) r^{n-\lambda+2j} + \sum_{j
\ge 0} \T_{2j}(\lambda) \mathcal{S}(\lambda)(f) r^{\lambda+2j},
$$
where $\mathcal{S}(\lambda)$ is the scattering operator; here we use
the conventions of \cite{juhl-book}. Now the contributions
$$
\T_{2N}(n\!-\!\lambda)(f) r^{n-\lambda+2N} \quad \mbox{and} \quad
\T_{2N-2k}(\lambda) \mathcal{S}(\lambda)(f) r^{\lambda+2N-2k}
$$
both have a simple pole at $\lambda = \f+k$; note that
$\T_{2N-2k}(\lambda)$ is regular at $\lambda=\f+k$. The cancellation
of poles in the sum implies the relation
\begin{equation}
\Res_{\f-k}(\T_{2N}(\lambda)) r^{\f+2N-k} +
\T_{2N-2k}\left(\f\!+\!k\right) \Res_{\f+k}(\mathcal{S}(\lambda))
r^{\f+2N-k} =0.
\end{equation}
But since $\Res_{\f+k}(\mathcal{S}(\lambda))$ is proportional to
$P_{2k}$ (\cite{GZ}), we have proved that both sides of
\eqref{solution-factor} are proportional. Since both sides are of
the form $\Delta^N + \cdots$, this proves the equality. In
particular, by taking adjoints, \eqref{solution-factor} shows that
\begin{equation}\label{factor-sol}
P_{2k+2j}^*\left(\f\!-\!j\right) = P_{2j}^*
P_{2k}^*\left(\f\!+\!j\right)
\end{equation}
for $0 \le k \le N\!-\!j$. Therefore, the self-adjointness $P_{2j}^*
= P_{2j}$ implies that the sum \eqref{d-r-simple} coincides with the
product of $P_{2j}$ and \eqref{d-l-simple}. The proof is complete.
\end{proof}

We continue with the

\begin{proof}[{\bf Proof of Theorem \ref{Factor-B}}] The proof rests
on the interpretation of residue families as residues as in
\eqref{residue}. We set
\begin{equation}\label{param}
\mu = \frac{n\!-\!1}{2} \!+\! M \!-\! 2N \quad \mbox{with} \quad
M=1,\dots,N.
\end{equation}
We choose an arbitrary $f \in C^\infty(M)$. We consider formal
approximate eigenfunctions of the Laplacian $-\Delta_{g_+}$ for the
eigenvalue $\mu(n-\mu)$ and with boundary value $f$. More precisely,
for odd $n$ and $N \ge 1$ as well as for even $n$ and $N \le \f$, we
let $u$ be the sum $\sum_{j=0}^N r^{\mu+2j} a_{2j}(\mu)$ with $a_0 =
f$ so that
$$
-\Delta_{g_+} u - \mu(n\!-\!\mu) u = O(r^{\mu+2N+2}).
$$
Then the coefficients $a_{2j}(\mu)$ ($j=1,\dots,N$) are uniquely
determined and are given by the differential operators
$\T_{2j}(\mu)$ acting on $f$. Eq.~\eqref{solution} shows that the
regularity of the families $\T_{2j}(\lambda)$ at $\lambda = \mu$ is
guaranteed by \eqref{param}. In order to simplify the following
arguments, we shall suppress the minor modifications caused by the
fact that $u$ is only an approximate eigenfunction. Now
\begin{equation}\label{scalar}
P_{2M}(g_+) (u) = \left( \prod_{j=0}^{2M-1} (2N\!-\!j) \right) u \st
\kappa u.
\end{equation}
In fact, since $g_+$ is Einstein with scalar curvature $-n(n+1)$,
the product formula (see \cite{G-ein} and \cite{BJ}, Theorem 1.3.7)
$$
P_{2M}(g_+) = \prod_{j=\frac{n+1}{2}}^{\frac{n+1}{2}+M-1}
(\Delta_{g_+} \!+\! j(n\!-\!j))
$$
implies that $P_{2M}(g_+)$ acts on $u$ by the scalar
\begin{align*}
& \prod_{j=\frac{n+1}{2}}^{\frac{n+1}{2}+M-1} (-\mu(n\!-\!\mu) + j(n\!-\!j)) \\
& = \prod_{k=0}^{M-1}
\left(-\left(\fm\!+\!M\!-\!2N\right)\left(\fp\!-\!M\!+\!2N\right)
+ \left(\fp\!+\!k\right)\left(\fm\!-\!k\right) \right) \\
& = \prod_{k=0}^{M-1} (M\!-\!2N\!+\!k)(M\!-\!2N\!-\!1\!-\!k) \\
& = \prod_{j=0}^{2M-1} (2N\!-\!j).
\end{align*}
This proves \eqref{scalar}. Note that $\kappa \ne 0$. Next, the
definition \eqref{res-fam} gives
\begin{multline}\label{D-one}
D_{2N}^{res}\left(g;-\fp\!+\!M\right) \\[-2mm] = 2^{2N} N! \left[\left(
-\frac{1}{2}\!+\!2N\!-\!M\right) \cdots \left(
\frac{1}{2}\!+\!N\!-\!M \right) \right] \delta_{2N}(g;\mu)
\end{multline}
and
\begin{multline}\label{D-two}
D_{2N-2M}^{res}\left(g;-\fp\!-\!M\right) \\[-2mm] = 2^{2N-2M} (N\!-\!M)!
\left[ \left(-\frac{1}{2}\!+\!2N\!-\!M\right) \cdots
\left(\frac{1}{2}\!+\!N \right) \right] \delta_{2N-2M}(g;\mu).
\end{multline}
We observe that the quotient of the overall factors on the
right-hand sides of \eqref{D-one} and \eqref{D-two} equals
\begin{equation}\label{kappa-rel}
2^{2M} \frac{N!}{(N\!-\!M)!} \left(-\frac{1}{2} \!+\! N \right)
\cdots \left(\frac{1}{2} \!+\! N \!-\! M \right) = \kappa
\end{equation}
with $\kappa$ as defined in \eqref{scalar}. Now \eqref{scalar}
implies
$$
r^\lambda u = r^{\lambda+\frac{n+1}{2}+M} (r^{-\frac{n+1}{2}-M} u) =
r^{\lambda+\frac{n+1}{2}+M} \kappa^{-1} \left(r^{-\frac{n+1}{2}-M}
P_{2M}(g_+) (u) \right).
$$
Therefore, by the conformal covariance of $P_{2M}$, we have
\begin{equation}\label{subst}
r^\lambda u = r^{\lambda+\frac{n+1}{2}+M} \kappa^{-1}
\bar{P}_{2M}(g) (r^{-\frac{n+1}{2}+M} u).
\end{equation}
Now let $\varphi \in C_0^\infty(X)$. We apply the identity
\eqref{subst} to rewrite the right-hand side of
$$
\int_M f \delta_{2N}(g;\mu) (\varphi) dvol (g) =
\Res_{\lambda=-\mu-1-2N} \left( \int_X r^\lambda u \varphi \,
dvol(\bar{g}) \right)
$$
(see Eq.~\eqref{residue}) in the form
\begin{equation}\label{trick}
\kappa^{-1} \Res_{\lambda=-\mu-1-2N} \left( \int_X \bar{P}_{2M}(g)
(r^{-\frac{n+1}{2}+M} u) \, r^{\lambda+\frac{n+1}{2}+M} \varphi \,
dvol(\bar{g}) \right).
\end{equation}
Now we use that $\bar{P}_{2M}(g) = P_{2M}(\bar{g})$ is self-adjoint
with respect to the volume form of $\bar{g}$ and apply partial
integration. Since boundary terms are holomorphic (in $\lambda$),
they do not contribute to the residue. Hence \eqref{trick} equals
$$
\kappa^{-1} \Res_{\lambda=-\mu-1-2N} \left( \int_X
r^{-\frac{n+1}{2}+M} u  \bar{P}_{2M}(g)
\left(r^{\lambda+\frac{n+1}{2}+M} \varphi \right) dvol(\bar{g})
\right).
$$
But the function
$$
\lambda \mapsto \bar{P}_{2M}(g) \left(r^{\lambda+\frac{n+1}{2}+M}
\varphi \right) - r^{\lambda+\frac{n+1}{2}+M}
\bar{P}_{2M}(g)(\varphi)
$$
vanishes at $\lambda = -\mu\!-\!1\!-\!2N = -\frac{n+1}{2}\!-\!M$. It
follows that the latter residue coincides with
$$
\kappa^{-1} \Res_{\lambda=-\mu-1-2N} \left( \int_X r^{\lambda+2M} u
\bar{P}_{2M}(g) (\varphi) dvol(\bar{g}) \right).
$$
Thus, we conclude that
$$
\int_M f \delta_{2N}(g;\mu) (\varphi) \, dvol(g)
$$
equals
$$
\kappa^{-1} \Res_{\lambda=-\mu-1-(2N-2M)} \left( \int_X r^\lambda u
\bar{P}_{2M}(g) (\varphi) dvol(\bar{g}) \right).
$$
By \eqref{residue}, the latter expression can be written in the form
$$
\kappa^{-1} \int_M f \delta_{2N-2M}(g;\mu)(\bar{P}_{2M}(g)
(\varphi)) dvol(g).
$$
But since $f \in C^\infty(M)$ is arbitrary, we have proved that
$$
\kappa \, \delta_{2N}(g;\mu) = \delta_{2N-2M}(g;\mu)
\bar{P}_{2M}(g).
$$
By \eqref{kappa-rel}, this identity is equivalent to the assertion.
\end{proof}

We finish this section with two remarks. The proofs of the
factorization identities utilize the condition that $g_+$ is
Einstein in {\em two} ways. In fact, the Einstein condition enters
into the proof of Theorem \ref{Factor-A} through the connection
between asymptotic expansions of eigenfunctions of the Laplacian on
$X = M \times (0,\varepsilon)$ and GJMS-operators on $M$. On the
other hand, the proof of Theorem \ref{Factor-B} rests on the fact
that all GJMS-operators for Einstein metrics on $X$ act by scalars
on eigenfunctions of the Laplacian (which is a consequence of the
product formula for these operators). Finally, the residue families
$D_{2N}^{res}(\lambda)$ should be regarded as curved analogs of
(differential) intertwining operators. From that perspective, the
factorization identities appear as curved analogs of the fact that,
in multiplicity-free decompositions, spaces of intertwining
operators are one-dimensional.

%%%%%%%%%%%%%%%%%%%%%%%%%%%%%%%%%%%%%%%%%%%%%%%%%%%%%%%%%%%%%%%%%%%%%%%%%%%
\section{Residue families in terms of GJMS-operators}\label{cc}

In the present section, we establish a fundamental representation
formula for residue families in terms of GJMS-operators. Let
\begin{equation}\label{poly}
\pi_{2N}(x) \st x(x\!-\!1) \cdots (x\!-\!N\!+\!1)
\left(x\!-\!\frac{1}{2}\!+\!N\right) \cdots
\left(x\!+\!\frac{1}{2}\right).
\end{equation}
Then
$$
\pi_{2N}^{-1}(0) = \{0,1,\dots,N\!-\!1\} \cup
\left\{-\frac{1}{2},\dots,-N\!+\!\frac{1}{2} \right\}.
$$
We also set $m_{\bar{I}} = - m_I$. In these terms, the main result
can be stated as follows.

\begin{thm}\label{D-rep} The residue family $D_{2N}^{res}(\lambda)$
coincides with the sum of
\begin{equation}\label{P}
-(-1)^N  \frac{2^{2N-1}}{(2N\!-\!1)!} \sum_{|I|=N}
\frac{\pi_{2N}(\lambda\!+\!\f\!-\!N)}{(\lambda\!+\!\f\!-\!2N\!+\!I_l)}
m_I P_{2I},
\end{equation}
\begin{equation}\label{P-bar}
-(-1)^N  \frac{2^{2N-1}}{(2N\!-\!1)!} \sum_{|J|=N}
\frac{\pi_{2N}(\lambda\!+\!\f\!-\!N)}
{(\lambda\!+\!\frac{n+1}{2}\!-\!J_r)} m_{\bar{J}} \bar{P}_{2J}
\end{equation}
and
\begin{multline}\label{mixed}
(-1)^N  \frac{2^{2N-1}}{(2N\!-\!1)!} \\ \times \sum_{|I|+|J|=N}
\frac{\pi_{2N}(\lambda\!+\!\f\!-\!N)}
{(\lambda\!+\!\f\!-\!2N\!+\!I_l)(\lambda\!+\!\frac{n+1}{2}\!-\!J_r)}
\frac{N!(N\!-\!1)!}{|I|!(|I|\!-\!1)! |J|!(|J|\!-\!1)!} m_I
m_{\bar{J}} P_{2I} \bar{P}_{2J}.
\end{multline}
\end{thm}

Some comments are in order.

We have slightly simplified the formulation of the representation
formula by omitting all compositions with the pull-back $i^*$. In
particular, $P_{2I} \bar{P}_{2J}$ means $P_{2I} i^* \bar{P}_{2J}$.
We shall use this convention throughout from now on.

$I_l$ and $J_r$ denote the respective most left and most right
entries of the compositions $I$ and $J$. Note also that the formulas
are valid for any $\lambda \in \c$ since all fractions actually are
polynomials in $\lambda$. In fact,
$$
\lambda\!+\!\f\!-\!2N\!+\!I_l = \left(\lambda\!+\!\f\!-\!N\right) -
(N\!-\!I_l)
$$
with $(N\!-\!I_l) \in \{0,1,\dots,N\!-\!1\} \subset
\pi_{2N}^{-1}(0)$ and
$$
\lambda\!+\!\frac{n\!+\!1}{2}\!-\!J_r =
\left(\lambda\!+\!\f\!-\!N\right) +
\left(N\!+\!\frac{1}{2}\!-\!J_r\right)
$$
with $-(N\!+\!\frac{1}{2}\!-\!J_r) \in
\{-\frac{1}{2},\dots,-N\!+\!\frac{1}{2}\} \subset \pi_{2N}^{-1}(0)$.
The sums \eqref{P} and \eqref{P-bar} can be viewed as degenerate
special cases of \eqref{mixed}. The formula in Theorem \ref{D-rep}
also reflects a certain {\em symmetry} of the families
$D_{2N}^{res}(\lambda)$. In fact, the polynomials $\pi_{2N}$ satisfy
the symmetry relations
\begin{equation*}\label{symm-p}
\pi_{2N}\left(\lambda\right) =
\pi_{2N}\left(-\lambda\!-\!\frac{1}{2}\right).
\end{equation*}
Now using Theorem \ref{D-rep} these symmetry relations imply that
\begin{equation}\label{symm-D}
\sigma D_{2N}^{res}\left(\lambda\!-\!\f\!+\!N\right) =
D_{2N}^{res}\left(-\lambda\!-\!\frac{n\!+\!1}{2}\!+\!N\right),
\end{equation}
where $\sigma$ maps $P_{2a} \bar{P}_{2b}$ into $P_{2b}
\bar{P}_{2a}$. But the relation \eqref{symm-D} directly follows from
the factorization relations for residue families using an induction
on $N$.

The proof of Theorem \ref{D-rep} will be given in Section
\ref{factor-D}. The idea of the proof is the following. For any $N
\ge 1$, we define a family $D_{2N}(\lambda)$ by their Taylor
expansions in terms of GJMS-operators. We prove that these families
coincide with the sum of \eqref{P}, \eqref{P-bar} and \eqref{mixed},
{\em and} satisfy the {\em same} systems of factorization identities
as the residue families $D_{2N}^{res}(\lambda)$ (see Section
\ref{RF}).

We continue with the definition of the families
\begin{equation}\label{def-d}
D_{2N}(\lambda) \st \sum_{k=1}^{2N} d_{2N}^{(k)}
\left(\lambda\!+\!\f\!-\!N\right)^{2N-k}.
\end{equation}
The coefficients $d_{2N}^{(k)}$ are certain linear combinations of
compositions of GJMS-operators $P_{2M}$ and $\bar{P}_{2M}$ for $1
\le M \le N$. In order to define these coefficients, we distinguish
between
\begin{itemize}
\item compositions of GJMS-operators for $g$,
\item compositions of GJMS-operators for $\bar{g}$ and
\item mixed compositions of GJMS-operators for $g$ and $\bar{g}$.
\end{itemize}
The respective multiplicities of these compositions will be defined
and studied in the following three subsections.

%%%%%%%%%%%%%%%%%%%%%%%%%%%%%%%%%%%%%%%%%%%%%%%%%%%%%%%%%%%%%%%%%%%%%%%%%%%
\subsection{The pure $P$-terms}\label{pure-P}

Let
$$
m^{(k)}_{(a,I)}, \; a+|I|=N, a \ge 1
$$
be the coefficient of $P_{2a} P_{2I}$ in the coefficient
$d_{2N}^{(k)}$. Here $I$ can be trivial, i.e., $a=N$. Then
\begin{equation}
m_{(a,I)}^{(k)} \st \left(\sum_{j=0}^{k-1} \eta_{2N}^{(2N-j-1)}
|I|^{k-1-j} \right) m^{(1)}_{(a,I)},
\end{equation}
where
$$
\eta_{2N}(x) \st \frac{\pi_{2N}(x)}{x} = \sum_{k=0}^{2N-1}
\eta_{2N}^{(k)} x^k.
$$
In particular, we find   %corrected misprint
$$
m_{(a,I)}^{(2N)} = \left( \sum_{j=0}^{2N-1} \eta_{2N}^{(2N-j-1)}
|I|^{2N-1-j} \right) m_{(a,I)}^{(1)} = \eta_{2N}(|I|)
m_{(a,I)}^{(1)}.
$$
Hence
\begin{equation}\label{ct}
m_{(a,I)}^{(2N)} = \begin{cases} 0 & \mbox{for } 1 \le |I| \le
N\!-\!1 \\ \eta_{2N}(0) m_{(N)}^{(1)} & \mbox{for } |I|=0.
\end{cases}
\end{equation}
A calculation shows that
\begin{equation}\label{eta-zero}
\eta_{2N}(0) = -(-1)^N 2^{-(2N-1)} (2N\!-\!1)! = \pi_{2N}'(0).
\end{equation}
Finally, we have
\begin{equation}\label{higher-m}
m^{(1)}_I \st -(-1)^N \frac{2^{2N-1}}{(2N\!-\!1)!} m_I
\end{equation}
for all compositions $I$ of size $N$.

The relations \eqref{ct}, \eqref{eta-zero} and \eqref{higher-m} show
that the constant term of $D_{2N}(\lambda)$ equals $P_{2N}$. Hence
the multiplicity $[P_{2N}:D_{2N}(\lambda)]$ of the total
contribution of $P_{2N}$ to $D_{2N}(\lambda)$ is given by
\begin{equation}\label{single}
\left(\sum_{k=1}^{2N} \left(\lambda\!+\!\f\!-\!N\right)^{2N-k}
\eta_{2N}^{(2N-k)} \right) m_{(N)}^{(1)} = -(-1)^N
\frac{2^{2N-1}}{(2N\!-\!1)!}
\eta_{2N}\left(\lambda\!+\!\f\!-\!N\right).
\end{equation}

More generally, for non-trivial $I$, the multiplicity $[P_{2a}
P_{2I}:D_{2N}(\lambda)]$ of the total contribution of $P_{2a}
P_{2I}$ to $D_{2N}(\lambda)$ is given by
$$
\left( \sum_{k=1}^{2N-1} \left(\lambda\!+\!\f\!-\!N\right)^{2N-k}
\left( \sum_{i=0}^{k-1} \eta_{2N}^{(2N-i-1)} |I|^{k-1-i} \right)
\right) m_{(a,I)}^{(1)}.
$$
The latter double sum can be written in the form
$$
\left(\sum_{1 \le i+k \le 2N-1 \atop 0 \le i \le 2N-2, \, 1 \le k
\le 2N-1} \eta_{2N}^{(i+k)} |I|^i
\left(\lambda\!+\!\f\!-\!N\right)^k \right) m_{(a,I)}^{(1)}.
$$
In order to determine that sum, we apply the following result.

\begin{lemm}\label{sum-gen} For $x \ne y$,
\begin{equation}\label{sum-generic}
\sum_{1 \le a + b \le 2N-1 \atop 1 \le a \le 2N-1, \; 0 \le b \le
2N-2} \eta_{2N}^{(a+b)} x^a y^b = x \left(\frac{\eta_{2N}(x) -
\eta_{2N}(y)}{x-y}\right).
\end{equation}
Moreover,
\begin{equation}\label{sum-diag}
\sum_{1 \le a + b \le 2N-1 \atop 1 \le a \le 2N-1, \; 0 \le b \le
2N-2} \eta_{2N}^{(a+b)} M^{a+b-1} = \eta_{2N}'(M)
\end{equation}
for $M=1,\dots,N-1$.
\end{lemm}

\begin{proof} Let $x \ne y$. The left-hand side of
\eqref{sum-generic} equals
\begin{align*}
\sum_{k=1}^{2N-1} \eta_{2N}^{(k)} \left( \sum_{a=1}^k x^a y^{k-a}
\right) & = \sum_{k=1}^{2N-1} \eta_{2N}^{(k)} \left( \sum_{a=0}^k
x^a y^{k-a} - y^k \right) \\
& = \sum_{k=1}^{2N-1} \eta_{2N}^{(k)} \left(
\frac{x^{k+1}\!-\!y^{k+1}}{x\!-\!y} - y^k \right).
\end{align*}
Using the definition of $\eta_{2N}$, the latter sum simplifies to
$$
\frac{1}{x-y} \left( x (\eta_{2N}(x) - \eta_{2N}^{(0)}) - y (\eta_{2
N}(y) - \eta_{2N}^{(0)})\right) - (\eta_{2N}(y) - \eta_{2N}^{(0)}) .
$$
Now the first assertion follows by simplification. For
$y=1,\dots,N-1$, the second assertion follows by taking the limit $x
\to y$.
\end{proof}

Lemma \ref{sum-gen} and $\eta_{2N}(|I|) = 0$ for $1 \le |I| \le N-1$
imply that the multiplicity $[P_{2a}P_{2I}:D_{2N}(\lambda)]$ of the
total contribution of $P_{2a} P_{2I}$ to $D_{2N}(\lambda)$ is given
by
\begin{equation}\label{c-pure-generic} -(-1)^N
\frac{2^{2N-1}}{(2N\!-\!1)!}
\frac{\pi_{2N}(\lambda\!+\!\f\!-\!N)}{(\lambda\!+\!\f\!-\!N\!-\!|I|)}
m_{(a,I)},
\end{equation}
if $\lambda\!+\!\f\!-\!N \ne |I|$, and
\begin{equation}\label{c-pure-deg}
-(-1)^N  \frac{2^{2N-1}}{(2N\!-\!1)!} \pi_{2N}'(|I|) m_{(a,I)}
\end{equation}
if $\lambda\!+\!\f\!-\!N = |I|$. Note that for trivial $I$
\eqref{c-pure-generic} specializes to \eqref{single}. These results
yield the sum in \eqref{P}.

%%%%%%%%%%%%%%%%%%%%%%%%%%%%%%%%%%%%%%%%%%%%%%%%%%%%%%%%%%%%%%%%%%%%%%%%%%%%
\subsection{The pure $\bar{P}$-terms}\label{pure-P-bar}

Let
$$
m^{(k)}_{(\bar{I},\bar{b})}, \; |I|+b=N
$$
be the coefficient of $\bar{P}_{2I} \bar{P}_{2b}$ in the coefficient
$d_{2N}^{(k)}$. Here $I$ can be trivial, i.e., $b=N$. Then
\begin{multline}
m_{(\bar{I},\bar{b})}^{(k)} \\
\st \left( \sum_{j=0}^{k-1} \tau_{2N}^{(2N-j-1)}
\left(-|I|\!-\!\frac{1}{2}\right)^{k-1-j} + \frac{1}{2}
\sum_{j=0}^{k-2} \tau_{2N}^{(2N-j-1)}
\left(-|I|\!-\!\frac{1}{2}\right)^{k-2-j} \right)
m^{(1)}_{(\bar{I},\bar{b})},
\end{multline}
where
\begin{equation}
\tau_{2N}(x) \st \frac{\pi_{2N}(x)}{x\!+\!\frac{1}{2}} =
\sum_{k=0}^{2N-1} \tau_{2N}^{(k)} x^k.
\end{equation}
Finally, we set $m^{(1)}_{\bar{I}} = - m^{(1)}_I$. In particular, we
find
\begin{multline*}
m_{(\bar{I},\bar{b})}^{(2N)}/m^{(1)}_{(\bar{I},\bar{b})} \\
= \left( \sum_{j=0}^{2N-1} \tau_{2N}^{2N-j-1}
\left(-|I|\!-\!\frac{1}{2}\right)^{2N-1-j} + \frac{1}{2} \left(
\sum_{j=0}^{2N-2} \tau_{2N}^{2N-j-1}
\left(-|I|\!-\!\frac{1}{2}\right)^{2N-2-j} \right) \right).
\end{multline*}
Now the latter sum simplifies to
$$
- \frac{\pi_{2N}(-|I|\!-\!\frac{1}{2})}{|I|+\frac{1}{2}}.
$$
It follows that the contribution of $\bar{P}_{2N}$ to the constant
term of $D_{2N}(\lambda)$ vanishes. Hence for all $I$ (including
trivial $I$), the multiplicity $[\bar{P}_{2I}\bar{P}_{2b} :
D_{2N}(\lambda)]$ of the total contribution of $\bar{P}_{2I}
\bar{P}_{2b}$ to $D_{2N}(\lambda)$ is given by the product of
\begin{multline}\label{bar-inter}
\sum_{k=1}^{2N-1} \left(\lambda\!+\!\f\!-\!N\right)^{2N-k} \\
\times \left( \sum_{j=0}^{k-1} \tau_{2N}^{(2N-j-1)}
\left(-|I|\!-\!\frac{1}{2}\right)^{k-1-j} + \frac{1}{2}
\sum_{j=0}^{k-2} \tau_{2N}^{(2N-j-1)}
\left(-|I|\!-\!\frac{1}{2}\right)^{k-2-j} \right)
\end{multline}
and
$$
m^{(1)}_{(\bar{I},\bar{b})} = (-1)^N \frac{2^{2N-1}}{(2N\!-\!1)!}
m_{(I,b)}.
$$
Now the sum \eqref{bar-inter} can be written in the form
\begin{multline*}
\sum_{1 \le a+b \le 2N-1 \atop 1 \le a \le 2N-1, \, 0 \le b \le
2N-2} \tau_{2N}^{(a+b)} \left(\lambda\!+\!\f\!-\!N\right)^a
\left(-|I|\!-\!\frac{1}{2}\right)^b \\
+ \frac{1}{2} \sum_{2 \le a+b \le 2N-1 \atop 1 \le a \le 2N-1, \, 1
\le b \le 2N-2} \tau_{2N}^{(a+b)}
\left(\lambda\!+\!\f\!-\!N\right)^a
\left(-|I|\!-\!\frac{1}{2}\right)^{b-1}.
\end{multline*}
In order to determine that sum, we apply the following result.

\begin{lemm}\label{sum-bar} For $x \ne y$,
\begin{equation*}
\sum_{1 \le a+b \le 2N-1 \atop 1 \le a \le 2N-1, \, 0 \le b \le
2N-2} \tau_{2N}^{(a+b)} x^a y^b + \frac{1}{2} \sum_{2 \le a+b \le
2N-1 \atop 1 \le a \le 2N-1, \, 1 \le b \le 2N-2} \tau_{2N}^{(a+b)}
x^a y^{b-1} = \frac{y \pi_{2N}(x) - x \pi_{2N}(y)}{(x-y)y}.
\end{equation*}
\end{lemm}

\begin{proof} By an analog of Lemma \ref{sum-gen}, the first sum
equals
$$
x \left(\frac{\tau_{2N}(x)-\tau_{2N}(y)}{x-y}\right).
$$
For the second sum, we find
\begin{align*}
\sum_{k=2}^{2N-1} \tau_{2N}^{(k)} \left( \sum_{a=1}^{k-1} x^a
y^{k-a-1} \right)
& = \sum_{k=2}^{2N-1} \left( \frac{x^k-y^k}{x-y} - y^{k-1}\right)\\
& = \frac{1}{x-y} \left( \sum_{k=2}^{2N-1} \tau_{2N}^{(k)} x^k -
\tau_{2N}^{(k)} y^k \right) - \sum_{k=2}^{2N-1} \tau_{2N}^{(k)}
y^{k-1} \\ & = \frac{\tau_{2N}(x) - \tau_{2N}(y)}{x-y} -
\frac{\tau_{2N}(y)}{y}.
\end{align*}
Hence the total sum equals
\begin{align*}
& x \frac{\tau_{2N}(x)-\tau_{2N}(y)}{x-y} + \frac{1}{2}
\frac{\tau_{2N}(x) - \tau_{2N}(y)}{x-y} - \frac{1}{2}
\frac{\tau_{2N}(y)}{y} \\
& = \frac{y (x+\frac{1}{2}) \tau_{2N}(x) - x (y+\frac{1}{2})
\tau_{2N}(y)}{(x-y) y} \\
& = \frac{y \pi_{2N}(x) - x \pi_{2N}(y)}{(x-y)y}.
\end{align*}
The proof is complete.
\end{proof}

Lemma \ref{sum-bar} and $\pi_{2N}(-|I|\!-\!\frac{1}{2}) = 0$ for $0
\le |I| \le N-1$ imply that the multiplicity $[\bar{P}_{2I}
\bar{P}_{2j} : D_{2N}(\lambda)]$ of the total contribution of
$\bar{P}_{2I} \bar{P}_{2j}$ to $D_{2N}(\lambda)$ equals
\begin{equation}\label{c-pure-bar-generic}
-(-1)^N \frac{2^{2N-1}}{(2N\!-\!1)!}
\frac{\pi_{2N}(\lambda\!+\!\f\!-\!N)}
{\lambda\!+\!\frac{n+1}{2}\!-\!N\!+\!|I|} m_{(\bar{I},\bar{b})}
\end{equation}
if $\lambda\!+\!\frac{n+1}{2}\!-\!N\!+\!|I| \ne 0$, and
\begin{equation}\label{c-pure-bar-deg}
-(-1)^N \frac{2^{2N-1}}{(2N\!-\!1)!} \pi_{2N}'\left(-|I|\!-\!
\frac{1}{2}\right) m_{(\bar{I},\bar{b})}
\end{equation}
if $\lambda\!+\!\frac{n+1}{2}\!-\!N\!+\!|I| = 0$. In particular,
$\bar{P}_{2N}$ contributes to $D_{2N}(\lambda)$ by
\begin{equation}
-(-1)^N  \frac{2^{2N-1}}{(2N\!-\!1)!}
\frac{\pi_{2N}(\lambda\!+\!\f\!-\!N)}{\lambda\!+\!\frac{n+1}{2}\!-\!N}.
\end{equation}
These results yield the sum in \eqref{P-bar}.

%%%%%%%%%%%%%%%%%%%%%%%%%%%%%%%%%%%%%%%%%%%%%%%%%%%%%%%%%%%%%%%%%%%%%%%%%%%
\subsection{The mixed terms}\label{mixed-terms}

Let
$$
m_{(I,\bar{J})}^{(k)}, \; |I|+|J|=N
$$
be the coefficient of $P_{2I} \bar{P}_{2J}$ in the coefficient
$d_{2N}^{(k)}$. Here both $I$ and $J$ are non-trivial. Then this
coefficient is given by
\begin{equation}
m^{(k)}_{(I,\bar{J})} \st \mu_{(I_l,J_r)}^{(2N-k)}(N)
m^{(2)}_{(I,\bar{J})},
\end{equation}
where
$$
\frac{\pi_{2N}(x)}{(x\!-\!N\!+\!a)(x\!+\!N\!-\!b\!+\!\frac{1}{2})} =
\sum_{k=0}^{2N-2} \mu_{(a,b)}^{(k)}(N) x^k.
$$
Finally, we have
\begin{equation}\label{mixed-base}
m^{(2)}_{(I,\bar{J})} \st -(-1)^N 2^{2N-1}
\frac{N!(N\!-\!1)!}{(2N\!-\!1)!} \frac{m_I m_J}{|I|!(|I|\!-\!1)!
|J|!(|J|\!-\!1)!}.
\end{equation}
Hence the mixed contributions to $D_{2N}(\lambda)$ are given by the
sum
$$
\sum_{|I|+|J|=N} \frac{\pi_{2N}(\lambda\!+\!\f\!-\!N)}{
(\lambda\!+\!\f\!-\!2N\!+\!I_l)(\lambda\!+\!\frac{n+1}{2}\!-\!J_r)}
m^{(2)}_{(I,\bar{J})} P_{2I} \bar{P}_{2J}
$$
or, equivalently,
\begin{multline}\label{mixed-total}
-(-1)^N 2^{2N-1} \frac{N!(N\!-\!1)!}{(2N\!-\!1)!} \\ \times
\sum_{|I|+|J|=N} \frac{\pi_{2N}(\lambda\!+\!\f\!-\!N)}
{(\lambda\!+\!\f\!-\!2N\!+\!I_l)(\lambda\!+\!\frac{n+1}{2}\!-\!J_r)}
\frac{m_I m_J}{|I|!(|I|\!-\!1)! |J|!(|J|\!-\!1)!} P_{2I}
\bar{P}_{2J}.
\end{multline}
Here $I_l$ and $J_r$ denote the most left and most right entries of
the compositions $I$ and $J$, respectively. These formulas are valid
for all $\lambda$ by reducing the fractions if necessary. These
results yield the sum in \eqref{mixed}.

%%%%%%%%%%%%%%%%%%%%%%%%%%%%%%%%%%%%%%%%%%%%%%%%%%%%%%%%%%%%%%%%%%%%%%%%%%
\section{Factorization relations for $D_{2N}(\lambda)$}\label{factor-D}

For the proof of Theorem \ref{D-rep}, it suffices to prove that the
families $D_{2N}(\lambda)$ satisfy the factorization relations
\begin{equation}\label{fact-1}
D_{2N}\left(-\f\!+\!2N\!-\!j\right) = P_{2j}
D_{2N-2j}\left(-\f\!+\!2N\!-\!j\right), \; j=1,\dots,N
\end{equation}
and
\begin{equation}\label{fact-2}
D_{2N}\left(-\frac{n\!+\!1}{2}\!+\!j\right) =
D_{2N-2j}\left(-\frac{n\!+\!1}{2}\!-\!j\right) \bar{P}_{2j}, \;
j=1,\dots,N.
\end{equation}
In fact, the relations in the systems \eqref{fact-1} and
\eqref{fact-2} are analogs of the relations \eqref{Fact-A} and
\eqref{Fact-B}. But since both families $D_{2N}(\lambda)$ and
$D_{2N}^{res}(\lambda)$ are polynomials of degree $\le 2N\!-\!1$ in
$\lambda$, they are characterized by the respective systems of
factorization identities.

A central role in the following arguments will be played by the
identity
\begin{multline}\label{final-1}
\left(2N\!-\!j\!-\!\frac{1}{2}\right) \cdots
\left(N\!-\!j\!+\!\frac{1}{2}\right) \frac{(N\!-\!1)!}{(2N\!-\!1)!} \\
= \left(2N\!-\!j\!-\!\frac{1}{2}\right) \cdots
\left(N\!+\!\frac{1}{2}\right) 2^{-2j}
\frac{(N\!-\!j\!-\!1)!}{(2N\!-\!2j\!-\!1)!}.
\end{multline}
It is straightforward to verify this relation.

Now, in order to prove the factorizations \eqref{fact-1} and
\eqref{fact-2}, we compare the contributions of all possible types
of products of GJMS-operators on both sides.

We start by proving that the contributions of the compositions
$$
P_{2j} P_{2I} \bar{P}_{2J}, \quad  j\!+\!|I|\!+\!|J| = N
$$
with non-trivial $I$ and $J$ on both sides of \eqref{fact-1}
coincide. By \eqref{mixed-total}, the assertion is equivalent to the
identity
\begin{multline*}
\frac{\pi'_{2N}(N\!-\!j)}{(2N\!+\!\frac{1}{2}\!-\!j\!-\!J_r)}
\frac{2^{2N-1}}{(2N\!-\!1)!} \frac{N!(N\!-\!1)!}{(|I|\!+\!j)!
(|I|\!+\!j\!-\!1)! |J|! (|J|\!-\!1)!} m_{(j,I)} m_J \\
= (-1)^j \frac{\pi_{2N-2j}(N)}
{(j+I_l)(2N\!+\!\frac{1}{2}\!-\!j\!-\!J_r)}
\frac{2^{2N-2j-1}}{(2N\!-\!2j\!-\!1)!}\frac{(N\!-\!j)!(N\!-\!j\!-\!1)!}
{|I|!(|I|\!-\!1)!|J|!(|J|\!-\!1)!} m_I m_J.
\end{multline*}
Now the identities
\begin{equation}\label{pi-a}
\pi'_{2N}(N\!-\!j) = -(-1)^j(N\!-\!j)! (j\!-\!1)!
\left(2N\!-\!j\!-\!\frac{1}{2}\right) \cdots
\left(N\!-\!j\!+\!\frac{1}{2}\right)
\end{equation}
and
\begin{equation}\label{pi-b}
\pi_{2N-2j}(N) = \frac{N!}{j!}\left(2N\!-\!j\!-\!\frac{1}{2}\right)
\cdots \left(N\!+\!\frac{1}{2}\right)
\end{equation}
simplify the assertion to
\begin{multline*}
(j\!-\!1)! \left(2N\!-\!j\!-\!\frac{1}{2}\right) \cdots
\left(N\!-\!j\!+\!\frac{1}{2}\right) \frac{1}{(2N\!-\!1)!}
\frac{N!(N\!-\!1)!}{(|I|\!+\!j)!(|I|\!+\!j\!-\!1)!} m_{(j,I)} \\
= - \frac{N!}{j!} \left(2N\!-\!j\!-\!\frac{1}{2}\right) \cdots
\left(N\!+\!\frac{1}{2}\right) \frac{1}{j+I_l} 2^{-2j}
\frac{(N\!-\!j\!-\!1)!}{(2N\!-\!2j\!-\!1)!|I|!(|I|\!-\!1)!} m_I.
\end{multline*}
But the definition of the multiplicities $m_I$ implies the relation
$$
m_{(j,I)} = -\frac{1}{j\!+\!I_l} \binom{j\!+\!|I|}{j}^2
\frac{j|I|}{j\!+\!|I|} m_I = -\frac{1}{j\!+\!I_l} \frac{(j\!+\!|I|)!
(j\!+\!|I|\!-\!1)!}{j!(j\!-\!1)! |I|!(|I|\!-\!1)!} m_I.
$$
It further simplifies the assertion to \eqref{final-1}.

Next, we prove that the contributions of the compositions
$$
P_{2I} \bar{P}_{2J} \bar{P}_{2j}, \quad |I|\!+\!|J|\!+\!j = N
$$
with non-trivial $I$ and $J$ on both sides of \eqref{fact-2}
coincide. By \eqref{mixed-total}, the assertion is equivalent to
\begin{multline*}
\frac{\pi'_{2N}\left(-\frac{1}{2}\!+\!j\!-\!N\right)}
{(-\!\frac{1}{2}\!-\!2N\!+\!j\!+\!I_l)} \frac{2^{2N-1}}{(2N\!-\!1)!}
\frac{N!(N\!-\!1)!}{(|I|)!(|I|\!-\!1)! (|J|\!+\!j)!
(|J|\!+\!j\!-\!1)!} m_{I} m_{(J,j)} \\ = (-1)^j
\frac{\pi_{2N-2j}\left(-\frac{1}{2}\!-\!N\right)}
{(-J_r\!-\!j)(-\frac{1}{2}\!+\!j\!-\!2N\!+\!I_l)}
\frac{2^{2N-2j-1}}{(2N\!-\!2j\!-\!1)!}
\frac{(N\!-\!j)!(N\!-\!j\!-\!1)!}{|I|!(|I|\!-\!1)!|J|!(|J|\!-\!1)!}
m_I m_J.
\end{multline*}
Now the identities
\begin{equation}\label{pi-c}
\pi'_{2N}\left(-\frac{1}{2}\!+\!j\!-\!N\right) = (-1)^j (N\!-\!j)!
(j\!-\!1)! \left(2N\!-\!j\!-\!\frac{1}{2}\right) \cdots
\left(N\!-\!j\!+\!\frac{1}{2}\right)
\end{equation}
and
\begin{equation}\label{pi-d}
\pi_{2N-2j} \left(-\frac{1}{2}\!-\!N\right) =
\frac{N!}{j!}\left(2N\!-\!j\!-\!\frac{1}{2}\right) \cdots
\left(N\!+\!\frac{1}{2}\right)
\end{equation}
simplify the assertion to
\begin{multline*}
-(j\!-\!1)! \left(2N\!-\!j\!-\!\frac{1}{2}\right) \cdots
\left(N\!-\!j\!+\!\frac{1}{2}\right) \frac{1}{(2N\!-\!1)!}
\frac{N!(N\!-\!1)!}{(|J|\!+\!j)!(|J|\!+\!j\!-\!1)!} m_{(J,j)} \\
= \frac{N!}{j!} \left(2N\!-\!j\!-\!\frac{1}{2}\right) \cdots
\left(N\!+\!\frac{1}{2}\right) \frac{1}{J_r\!+\!j} 2^{-2j}
\frac{(N\!-\!j\!-\!1)!}{(2N\!-\!2j\!-\!1)!|J|!(|J|\!-\!1)!} m_J.
\end{multline*}
But the definition of the multiplicities $m_I$ implies the relation
$$
m_{(J,j)} = -\frac{1}{J_r\!+\!j} \binom{j\!+\!|J|}{j}^2
\frac{j|J|}{j\!+\!|J|} m_J = -\frac{1}{J_r\!+\!j} \frac{(j\!+\!|J|)!
(j\!+\!|J|\!-\!1)!}{j!(j\!-\!1)! |J|!(|J|\!-\!1)!} m_J
$$
which simplifies the assertion to \eqref{final-1}.

Next, we prove that the contributions of
$$
P_{2j} P_{2I}, \quad j\!+\!|I|=N
$$
on both sides of \eqref{fact-1} coincide. On the one hand,
\eqref{c-pure-deg} shows that $P_{2j} P_{2I}$ contributes to
$$
D_{2N}\left(-\f\!+\!2N\!-\!j\right)
$$
with the coefficient
$$
-(-1)^N \frac{2^{2N-1}}{(2N\!-\!1)!} \pi_{2N}'(N\!-\!j) m_{(j,I)}.
$$
By \eqref{pi-a}, the latter term equals
$$
(-1)^{N-j} 2^{2N-1} \frac{(N\!-\!j)!(j\!-\!1)!}{(2N\!-\!1)!}
\left(2N\!-\!j\!-\!\frac{1}{2}\right) \cdots
\left(N\!-\!j\!+\!\frac{1}{2}\right) m_{(j,I)}.
$$
We compare this coefficient with the coefficient of the contribution
of $P_{2I}$ to
$$
D_{2N-2j}\left(-\f\!+\!2N\!-\!j\right).
$$
For this purpose, we write $I$ in the form $(I_l,K)$ with a possibly
trivial $K$. Let $K$ be non-trivial. Then $|K|=N\!-\!I_l\!-\!j$. Now
by \eqref{c-pure-generic}, the composition $P_{2I_l} P_{2K}$
contributes with the coefficient
\begin{equation*}
\frac{\pi_{2N-2j}\left(\lambda\!+\!\f\!-\!(N\!-\!j)\right)}
{\lambda\!+\!\f\!-\!(N\!-\!j)-|K|} m_{(I_l,K)}^{(1)} =
\frac{\pi_{2N-2j}(N)}{I_l\!+\!j} m_{(I_l,K)}^{(1)}.
\end{equation*}
By \eqref{pi-b}, it equals
\begin{equation*}
-(-1)^{N-j} 2^{2(N-j)-1} \frac{N!}{(2N\!-\!2j\!-\!1)!j!}
\frac{1}{I_l\!+\!j}\left(2N\!-\!j\!-\!\frac{1}{2}\right) \cdots
\left(N\!+\!\frac{1}{2}\right) m_{(I_l,K)}.
\end{equation*}
We claim that the latter coefficient coincides with
$$
(-1)^{N-j} 2^{2N-1} \frac{(N\!-\!j)!(j\!-\!1)!}{(2N\!-\!1)!}
\left(2N\!-\!j\!-\!\frac{1}{2}\right) \cdots
\left(N\!-\!j\!+\!\frac{1}{2}\right) m_{(j,I_l,K)}.
$$
The assertion is equivalent to the identity
\begin{multline}
-2^{-2j} \frac{N!}{(2N\!-\!2j\!-\!1)!j!}
\left(2N\!-\!j\!-\!\frac{1}{2}\right)\cdots\left(N\!+\!\frac{1}{2}\right)
\frac{1}{j+I_l} m_{(I_l,K)} \\ =
\frac{(N\!-\!j)!(j\!-\!1)!}{(2N\!-\!1)!}
\left(2N\!-\!j\!-\!\frac{1}{2}\right) \cdots
\left(N\!-\!j\!+\!\frac{1}{2}\right) m_{(j,I_l,K)}.
\end{multline}
But by the definition of the multiplicities $m_I$,
$$
m_{(j,I_l,K)} = -\frac{1}{j+I_l}
\frac{N!(N\!-\!1)!}{j!(j\!-\!1)!(N\!-\!j)!(N\!-\!j\!-\!1)!}
m_{(I_l,K)}.
$$
The latter relation shows that the assertion is equivalent to
\begin{multline*}
2^{-2j} \frac{1}{(2N\!-\!2j\!-\!1)!}
\left(2N\!-\!j\!-\!\frac{1}{2}\right)\cdots\left(N\!+\!\frac{1}{2}\right) \\
= \frac{(N\!-\!1)!}{(2N\!-\!1)!}
\left(2N\!-\!j\!-\!\frac{1}{2}\right) \cdots
\left(N\!-\!j\!+\!\frac{1}{2}\right) \frac{1}{(N\!-\!j\!-\!1)!}.
\end{multline*}
But this identity is equivalent to \eqref{final-1}. This completes
the proof for non-trivial $K$. The proof immediately extends to
trivial $K$.

Next, we prove that the contributions of
$$
P_{2j} \bar{P}_{2J}, \quad j\!+\!|J|=N
$$
on both sides of \eqref{fact-1} coincide. On the one hand,
\eqref{mixed-total} shows that $P_{2j} \bar{P}_{2J}$ contributes to
$$
D_{2N}\left(-\f\!+\!2N\!-\!j\right)
$$
with the coefficient
$$
-(-1)^N 2^{2N-1} \frac{N!(N\!-\!1)!}{(2N\!-\!1)!}
\frac{\pi_{2N}'(N\!-\!j)}{(2N\!-\!j\!-\!J_r\!+\frac{1}{2})}
\frac{1}{j!(j\!-\!1)!(N\!-\!j)!(N\!-\!j\!-\!1)!} m_J.
$$
By \eqref{pi-a}, the latter term equals
\begin{multline}\label{j-bar-l}
-(-1)^N 2^{2N-1} \frac{N!(N\!-\!1)!}{(2N\!-\!1)!}
\frac{1}{j!(N\!-\!j\!-\!1)!} \\
\times \frac{1}{(2N\!-\!j\!-\!J_r\!+\frac{1}{2})}
\left(2N\!-\!j\!-\!\frac{1}{2}\right) \cdots
\left(N\!-\!j\!+\!\frac{1}{2}\right) m_J.
\end{multline}
We compare this coefficient with the coefficient of the contribution
of $\bar{P}_{2J}$ to
$$
D_{2N-2j}\left(-\f\!+\!2N\!-\!j\right).
$$
For this purpose, we write $J = (K,J_r)$ with a possibly trivial
$K$. Let $K$ be non-trivial. Then $|K|=N\!-\!J_r\!-\!j$. Now by
\eqref{c-pure-bar-generic}, the composition $\bar{P}_{2K}
\bar{P}_{2J_r}$ contributes with the coefficient
\begin{equation*}
(-1)^{N-j} \frac{2^{2N-2j-1}}{(2N\!-\!2j\!-\!1)!}
\frac{\pi_{2N-2j}(\lambda\!+\!\f\!-\!(N\!-\!j))}{(\lambda\!+\!\frac{n+1}{2}\!-\!(N\!-\!j)\!+\!|K|)}
m_J,
\end{equation*}
i.e., with
\begin{equation*}
(-1)^{N-j} \frac{2^{2N-2j-1}}{(2N\!-\!2j\!-\!1)!}
\frac{\pi_{2N-2j}(N)}{(2N\!-\!j\!-\!J_r\!+\!\frac{1}{2})} m_J.
\end{equation*}
By \eqref{pi-b}, it equals
\begin{equation}\label{j-bar-r}
(-1)^{N-j} \frac{2^{2N-2j-1}}{(2N\!-\!2j\!-\!1)!} \frac{N!}{j!}
\frac{1}{(2N\!-\!j\!-\!J_r\!+\!\frac{1}{2})}
\left(2N\!-\!j\!-\!\frac{1}{2}\right) \cdots
\left(N\!+\!\frac{1}{2}\right) m_J.
\end{equation}
It is straightforward to verify that the coincidence of
\eqref{j-bar-l} and \eqref{j-bar-r} is equivalent to
\eqref{final-1}. This completes the proof for non-trivial $K$. The
proof immediately extends to trivial $K$.

Next, we prove that the contributions of
$$
P_{2I} \bar{P}_{2j}, \quad |I|\!+\!j=N
$$
on both sides of \eqref{fact-2} coincide. On the one hand, we
determine the coefficient of the contribution of $P_{2I}$ to
$$
D_{2N-2j}\left(-\frac{n\!+\!1}{2}\!-\!j\right).
$$
For this purpose, we write $I=(I_l,K)$ with a possibly trivial $K$.
Let $K$ be non-trivial. By \eqref{c-pure-generic}, $P_{2I}$
contributes with the coefficient
$$
\frac{\pi_{2N-2j}(\lambda\!+\!\f\!-\!(N\!-\!j))}{(\lambda\!+\!\f\!-\!(N\!-\!j)-|K|)}
m_I^{(1)}.
$$
Using $|K|=N\!-\!j\!-\!I_l$ and \eqref{pi-d}, the latter expression
simplifies to
\begin{equation}\label{left}
(-1)^{N-j} \frac{2^{2N-2j-1}}{(2N\!-\!2j\!-\!1)!} \frac{N!}{j!}
\frac{1}{(2N\!+\!\frac{1}{2}\!-\!j\!-\!I_l)}
\left(2N\!-\!j-\!\frac{1}{2}\right) \cdots
\left(N\!+\!\frac{1}{2}\right) m_I.
\end{equation}
On the other hand, we determine the coefficient of the contribution
of the operator $P_{2I} \bar{P}_{2j}$ to
$$
D_{2N}\left(-\frac{n\!+\!1}{2}\!+\!j\right).
$$
Eq.~\eqref{mixed-total} yields
$$
-(-1)^N 2^{2N-1} \frac{N!(N\!-\!1)!}{(2N\!-\!1)!} \frac{\pi_{2N}'
(-\frac{1}{2}\!+\!j\!-\!N)}{(-\frac{1}{2}\!+\!j\!-\!2N\!+\!I_l)}
\frac{1}{(N\!-\!j)!(N\!-\!j\!-\!1)!j!(j\!-\!1)!} m_I.
$$
By \eqref{pi-c}, the latter formula simplifies to
\begin{equation}\label{right}
(-1)^{N-j} 2^{2N-1} \frac{N!(N\!-\!1)!}{(2N\!-\!1)!}
\frac{(2N\!-\!j\!-\!\frac{1}{2}) \cdots
(N\!-\!j\!+\!\frac{1}{2})}{(2N\!+\!\frac{1}{2}\!-\!j\!+\!I_l)(N\!-\!j\!-\!1)!
j!} m_I.
\end{equation}
It is straightforward to verify that the equality of \eqref{left}
and \eqref{right} is equivalent to the relation \eqref{final-1}.
This completes the proof for non-trivial $K$. The proof immediately
extends to trivial $K$.

Next, we prove that the contributions of
$$
\bar{P}_{2J} \bar{P}_{2j}, \quad |J|\!+\!j=N
$$
on both sides of \eqref{fact-2} coincide. On the one hand,
Eq.~\eqref{c-pure-bar-generic} shows that these terms contribute to
$$
D_{2N}\left(-\frac{n\!+\!1}{2}\!+\!j\right)
$$
with the coefficient
$$
\pi_{2N}' \left(-\frac{1}{2}\!+\!j\!-\!N\right) m_{(J,j)}^{(1)},
$$
i.e., with
\begin{equation}\label{bar-left}
(-1)^j (N\!-\!j)! (j\!-\!1)! \left(2N\!-\!j\!-\!\frac{1}{2}\right)
\cdots \left(N\!-\!j\!+\!\frac{1}{2}\right) m_{(J,j)}^{(1)}.
\end{equation}
We compare this coefficient with the coefficient of the contribution
of $\bar{P}_{2J}$ to
$$
D_{2N-2j}\left(-\frac{n\!+\!1}{2}\!-\!j\right).
$$
For this purpose, we write $J$ in the form $(K,J_r)$ with a possibly
trivial $K$. Let $K$ be non-trivial. Now by
\eqref{c-pure-bar-generic} $\bar{P}_{2K} \bar{P}_{2J_r}$ contributes
with the coefficient
$$
-\frac{\pi_{2N-2j}(\lambda\!+\!\f\!-\!(N\!-\!j))}
{\lambda\!+\!\f\!-\!(N\!-\!j)\!+\!|K|\!+\!\frac{1}{2}}
m_{(K,J_r)}^{(1)},
$$
i.e., with
\begin{equation}\label{bar-right}
-\frac{\pi_{2N-2j}(-N\!-\!\frac{1}{2})}{J_r\!+\!j}
m_{(K,J_r)}^{(1)}.
\end{equation}
Now we claim that the latter coefficient coincides with
\eqref{bar-left} for $J=(K,J_r)$, i.e., with
\begin{equation}\label{bar-left-2}
(-1)^j (N\!-\!j)! (j\!-\!1)! \left(2N\!-\!j\!-\!\frac{1}{2}\right)
\cdots \left(N\!-\!j\!+\!\frac{1}{2}\right) m_{(K,J_r,j)}^{(1)}.
\end{equation}
By \eqref{pi-d}, the assertion is equivalent to the identity
\begin{multline*}\label{bar-rel}
- \frac{N!}{(J_r\!+\!j) j!} \left(N\!+\!\frac{1}{2}\right) \cdots
\left(2N\!-\!j\!-\!\frac{1}{2}\right) m_{(K,J_r)}^{(1)}\\
= (-1)^j (N\!-\!j)! (j\!-\!1)! \left(2N\!-\!j\!-\!\frac{1}{2}\right)
\cdots \left(N\!-\!j\!+\!\frac{1}{2}\right) m_{(K,J_r,j)}^{(1)}.
\end{multline*}
But
$$
m_{(K,J_r,j)}^{(1)} = - (-1)^j \frac{1}{J_r\!+\!j}
\frac{N!(N\!-\!1)!}{j!(j\!-\!1)!(N\!-\!j)!(N\!-\!j\!-\!1)!} 2^{2j}
\frac{(2N\!-\!2j\!-\!1)!}{(2N\!-\!1)!} m_{(K,J_r)}^{(1)}
$$
shows that the latter identity is equivalent to \eqref{final-1}.
This completes the proof for non-trivial $K$. The proof immediately
extends to trivial $K$.

Next, we prove that all contributions of
$$
P_{2I} \;\; \mbox{with $I_l \ne j$}, \quad P_{2I} \bar{P}_{2J} \;\;
\mbox{with} \;I_l \ne j \quad \mbox{and} \quad \bar{P}_{2J}
$$
to
$$
D_{2N}\left(-\f\!+\!2N\!-\!j\right), \; j=1,\dots,N
$$
vanish. Let
$$
\lambda = -\f\!+\!2N\!-\!j.
$$
We recall that
\begin{equation}\label{pi-van}
\pi_{2N}\left(\lambda\!+\!\f\!-\!N\right) = \pi_{2N}(N\!-\!j) = 0
\quad \mbox{for $j=1,\dots,N$.}
\end{equation}

We first consider $P_{2I}$. We write $I=(I_l,J)$ with a possibly
trivial $J$. Let $J$ be non-trivial. Then we have
$$
\lambda\!+\!\f\!-\!N\!-\!|J| = J_l\!-\!j \ne 0,
$$
and the relations \eqref{c-pure-generic} and \eqref{pi-van} show
that the contribution of $P_{2I}$ vanishes. It remains to prove that
$P_{2N}$ does not contribute for $j \ne N$. But by \eqref{single},
the corresponding multiplicity is a multiple of
$$
\eta_{2N}\left(\lambda\!+\!\f\!-\!N\right) = \eta_{2N}(N\!-\!j) = 0.
$$

Next, we consider the mixed contributions. Then we have
$$
\lambda\!+\!\f\!-\!2N\!+\!I_l = I_l\!-\!j \ne 0,
$$
and the relations \eqref{mixed-total} and \eqref{pi-van} show that
the contribution of $P_{2I} \bar{P}_{2J}$ vanishes.

Finally, we write $J=(K,J_r)$ with a possibly trivial $K$. Then we
have
$$
\lambda\!+\!\frac{n\!+\!1}{2}\!-\!N\!+\!|K| =
\frac{1}{2}\!-\!j\!+\!2N\!-\!J_r\! \ne 0,
$$
and the relations \eqref{c-pure-bar-generic} and \eqref{pi-van} show
that the contribution of $\bar{P}_{2J}$ vanishes.

Finally, we prove that all contributions of
$$
P_{2I}, \quad P_{2I} \bar{P}_{2J} \;\; \mbox{with $J_r \ne j$} \quad
\mbox{and} \quad \bar{P}_{2J} \;\; \mbox{with $J_r \ne j$}
$$
to
$$
D_{2N}\left(-\frac{n\!+\!1}{2}\!+\!j\right)
$$
vanish. Let
$$
\lambda = -\frac{n\!+\!1}{2}\!+\!j.
$$
We recall that
\begin{equation}\label{pi-van-bar}
\pi_{2N}\left(\lambda\!+\!\f\!-\!N\right) =
\pi_{2N}\left(j\!-\!N\!-\!\frac{1}{2}\right) = 0 \quad \mbox{for
$j=1,\dots,N$.}
\end{equation}

We first consider $P_{2I}$. We write $I=(I_l,J)$ with a possibly
trivial $J$. Let $J$ be non-trivial. Then we have
$$
\lambda\!+\!\f\!-\!N\!-\!(N\!-\!I_l)=-\frac{1}{2}\!+\!j\!-\!2N\!+\!I_l
\ne 0,
$$
and the relations \eqref{c-pure-generic} and \eqref{pi-van-bar} show
that the contribution of $P_{2I}$ vanishes.

Next, we consider the mixed contributions. Then we have
$$
\lambda\!+\!\frac{n+1}{2}\!-\!J_r = j\!-\!J_r \ne 0,
$$
and the relations \eqref{mixed-total} and \eqref{pi-van-bar} show
that the contribution $P_{2I} \bar{P}_{2J}$ vanishes.

Finally, we write $J=(K,J_r)$ with a possibly trivial $K$. Then we
have
$$
\lambda\!+\!\frac{n\!+\!1}{2}\!-\!N\!+\!|K| = j- J_r \ne 0,
$$
and the relations \eqref{c-pure-bar-generic} and \eqref{pi-van-bar}
show that the contribution $\bar{P}_{2J}$ vanishes.

%%%%%%%%%%%%%%%%%%%%%%%%%%%%%%%%%%%%%%%%%%%%%%%%%%%%%%%%%%%%%%%%%%%%%
\section{The restriction property and the commutator relations}
\label{RP-and-CR}

In the present section, we derive two consequences of Theorem
\ref{D-rep} which will play a central role in what
follows.\footnote{In even dimensions and for general metrics the
usual restriction $2N \le n$ will be in place without mentioning it.
Moreover, for locally conformally flat metrics the statements extend
to all $N \ge 1$.}

\begin{thm}\label{AB} For $N \ge 1$, let
\begin{equation*}
D_{2N}^{res}(\lambda) = A_{2N}
\left(\lambda\!+\!\f\!-\!N\right)^{2N-1} + B_{2N}
\left(\lambda\!+\!\f\!-\!N\right)^{2N-2} + \cdots + P_{2N}.
\end{equation*}
Then the leading coefficient $A_{2N}$ is given by the formula
\begin{equation}\label{leading}
(-1)^{N-1} 2^{-(2N-1)} (2N\!-\!1)! A_{2N} =
\M_{2N}\!-\!\bar{\M}_{2N}.
\end{equation}
Moreover, the coefficient
$$
(-1)^{N-1} 2^{-(2N-2)} (2N\!-\!1)! B_{2N}
$$
coincides with the sum of
$$
\M_{2N} + (N\!-\!1) (\M_{2N} \!-\! \bar{\M}_{2N})
$$
and
\begin{equation}\label{comm}
2 \sum_{a=1}^{N-1} a \binom{N\!-\!1}{a}^2 \left[ (\M_{2N-2a} \!-\!
\bar{\M}_{2N-2a}) \bar{\M}_{2a} \!-\! \M_{2a} (\M_{2N-2a} \!-\!
\bar{\M}_{2N-2a})\right].
\end{equation}
Here we used the convention to simplify formulas by omitting the
restriction operators $i^*$.
\end{thm}

\begin{proof} The formula for $A_{2N}$ follows from the definitions
in Sections \ref{pure-P}--\ref{pure-P-bar}. In order to prove the
assertion for the sub-leading coefficient, we first make that
coefficient explicit by using the corresponding formulas in Sections
\ref{pure-P}--\ref{mixed-terms}. For the pure $P$-terms we find the
formula
\begin{align}\label{sub-pure-P}
m_I^{(2)} P_{2I} & = \left( (N\!-\!I_l) \eta_{2N}^{(2N-1)} +
\eta_{2N}^{(2N-2)} \right) m_I^{(1)} P_{2I} \nonumber \\
& = (-1)^{N-1} \frac{2^{2N-1}}{(2N\!-\!1)!}
\left( (N\!-\!I_l) + \frac{N}{2} \right)  m_I P_{2I} \nonumber \\
& = (-1)^{N-1} \frac{2^{2N-2}}{(2N\!-\!1)!} (3N\!-\!2I_l)  m_I
P_{2I}
\end{align}
using
$$
\eta_{2N}^{(2N-1)} = 1 \quad \mbox{and} \quad \eta_{2N}^{(2N-2)} =
\frac{N}{2}.
$$
Similarly, for the pure $\bar{P}$-terms we find
\begin{align}\label{sub-pure-P-bar}
m_{\bar{J}}^{(2)} \bar{P}_{2J} & = \left(
\left(-(N\!-\!J_r)\!-\!\frac{1}{2}\right) \tau_{2N}^{(2N-1)} +
\tau_{2N}^{(2N-2)} + \frac{1}{2} \tau_{2N}^{(2N-1)} \right)
m_{\bar{J}}^{(1)} \bar{P}_{2J} \nonumber \\
& = (-1)^{N-1} \frac{2^{2N-2}}{(2N\!-\!1)!} (N\!-\!2J_r\!+\!1)  m_J
\bar{P}_{2J}
\end{align}
using
$$
\tau_{2N}^{(2N-1)} = 1 \quad \mbox{and} \quad \tau_{2N}^{(2N-2)} =
\frac{N\!-\!1}{2}.
$$
Finally, for the mixed terms, \eqref{mixed-base} states
$$
m_{(I,\bar{J})}^{(2)} = (-1)^{N-1} 2^{2N-1}
\frac{N!(N\!-\!1)!}{(2N\!-\!1)!} \frac{m_I m_J}{|I|!(|I|\!-\!1)!
|J|!(|J|\!-\!1)!}.
$$
Now we prove that these contributions coincide with those given in
the theorem. In fact, for the pure $P$-terms and the pure
$\bar{P}$-terms, the assertions are
\begin{equation}\label{quad}
N m_I - 2 \sum_{a=1}^{N-1} a \binom{N\!-\!1}{a}^2 [P_{2I}: \M_{2a}
\M_{2N-2a}] = (3N\!-\!2I_l) m_I
\end{equation}
and
\begin{equation}\label{quad-bar}
-(N\!-\!1) m_J - 2 \sum_{a=1}^{N-1} a \binom{N\!-\!1}{a}^2 [
\bar{P}_{2J} : \bar{\M}_{2N-2a} \bar{\M}_{2a} ] = (N\!-\!2J_r\!+\!1)
m_J.
\end{equation}
Moreover, for the mixed terms the assertion is
\begin{multline*}
\sum_{a=1}^{N-1} a \binom{N\!-\!1}{a}^2 [P_{2I}\bar{P}_{2J}:
(\M_{2N-2a} \bar{\M}_{2a} \!+\! \M_{2a} \bar{\M}_{2N-2a})] \\[-4mm]
= \frac{N!(N\!-\!1)!}{|I|!(|I|\!-\!1)! |J|!(|J|\!-\!1)!} m_I m_J,
\end{multline*}
i.e.,
$$
|J| \binom{N\!-\!1}{|J|}^2 + |I| \binom{N\!-\!1}{|I|}^2 =
\frac{N!(N\!-\!1)!}{|I|!(|I|\!-\!1)! |J|!(|J|\!-\!1)!}.
$$
The latter relation easily follows from $|I|+|J|=N$. The remaining
relations \eqref{quad} and \eqref{quad-bar} are consequences of the
following conformal variational formula.

\begin{thm}[\cite{juhl-power}]\label{c-cv} On manifolds $M$ of dimension
$n$, we have
\begin{multline}\label{c-cv-gen}
-(d/dt) \big|_0 \left( e^{(\f+N)t\varphi} \M_{2N}(e^{2t\varphi}g)
e^{-(\f-N)t\varphi} \right) \\ = \sum_{a=1}^{N-1} a
\binom{N\!-\!1}{a}^2 \left[\M_{2N-2a}(g),[\M_{2a}(g),\varphi]\right]
\end{multline}
for all $\varphi \in C^\infty(M)$. On the right-hand side of
\eqref{c-cv-gen}, $\varphi$ is regarded as a multiplication
operator.
\end{thm}

In fact, by the conformal covariance of the GJMS-operators, one
finds that the equality of the respective contributions of the term
$\varphi P_{2I}$ for the composition $I=(I_1,\dots,I_r)$ to both
sides of \eqref{c-cv-gen} is equivalent to the relation
\begin{multline}\label{comb-first}
-(N\!-\!I_1) m_I = \binom{N\!-\!1}{I_1}^2 I_1 \;
m_{(I_1)} m_{(I_2,I_2,\dots,I_{r})} \\
+ \cdots + \binom{N\!-\!1}{I_1\!+\dots+\!I_{r-1}}^2
(I_1\!+\dots+\!I_{r-1}) \; m_{(I_1,\dots,I_{r-1})} m_{(I_r)}.
\end{multline}
Similarly, the equality of the respective contributions of the term
$P_{2J} \varphi$ for the composition $J=(J_1,\dots,J_r)$ to both
sides of \eqref{c-cv-gen} is equivalent to the relation
\begin{multline}\label{comb-last}
-(N\!-\!J_r) m_J = \binom{N\!-\!1}{J_2\!+\dots+\!J_r}^2
(J_2\!+\dots+\!J_r) \; m_{(J_1)} m_{(J_2,\dots,J_r)} \\
+ \cdots + \binom{N\!-\!1}{J_r}^2 J_r \; m_{(J_1,J_2,\dots,J_{r-1})}
m_{(J_r)}.
\end{multline}
In order to prove \eqref{comb-first}, we use the explicit formula
\eqref{m-form} for the coefficients $m_I$ to write the terms on the
right-hand side as multiples of $m_I$. We find
\begin{multline}\label{comb-red}
-\frac{1}{N} \Big[ (I_2+\cdots+I_r)(I_1+I_2) + (I_3+\cdots+I_r)(I_2+I_3) \\
+ \cdots + I_r (I_{r-1}+I_r)\Big] m_I.
\end{multline}
Now the relation
\begin{equation*}
(I_2+\cdots+I_r)(I_1+I_2) + \cdots + I_r (I_{r-1}+I_r) =
(I_1+\cdots+I_r)(I_2+\cdots+I_r)
\end{equation*}
(which follows by induction on the number of entries) implies that
in \eqref{comb-red} the sum in parentheses equals $N(N-I_1)$ if
$|I|=N$. Thus, \eqref{comb-red} equals $-(N-I_1)m_I$. This proves
the assertion and \eqref{quad}.

Similarly, the identity \eqref{comb-last} follows by applying
\eqref{comb-first} to the inverse composition $I^{-1}$ of $I$ using
the relations $m_{I^{-1}} = m_I$. This proves \eqref{quad-bar}.
\end{proof}

As a consequence of Theorem \ref{AB}, we obtain the restriction
property of $\M_{2N}$.

\begin{thm}[{\bf Restriction property}]\label{rest-prop} For $N \ge 2$,
we have
\begin{equation}\label{RP}
\M_{2N} i^* = i^* \bar{\M}_{2N}.
\end{equation}
\end{thm}

The special case $\M_{4} i^* = i^* \bar{\M}_{4}$ of the restriction
property \eqref{RP} was found and applied in \cite{juhl-book} (see
Lemma 6.11.6).

\begin{proof} We recall that the degree of the polynomial
$\lambda \mapsto D_{2N}^{res}(\lambda)$ is $N$. But since $N < 2N-1$
for $N \ge 2$, the coefficient $A_{2N}$ vanishes. Thus, the
assertion follows from \eqref{leading}.
\end{proof}

Similarly, for $N \ge 3$, the sub-leading coefficient of the degree
$N$ polynomial $D_{2N}^{res}(\lambda)$ vanishes. Hence the second
part of Theorem \ref{AB} and Theorem \ref{rest-prop} imply that
$$
\M_{2N} + (2N\!-\!2) ( (\M_{2} \!-\! \bar{\M}_{2}) \bar{\M}_{2N-2}
\!-\! \M_{2N-2} (\M_{2} \!-\! \bar{\M}_{2} )) = 0.
$$
Now adding a multiple of the relation
$$
\M_{2} (\M_{2N-2} - \bar{\M}_{2N-2}) - (\M_{2N-2} - \bar{\M}_{2N-2})
\bar{\M}_2 = 0
$$
cancels the mixed terms $\M_2 \bar{\M}_{2N-2}$ and $\M_{2N-2}
\bar{\M}_2$, and proves

\begin{thm}[{\bf Commutator relations}]\label{base-CR} For $N \ge 3$,
the commutator relations
\begin{equation}\label{CR}
\M_{2N} i^* = (2N\!-\!2) (i^*[\bar{\M}_2,\bar{\M}_{2N-2}] -
[\M_2,\M_{2N-2}]i^*)
\end{equation}
hold true.
\end{thm}

The special case $N=3$ of the commutator relations \eqref{CR}
already played an important role in \cite{juhl-book} (see Theorem
6.11.17).

We also make the lowest-order cases of Theorem \ref{AB} fully
explicit.

\begin{ex}\label{AB-4} $-3! A_4 = 8 (\M_4 \!-\! \bar{\M}_4)$ and
\begin{align*}
-\frac{3!}{4} B_4 & = \M_4 + (\M_4 \!-\! \bar{\M}_4) + 2 [(\M_2
\!-\! \bar{\M}_2) \bar{\M}_2 - \M_2 (\M_2 \!-\! \bar{\M}_2)].
\end{align*}
\end{ex}

\begin{ex}\label{AB-6} $5! A_6 = 32 (\M_6 \!-\! \bar{\M}_6)$ and
\begin{align*}
\frac{5!}{16} B_6 & = \M_6 + 2 (\M_6 \!-\! \bar{\M}_6) \\
& + 8\,[(\M_4 \!-\! \bar{\M}_4) \bar{\M}_2 - \M_2 (\M_4 \!-\! \bar{\M}_4)] \\
& + 4\,[(\M_2\!-\!\bar{\M}_2)\bar{\M}_4-\M_4(\M_2\!-\!\bar{\M}_2)].
\end{align*}
In particular, the restriction properties $\M_{2N} = \bar{\M}_{2N}$
for $N=2,3$ imply the commutator relation for $\M_6$.
\end{ex}

\begin{ex}\label{AB-8} $-7! A_8 = 128 (\M_8\!-\!\bar{\M}_8)$ and
\begin{align*}
-\frac{7!}{64} B_8 & = \M_8 + 3 (\M_8\!-\!\bar{\M}_8) \\
& + 18 \, [(\M_6\!-\!\bar{\M}_6) \bar{\M_2}-\M_2 (\M_6\!-\!
\bar{\M}_6)] \\
& + 36 \, [(\M_4 \!-\! \bar{\M}_4) \bar{\M}_4 - \M_4 ( \M_4 \!-\! \bar{\M}_4 )] \\
& + 6 \, [(\M_2 \!-\! \bar{\M}_2) \bar{\M}_6 - \M_6 ( \M_2 \!-\!
\bar{\M}_2)].
\end{align*}
In particular, the restriction properties $\M_{2N} = \bar{\M}_{2N}$
for $N=2,3,4$ imply the commutator relation for $\M_8$.
\end{ex}

It is natural and will be convenient later on to reformulate the
commutator relations \eqref{CR} in terms of generating functions.
Let
$$
\H(s) \quad \mbox{and} \quad \bar{\H}(s)
$$
be the respective generating functions \eqref{def-H} of the
sequences
$$
\M_2,\M_4,\cdots \quad \mbox{and} \quad
\bar{\M}_2,\bar{\M}_4,\cdots.
$$
Then the system of relations \eqref{CR} is equivalent to the
identity
\begin{equation}\label{CR-G}
\frac{1}{s} \frac{\partial \H}{\partial s}(s)i^* = \frac{1}{2} \M_4
i^* + i^* [\bar{P}_2,\bar{\H}(s)] - [P_2,\H(s)] i^*.
\end{equation}
Indeed, the equivalence of \eqref{CR} and \eqref{CR-G} follows from
the calculation
\begin{align*}
& \frac{1}{s} \frac{\partial \H}{\partial s} (s) i^* = \frac{1}{2}
\sum_{N \ge 2} \M_{2N} i^* \frac{1}{(N\!-\!2)!(N\!-\!1)!}
\left( \frac{s^2}{4} \right)^{N-2} \\
& = \frac{1}{2} \M_4 i^* + \sum_{N \ge 3} \left( i^*
[\bar{P}_2,\bar{\M}_{2N-2}] - [P_2,\M_{2N-2}] i^* \right)
\frac{(N\!-\!1)}{(N\!-\!2)!(N\!-\!1)!} \left( \frac{s^2}{4}
\right)^{N-2} \\
& = \frac{1}{2} \M_4 i^* + \left( i^* [\bar{P}_2,\bar{\H}(s)] -
[P_2,\H(s)] i^* \right).
\end{align*}

A first consequence of the commutator relations \eqref{CR} is that
the operators $\M_{2N}$ are only of order two. This is obvious for
$\M_2$ and easy to see for $\M_4$. The general case follows by
induction. In fact, if we assume that for any metric $\M_{2N}$ is a
second-order operator, then $\bar{\M}_{2N}$ is second-order, too. It
follows that the commutators on the right-hand side of the relation
$$
\M_{2N+2} i^* = 2N ( i^* [\bar{P}_2, \bar{\M}_{2N}] -
[P_2,\M_{2N}]i^*)
$$
are of third order. Moreover, this relation also states that its
right-hand side is a tangential operator. But, by the
self-adjointness of $\M_{2N+2}$, this tangential operator is only of
order two. This completes the induction. In the following section,
we shall substantially refine this argument.

%%%%%%%%%%%%%%%%%%%%%%%%%%%%%%%%%%%%%%%%%%%%%%%%%%%%%%%%%%%%%%%%%%%%%%%%%%
\section{The principal part of $\M_{2N}$}\label{main-c}

In this section we prove the formula for the principal part of the
operators $\M_{2N}$ stated in Theorem \ref{main-P}. The main tool in
the proof will be Theorem \ref{base-CR}.

%%%%%%%%%%%%%%%%%%%%%%%%%%%%%%%%%%%%%%%%%%%%%%%%%%%%%%%%%%%%%%%%%%%%%%%%%
\subsection{Description of the principal part}\label{principal-part}

The following result restates a part of Theorem \ref{main-P}.

\begin{thm}\label{base-c} On any Riemannian manifold $(M,g)$ of
dimension $n \ge 3$,
\begin{equation}\label{principal-M}
\sum_{N \ge 1} \M_{2N}^0(g) \frac{1}{(N\!-\!1)!^2}
\left(\frac{r^2}{4}\right)^{N-1} = - \delta (g_r^{-1} d).
\end{equation}
\end{thm}

Some comments are in order.

Here we use the following notational conventions. We define $L^0(u)
\st L(u) - L(1)u$ for any linear operator $L$ on $C^\infty(M)$,
i.e., $L^0$ is obtained by removing the zeroth-order term $L(1)$ of
$L$. We regard $g_r$ as an endomorphism on one-forms on $M$ using
$g$, and denote the adjoint of $d$ (negative divergence) with
respect to $g$ by $\delta$. Moreover, we recall that according to
the conventions in Section \ref{intro-sum}, in even dimension $n$
and for general metrics, the sum in \eqref{principal-M} only runs up
to $r^{n-2}$.

For a few special metrics, Theorem \ref{base-c} can be given a {\em
direct} proof. In fact, for round spheres $\S^n$ and even for their
conformally flat pseudo-Riemannian analogs $\S^{(p,q)}$, the
assertion follows from summation formulas established in
\cite{juhl-power} and \cite{JK}. These proofs utilize the fact that,
by the relation to representation theory, explicit formulas for
GJMS-operators are available in these cases.

In the low-order cases $N=1,2,3$, Theorem \ref{base-c} for general
metrics can be confirmed by explicit formulas. In fact, the relation
$\M_2^0 = -\delta d = \Delta$ is obvious, and\footnote{The subscript
$_{(2N)}$ indicates the coefficient of $r^{2N}$ in the Taylor
expansion.}
$$
\frac{1}{4} \M_4^0 = - \delta ((g_r)^{-1}_{(2)} d)
$$
is equivalent to
$$
\M_4^0 = - 4 \delta (\Rho d)
$$
which follows from \eqref{Y} and \eqref{Pan} by a direct
calculation. Finally, for $n \ge 6$, we have (by Section 6.12 of
\cite{juhl-book}) the explicit formula
\begin{equation}\label{P6}
\M_6^0 = - 48 \delta (\Rho^2 d) - \frac{16}{n\!-\!4} \delta (\B d)
\end{equation}
(for more details see also Section \ref{P6-recursive}). But
$$
g_r = g - \Rho r^2 + \frac{1}{4} \left(\Rho^2 \!-\!
\frac{\B}{n\!-\!4}\right) r^4 + \cdots
$$
implies
$$
g_r^{-1} = g + \Rho r^2 + \frac{1}{4} \left( 3 \Rho^2 +
\frac{\B}{n\!-\!4}\right) r^4 + \cdots,
$$
and the above formula for $\M_6$ can be written in the form
$$
\M_6^0 = - 2!2!2^4 \delta ((g_r^{-1})_{(4)} d).
$$

Theorem \ref{base-c} confirms Conjecture 11.1 in \cite{juhl-power}.

%%%%%%%%%%%%%%%%%%%%%%%%%%%%%%%%%%%%%%%%%%%%%%%%%%%%%%%%%%%%%%%%%%%%%%
\subsection{Proof of Theorem \ref{base-c}}\label{C-CR}

In this section, we prove that the commutator relations \eqref{CR}
imply the precise form of the principal part of $\M_{2N}$ claimed in
Theorem \ref{base-c}.

As a preparation, we establish a result which might be also of
independent interest.

\begin{thm}\label{double} Let $\bar{g}(r) = dr^2 + g(r)$ (with $g_r =
g(r)$). Then the metric
\begin{equation}\label{aston}
g_{++} = s^{-2} (ds^2 + g(r,s)) = s^{-2} (ds^2 + dr^2 + \sum_{k \ge
0} (s^2 \!+\! r^2)^k g_{(2k)})
\end{equation}
satisfies
$$
\Ric(g_{++}) + (n\!+\!1) g_{++} = 0 \quad \mbox{and} \quad g(r,0) =
\bar{g}(r),
$$
i.e., is a Poincar\'e-Einstein metric relative to $\bar{g}(r)$.
\end{thm}

\begin{proof} We calculate the Ricci curvature of \eqref{aston}. For
this purpose, we use polar coordinates $r = R \cos \theta$ and $s =
R \sin \theta$ and write the metric in the form
\begin{equation}\label{polar}
g_{++} = \sin^{-2} \theta \left( R^{-2} (dR^2 + \sum_{k \ge 0}
R^{2k} g_{(2k)}) + d\theta^2 \right) = \sin^{-2} \theta (g_+ +
d\theta^2)
\end{equation}
on $M \times (0,\varepsilon) \times (0,\pi)$. We recall the
conformal transformation law
\begin{multline}\label{Ricci}
\Ric (\hat{g})(X,Y) = \Ric(g)(X,Y) - (n\!-\!2) g(\nabla^g_X \grad
\varphi,Y) - g(X,Y) \Delta_g (\varphi) \\ - (n\!-\!2) |d \varphi|^2
g(X,Y) + (n\!-\!2) \langle d\varphi,X \rangle) \langle d\varphi,Y
\rangle,
\end{multline}
where $\hat{g} = e^{2\varphi} g$. For $\varphi = - \log \sin
\theta$, we have
$$
d \varphi = - \cot \theta d\theta \quad \mbox{and} \quad \Delta
\varphi = \sin^{-2}\theta.
$$
Thus, \eqref{polar} and \eqref{Ricci} imply that on the
$n+2$-dimensional space $M \times (0,\varepsilon) \times (0,\pi)$
\begin{align*}
\Ric(g_{++}) & = \Ric(g_+ \!+\! d\theta^2) - \sin^{-2} \theta \, g_+
- n \cot^2 \theta \, g_+ \\
& = \Ric(g_+) - \sin^{-2}\theta \, g_+ - n \cot^2\theta \, g_+ \\
& = -n g_+ - \sin^{-2}\theta \, g_+ - n \cot^2\theta \, g_+ \\
& = -(n\!+\!1) \sin^{-2}\theta \, g_+ \\
& = -(n\!+\!1) g_{++}
\end{align*}
on tangential vectors with a vanishing $\partial_\theta$-component.
Similarly, on tangential vectors $X=Y=\partial_\theta$, we find
\begin{align*}
\Ric(g_{++}) & = - n \partial(-\cot \theta)/\partial \theta
- \sin^{-2} \theta - n \cot^2 \theta + n \cot^2 \theta \\
& = - (n\!+\!1) \sin^{-2} \theta \\
& = -(n\!+\!1) g_{++}.
\end{align*}
The remaining components of $\Ric(g_{++})$ vanish. Thus, we obtain
$$
\Ric(g_{++}) + (n\!+\!1) g_{++} = 0.
$$
The proof is complete.
\end{proof}

\begin{rem} For odd $n$, $\bar{g}$ is uniquely determined by $g$ to any
order. Although $\bar{g}$ lives on a space of even dimension $n+1$,
Theorem \ref{double} tells that the Poincar\'e-Einstein metric
$g_{++}$ relative to $\bar{g}$ is well-defined to any
order.\footnote{An explicit formula for the diffeomorphism which
relates the Poincar\'e-Einstein metrics relative to $\bar{g}$ and
the Einstein metric $r^{-2}\bar{g}$ is given in Section 4 of
\cite{GW}.}
\end{rem}

As an immediate consequence of Theorem \ref{double}, we obtain

\begin{corr}\label{basic-symm} The coefficients $g_{(2k,2l)}$ in the
Poincar\'e-Einstein metric
$$
g_{++} = s^{-2}(ds^2 + dr^2 + \sum_{k,l \ge 0} r^{2k} s^{2l}
g_{(2k,2l)})
$$
are given by
$$
g_{(2k,2l)} = \binom{k+l}{k} g_{(2k+2l)}.
$$
In particular, they satisfy the symmetry relations
\begin{equation}\label{sy}
g_{(2k,2l)} = g_{(2l,2k)}.
\end{equation}
\end{corr}

For the following discussion, it will be convenient to introduce
some additional notational conventions. We shall write
Poincar\'e-Einstein metrics also in the form $r^{-2} (dr^2 +
E(g)(r))$ with $E(g)(0)=g$. This notation enables us to emphasize
the dependence on the metric $g$. As usual, we shall view a bilinear
form $B$ also as an endomorphism on one-forms. The latter acts on
the one-form $\omega = \sum_i \omega_i dx^i$ by
$$
\omega \mapsto \sum_{i,j} B_j^i \omega_i dx^j,
$$
where $B_j^i = \sum_k g^{ik} B_{kj}$. In particular, $E(g)(r)$ will
denote a one-parameter family of metrics with $E(g)(0)=g$ and a
one-parameter family of linear operators on one-forms with $E(g)(0)
= 1$.

Theorem \ref{double} immediately yields the following restriction
property. For similar results in a more general setting see Section
6 of \cite{AE}.

\begin{lemm}\label{PE-rest} As families of endomorphisms on
one-forms, we have
\begin{equation*}
E(dr^2\!+\!E(g)(r))(s)|_{r=0} = \begin{pmatrix} 1 & 0 \\ 0 & E(g)(s)
\end{pmatrix}.
\end{equation*}
\end{lemm}

\begin{proof} By Theorem \ref{double}, the restriction to $r=0$
of the metric
$$
E(dr^2\!+\!E(g)(r))(s)
$$
is given by
$$
dr^2 + \sum_{l \ge 0} s^{2l} g_{(2l)}.
$$
This shows that the corresponding operator (with respect to the
background metric $(dr^2\!+\!E(g)(r))|_{r=0} = dr^2 \!+\! g$) equals
$$
\begin{pmatrix} 1 & 0 \\ 0 & E(g)(s) \end{pmatrix}.
$$
The proof is complete.
\end{proof}

Next, we apply Lemma \ref{PE-rest} to prove the following technical
result.

\begin{lemm}\label{PE-d2} The following identities of endomorphisms
on one-forms are equivalent:
\begin{multline}\label{CR-GF-2}
\begin{pmatrix}
0 & 0 \\ 0 & - s^{-1}(\partial/\partial s) (E(g)(s)^{-1}) + 2
E(g)(s)^{-1} \Rho
\end{pmatrix} \\
= - (\partial^2/\partial r^2) |_{r=0} (E(dr^2\!+\!E(g)(r))(s)^{-1})
\end{multline}
and
\begin{multline}\label{CR-GF}
\begin{pmatrix}
0 & 0 \\ 0 & s^{-1} (\partial/\partial s)(E(g)(s)) + 2 \Rho E(g)(s)
\end{pmatrix} = (\partial^2/\partial r^2) |_{r=0}
(E(dr^2\!+\!E(g)(r))(s)).
\end{multline}
\end{lemm}

\begin{proof} Composing \eqref{CR-GF-2} from left
and right with $\begin{pmatrix} 1 & 0 \\ 0 & E(g)(s)\end{pmatrix}$
shows that \eqref{CR-GF-2} is equivalent to the identity
\begin{multline*}
\begin{pmatrix}
0 & 0 \\
0 & - s^{-1} E(g)(s) \circ (\partial/\partial s) (E(g)(s)^{-1})
\circ E(g)(s) + 2 \Rho E(g)(s) \end{pmatrix} \\ = -
\begin{pmatrix} 1 & 0 \\ 0 & E(g)(s) \end{pmatrix} \circ
(\partial^2/\partial r^2) |_{r=0}(E(dr^2\!+\!E(g)(r))(s)^{-1}) \circ
\begin{pmatrix} 1 & 0 \\ 0 & E(g)(s) \end{pmatrix}.
\end{multline*}
Now $E(g)(s)^{-1} E(g)(s) = 1$ implies
$$
(\partial/\partial s) (E(g)(s)^{-1}) = - E(g)(s)^{-1} \circ
(\partial/\partial s) (E(g)(s)) \circ E(g)(s)^{-1}.
$$
Similarly, we have
\begin{multline*}
(\partial^2/\partial r^2)|_{r=0} (E(dr^2\!+\!E(g)(r))(s)^{-1})
\circ E(dr^2\!+\!E(g)(r))(s)|_{r=0} \\
+ E(dr^2\!+\!E(g)(r))(s)^{-1}|_{r=0} \circ (\partial^2/\partial r^2)
|_{r=0} (E(dr^2\!+\!E(g)(r))(s)) = 0.
\end{multline*}
Combining these observations with Lemma \ref{PE-rest} proves the
assertion. \end{proof}

Next, we use Theorem \ref{double} to establish the identity
\eqref{CR-GF}.

\begin{lemm}\label{base-ident} The identity \eqref{CR-GF} holds true.
\end{lemm}

\begin{proof} Obviously, it suffices to prove that
\begin{multline}\label{deq}
\frac{1}{s} \frac{\partial}{\partial s} \left( g^{ia} \left( \sum_{k
\ge 0} s^{2k} g_{(2k)} \right)_{aj}  \right) + 2 \Rho^i_b g^{bc}
\left( \sum_{k \ge 0} s^{2k} g_{(2k)} \right)_{cj} \\
= \frac{\partial^2}{\partial r^2}\big|_{r=0} \left( \left(\sum_{k\ge
0} r^{2k} g_{(2k)} \right)^{-1}_{ia} \left( \sum_{l,m\ge 0} r^{2l}
s^{2m} g_{(2l,2m)}\right)_{aj} \right).
\end{multline}
In order to prove \eqref{deq}, we compare the coefficients of
$s^{2k-2}$ for $k \ge 1$ on both sides. On the left-hand side, this
coefficient is given by
\begin{equation}\label{left-deq}
2k g^{ia} (g_{(2k)})_{aj} + 2 \Rho^{ic} (g_{(2k-2)})_{cj}.
\end{equation}
Since the right-hand side of \eqref{deq} equals
$$
\frac{\partial^2}{\partial r^2}\big|_{r=0} \left( \left( g + r^2
g_{(2)} + \cdots \right)^{-1}_{ia} \right) \sum_{m \ge 0} s^{2m}
(g_{(0,2m)})_{aj} + 2 g^{ia} \sum_{m \ge 0} s^{2m}
(g_{(2,2m)})_{aj},
$$
it follows that $s^{2k-2}$ contributes with the coefficient
\begin{equation}\label{right-deq}
-2 g^{ia}_{(2)} (g_{(2k-2)})_{aj} + 2 g^{ia} (g_{(2,2k-2)})_{aj}.
\end{equation}
But $g_{(2)} = -\Rho$ and
$$
g_{(2,2k-2)} = k g_{(2k)} \quad \mbox{(by Theorem \ref{double})}
$$
show that \eqref{left-deq} and \eqref{right-deq} coincide.
\end{proof}

Now we are ready to prove the following result.

\begin{thm}\label{reduce} The relation \eqref{CR-G} implies Theorem \ref{base-c}.
\end{thm}

\begin{proof} We start by determining the principal part
of the right-hand side of \eqref{CR-G} using Theorem \ref{base-c}.
In fact, the principal part of the right-hand side of \eqref{CR-G}
coincides with the principal part of
$$
-2 \delta (\Rho d) s^0 + i^* [\Delta_{\bar{g}},\bar{\H}^0(s)] -
[\Delta_g,\H^0(s)] i^*.
$$
Now we have
\begin{equation}\label{Laplace-bar}
\Delta_{\bar{g}} = \frac{\partial^2}{\partial r^2} + \frac{1}{2} \tr
\left(\frac{\partial g_r/\partial r}{g_r}\right)
\frac{\partial}{\partial r} + \Delta_{g_r}
\end{equation}
and
\begin{align*}
\bar{\H}^0(s) & = - \delta_{dr^2+g_r} (E(dr^2\!+\!E(g)(r))(s)^{-1}d)
\end{align*}
by Theorem \ref{base-c}. Lemma \ref{PE-rest} shows that
$$
i^* \delta_{dr^2+g_r} (E(dr^2\!+\!E(g)(r))(s)^{-1}d) = \delta_g
(E(g)(s)^{-1} d) i^*.
$$
Hence using $i^* \Delta_{g_r} = \Delta_g i^*$ it follows that
$$
i^* [\Delta_{g_r},\bar{\H}^0(s)] - [\Delta_g,\H^0(s)]i^*
$$
vanishes. Thus, it suffices to determine the principal part of
$$
i^* [(\partial^2/\partial r^2),
-\delta_{dr^2+g_r}(E(dr^2\!+\!E(g)(r))(s)^{-1}d)].
$$
Since partial derivatives commute, it can be written in the form
$$
-\delta_g (T(s) d)
$$
with
$$
T(s) = \begin{pmatrix} 0 & 0 \\ 0 & 2 E(g)(s)^{-1} \Rho
\end{pmatrix} + (\partial^2/\partial r^2) |_{r=0} (E(dr^2\!+\!E(g)(r))(s)^{-1}).
$$
But by Lemma \ref{PE-d2} and Lemma \ref{base-ident}, the operator
$T(s)$ coincides with
$$
\begin{pmatrix}
0 & 0 \\
0 & s^{-1} (\partial/\partial s) (E(g)(s)^{-1})
\end{pmatrix}.
$$
Now we use induction on $N$. Let $E_{2N}(g)(s)^{-1}$ denote the
coefficient of $s^{2N}$ in the Taylor expansion of $E(g)(s)^{-1}$.
Assume that the relation
$$
\M_{2N-2}^0(g) \frac{1}{(N\!-\!2)!^2} \frac{1}{2^{2N-4}} = -
\delta_g( E_{2N-4}(g)(s)^{-1} d), \; N \ge 3
$$
has been proved for all metrics $g$. Then the above arguments also
show that the principal part of
$$
(i^* [\bar{P}_2,\bar{\M}_{2N-2}] - [P_2,\M_{2N-2}] i^*)
\frac{1}{(N\!-\!2)!^2} \left( \frac{s^2}{4} \right)^{N-2}, \; N \ge
3
$$
is given by
$$
-(2N\!-\!2) \delta_g (E_{2N-2}(g)(s)^{-1} d) s^{2N-4}.
$$
Now \eqref{CR} implies that the operator
$$
\frac{1}{(N\!-\!1)!^2} \frac{1}{2^{2N-2}} \M_{2N}(g)
$$
has the principal part
\begin{multline*}
- \frac{1}{(N\!-\!1)!^2} \frac{1}{2^{2N-2}} (2N\!-\!2) (N\!-\!2)!^2
2^{2N-4} (2N\!-\!2) \, \delta_g (E_{2N-2}(g)(s)^{-1} d) \\ =
-\delta_g (E_{2N-2}(g)(s)^{-1} d).
\end{multline*}
This completes the induction. \end{proof}

Finally, we note that for locally conformally flat metrics, the
above discussion can be made completely explicit. In fact, in that
case, we have
\begin{equation}\label{c-flat}
E(s) = (1- s^2/2 \Rho)^2: \Omega^{1}(M) \to \Omega^1(M)
\end{equation}
and
\begin{equation}\label{CF-double}
E(dr^2\!+\!E(g)(r))(s) = \begin{pmatrix} 1 & 0 \\ 0 & (1-s^2/2 \Rho
(1-r^2/2 \Rho)^{-1})^2 \end{pmatrix}.
\end{equation}
In order to prove \eqref{CF-double}, we apply \eqref{c-flat} to the
metric $\bar{g}$. In fact, we recall that $\bar{g}$ is locally
conformally flat, too. We find
$$
E(dr^2\!+\!E(g)(r))(s) = (1-s^2/2 \bar{\Rho})^2.
$$
But the relation
\begin{equation}\label{H2}
\bar{\Rho} = \Rho(dr^2\!+\!g_r) = - \frac{1}{2r} (\partial/\partial
r)(g_r)
\end{equation}
(see \cite{juhl-book}, Lemma 6.11.2 and Appendix \ref{more-double},
Lemma \ref{schouten-bar}) implies
$$
\bar{\Rho} = \Rho (1-r^2/2 \Rho)^{-1}.
$$
This yields \eqref{CF-double}. Now by \eqref{c-flat}, the left-hand
side of \eqref{CR-GF} equals
$$
-2 \Rho(1-s^2/2\Rho) + 2 \Rho (1-s^2/2 \Rho)^2 = -s^2 \Rho^2
(1-s^2/2 \Rho).
$$
But an easy calculation using \eqref{CF-double} shows that this
result coincides with the right-hand side of \eqref{CR-GF}.

%%%%%%%%%%%%%%%%%%%%%%%%%%%%%%%%%%%%%%%%%%%%%%%%%%%%%%%%%%%%%%%%%%%%%%%%
\section{The zeroth-order term of $\M_{2N}$}\label{structure}

The inversion formula in Theorem \ref{duality} proves the first part
of Theorem \ref{main-P}. In the present section, we complete the
proof of the second part of Theorem \ref{main-P}.

First, we recall the definition \eqref{def-H} of the generating
function $\H(r)$ and restate the second part of Theorem \ref{main-P}
in the following form.

\begin{thm}\label{GF} We have,
\begin{equation}\label{GF-mu}
\H(g)(r) = -\delta (g_r^{-1} d) - \frac{\left(\partial^2/\partial
r^2 - (n\!-\!1) r^{-1} \partial/\partial r - \delta (g_r^{-1}d)
\right) (w(r))}{w(r)}.
\end{equation}
\end{thm}

The equivalence of the second part of Theorem \ref{main-P} and
Theorem \ref{GF} is a consequence of the following result.

\begin{lemm}\label{ct-simpel}
\begin{multline}\label{ct-simplify}
(\partial^2/\partial r^2 - (n\!-\!1) r^{-1}
\partial/\partial r - \delta(g_r^{-1}d))(w(r))/w(r) \\
= r^{-2} \left( \Delta_{g_+}(\log w) - |d \log w|^2_{g_+} \right).
\end{multline}
\end{lemm}

\begin{proof} On the one hand, the formulas
\begin{equation}\label{laplace}
\Delta_{g_r} = \frac{1}{\sqrt{\det g_r}} \sum_{j,k}
\frac{\partial}{\partial x_j}\left( \sqrt{\det g_r} (g_r^{-1})_{jk}
\frac{\partial }{\partial x_k}\right)
\end{equation}
and
\begin{equation}\label{principal}
-\delta_g (g_r^{-1} d) = \frac{1}{\sqrt{\det g}} \sum_{j,k}
\frac{\partial}{\partial x_j}\left( \sqrt{\det g} (g_r^{-1})_{jk}
\frac{\partial }{\partial x_k}\right)
\end{equation}
imply that
$$
\Delta_{g_r} = - \delta_g (g_r^{-1} d) + g_r(d \log v,d).
$$
Here $v$ is regarded as a function on $M$. On the other hand, a
straightforward calculation shows that the Laplacian of the metric
$g_+ = r^{-2}(dr^2 + g_r)$ is given by
$$
\Delta_{g_+} = r^2 \frac{\partial^2}{\partial r^2} - (n\!-\!1)r
\frac{\partial }{\partial r} + r^2 \frac{\partial}{\partial r} \log
v(r) \frac{\partial}{\partial r} + r^2 \Delta_{g_r}.
$$
Hence
$$
\frac{\partial^2 w}{\partial r^2} - \frac{n\!-\!1}{r} \frac{\partial
w}{\partial r} - \delta_g (g_r^{-1} d w)  = \frac{1}{r^2} \left(
\Delta_{g_+} (w) - g_+(d \log v,d w) \right),
$$
where $v$ and $w$ are regarded as functions on $X$. Therefore, we
find that the left-hand side of \eqref{ct-simplify} equals
$$
\frac{1}{r^2} \left(\frac{\Delta_{g_+}(w)}{w} - 2 \left| \frac{d
w}{w} \right|^2_{g_+}\right).
$$
Now a simple calculation shows that
$$
\Delta_{g_+}(\log w) = \frac{\Delta_{g_+}(w)}{w}-|d\log w|_{g_+}^2.
$$
This proves the assertion. \end{proof}

Theorem \ref{GF} contains the following two claims.
\begin{itemize}
\item [(1)] The principal part of $\H(r)$ is given by the self-adjoint
operator
$$
- \delta (g_r^{-1} d).
$$
This assertion was established in Theorem \ref{base-c}.
\item [(2)] The identity
\begin{equation}\label{key}
\left(\partial^2/\partial r^2 - \frac{n\!-\!1}{r} \,
\partial/\partial r + \H(r) \right)(w(r)) = 0.
\end{equation}
\end{itemize}

Conversely, by the formal self-adjointness of all $\M_{2N}$, these
two properties suffice to prove Theorem \ref{GF}. In fact, the
property (1) shows that the operator $\H(r) + \delta(g_r^{-1} d)$ is
self-adjoint (here $r$ is regarded as a parameter) of first-order.
Hence this operator is only scalar, i.e.,
$$
\H(r) = - \delta(g_r^{-1}d) + \kappa
$$
for some $\kappa \in C^\infty(M)$. But $\kappa$ is determined by the
action of $\H(r)$ on the function $w(r)$ using \eqref{key}. This
yields the assertion \eqref{GF-mu}. Thus, using Theorem
\ref{base-c}, it only remains to prove the relation \eqref{key}.

Before we present the proof of \eqref{key}, we discuss two important
special cases of Theorem \ref{GF}.

\begin{ex}\label{sphere} Let $M=\S^n$. Then
$$
w(r) = (1-r^2/4)^\f.
$$
In particular, $w(r)$ is constant on $\S^n$. A calculation shows
that
$$
\left(\frac{\partial^2}{\partial r^2} - (n\!-\!1) r^{-1}
\frac{\partial}{\partial r}\right)(w(r)) = \f \left(\f\!-\!1\right)
(1-r^2/4)^{\f-2}.
$$
Thus,
$$
\H(r)(1) = -\f \left(\f-1\right) (1-r^2/4)^{-2}.
$$
The latter formula implies
$$
\mu_{2N} = - \f \left(\f-1\right) N!(N-1)!.
$$
This fits with the summation formula
\begin{equation}\label{sphere-sum}
\M_{2N} = N!(N\!-\!1)! P_2
\end{equation}
(proved in \cite{juhl-power}, Theorem 6.1).
\end{ex}

The following example leaves the framework of Riemannian metrics. It
shows, in a special case, that the literal extension of Theorem
\ref{GF} remains valid for metrics with general signature (see the
comments in Subsection \ref{sign}).

\begin{ex}\label{pseudo} More generally, let $M=\S^{q,p}$ be the
conformally flat pseudo-spheres
$$
\S^{(q,p)} = \S^q \times \S^p, \; p \ge 1, \; q \ge 1
$$
with the metrics $g_{\S^q} - g_{\S^p}$ given by the round metrics on
the factors. Then
$$
w(r) = (1-r^2/4)^{q/2} (1+r^2/4)^{p/2}.
$$
In particular, $w(r)$ is constant on $\S^{q,p}$. Now a calculation
shows that
\begin{multline*}
\left(\frac{\partial^2}{\partial r^2} - (q\!+\!p\!-\!1) \frac{1}{r}
\frac{\partial}{\partial r}\right)(w(r))/w(r) \\ = q/2
\left(q/2\!-\!1\right) (1-r^2/4)^{-2} - p/2 \left(p/2\!-\!1\right)
(1+r^2/4)^{-2}.
\end{multline*}
The resulting formula for $\H(r)(1)$ is equivalent to
$$
-\tr \left( \begin{pmatrix} q/2-1 & 0 \\ 0 & p/2-1\end{pmatrix} \Rho
(1\!-\!r^2/2\Rho)^{-2}\right)
$$
(see Corollary 7.2 in \cite{juhl-power}). The latter result is
confirmed by the summation formula for $\M_{2N}$ established in
\cite{JK}.
\end{ex}

In the following subsection, we prove \eqref{key} for general
metrics subject to the usual conditions for the parameters $n$ and
$N$. The case of general $N \ge 1$ for locally conformally flat
metrics in even dimensions will be discussed in Subsection
\ref{proof-2}.

%%%%%%%%%%%%%%%%%%%%%%%%%%%%%%%%%%%%%%%%%%%%%%%%%%%%%%%%%%%%%%%%%%%%%
\subsection{General metrics}\label{proof-1}

The following proof of Theorem \ref{GF}, i.e., of the relation
\eqref{key}, rests on the fact that the zeroth-order terms $\mu_{2N}
= \M_{2N}(1)$ are determined by recursive formulas in terms of
compositions of lower-order operators $\M_{2M}$ acting on
lower-order $\mu_{2M}$. In the critical case $2N=n$, such a
recursive formulas follows from the identity
$$
P_n = \sum_{|I|=\f} n_I \M_{2I}.
$$
In fact, $P_n (1) = 0$ implies the recursive formula
$$
0 = \sum_{|I|=\f} n_I \M_{2I}(1) = \mu_n + \sum_{|I| + a = \f, \; a
\ne \f} n_{(I,a)} \M_{2I} (\mu_{2a})
$$
for the curvature quantity $\mu_n$. These recursive formulas admit
generalizations to general dimensions. The first three low-order
examples will also be useful for later explicit calculations.

\begin{ex}\label{E1} For $n \ge 2$,
$$
\mu_2 = 4 \left(\f-1\right) w_2.
$$
\end{ex}

\begin{ex}\label{E2} For $n \ge 3$,
\begin{equation}\label{E2-f}
\left(\f-1\right) \mu_4 + \M_2 (\mu_2) = 2! 2^4 \left(\f-1\right)
\left(\f-2\right) w_4.
\end{equation}
\end{ex}

\begin{proof} The definition $\M_4 = P_4 - P_2^2$ gives
$$
\mu_4 = \left(\f\!-\!2\right) Q_4 + \left(\f-1\right) P_2(Q_2).
$$
Hence the universal recursive formula
$$
Q_4 = -P_2(Q_2) + 32 w_4
$$
(see Example \ref{Q-4}) implies
\begin{align*}
\mu_4 & = -\left(\f-2\right) P_2(Q_2) + 32 \left(\f-2\right) w_4 +
\left(\f-1\right) P_2(Q_2) \\
& = P_2(Q_2) + 32 \left(\f-2\right) w_4.
\end{align*}
It follows that
$$
\left(\f-1\right) \mu_4 + \M_2 (\mu_2) = 32 \left(\f-1\right)
\left(\f-2\right) w_4
$$
using
$$
\mu_2 = - \left(\f-1\right) Q_2.
$$
The proof is complete. \end{proof}

But Example \ref{E1} shows that Eq.~\eqref{E2-f} is equivalent to
\begin{equation}\label{mu4-case}
\M_4(1) + 4 \M_2 (w_2) = 32 \left(\f-2\right) w_4.
\end{equation}

\begin{ex}\label{E3} For even $n \ge 6$ and odd $n \ge 3$,
\begin{multline}\label{mu-6}
\frac{1}{2} \left(\f-1\right)\left(\f-2\right) \mu_6 +
\left(\f-1\right) \M_2 (\mu_4) \\ + 2 \left(\f-2\right) \M_4 (\mu_2)
+ \M_2^2(\mu_2) \\ = 3! 2^6
\left(\f-1\right)\left(\f-2\right)\left(\f-3\right) w_6.
\end{multline}
\end{ex}

\begin{proof} First, we use the universal recursive formula
\begin{equation}\label{Q6}
Q_6 = -2P_2(Q_4) + 2 P_4(Q_2) - 3 P_2^2(Q_2) - 3!2!2^6 w_6
\end{equation}
to prove that
\begin{equation}\label{mu6}
\mu_6 = -2 P_2 (Q_4) + 4P_4 (Q_2) - 6P_2^2 (Q_2) + \left(\f-3\right)
3!2!2^6 w_6.
\end{equation}
In fact, the definition $\M_6 = P_6 - 2P_2P_4 - 2P_4P_2 + 3P_2^3$
gives
$$
\mu_6 = -\left(\f-3\right) Q_6 - 2 \left(\f-2\right) P_2(Q_4) + 2
\left(\f-1\right) P_4 (Q_2) - 3 \left(\f-1\right) P_2^2 (Q_2).
$$
Combining this relation with \eqref{Q6} yields \eqref{mu6}. Now
applying the definition $\M_4 = P_4 - P_2^2$ for $u=1$ shows that
$$
\mu_4 = \left(\f-2\right) Q_4 + \left(\f-1\right) P_2(Q_2).
$$
It follows that the left-hand side of \eqref{mu-6} equals the sum of
$$
3! 2^6 \left(\f-1\right)\left(\f-2\right)\left(\f-3\right) w_6
$$
and
\begin{multline*}
\frac{1}{2} \left(\f-1\right)\left(\f-2\right) (-2 P_2 (Q_4) + 4
P_4(Q_2) - 6 P_2^2 (Q_2)) \\ - 2 \left(\f - 2\right) \left(\f -
1\right) (P_4 - P_2^2) (Q_2) \\ + \left(\f - 1\right) P_2
\left(\left(\f-2\right) Q_4 + \left(\f-1\right) P_2 (Q_2)\right) -
\left(\f - 1\right) P_2^2 (Q_2).
\end{multline*}
It is easy to verify that the latter sum vanishes. The proof is
complete. \end{proof}

By Examples \ref{E1}--\ref{E2}, the identity \eqref{mu6} is
equivalent to
\begin{equation}\label{mu6-case}
\M_6(1) + 4^2 \M_4 (w_2) + 4^2 2^2 \M_2 (w_4) = 3! 2! 2^6
\left(\f-3\right) w_6.
\end{equation}

The proof of the formula in Example \ref{E3} combines a special case
of the universal recursive formula for $Q$-curvature \eqref{URQ}
with only {\em algebraic} calculations, i.e., no further properties
of GJMS-operators and $Q$-curvatures were used. Next, we shall apply
a similar method to prove the following general result.

\begin{thm}\label{basic} For $1 \le N \le \f$ (if $n$ is even) and
all $N \ge 1$ (if $n$ is odd),
\begin{multline}\label{sum}
\sum_{|I|=N} n_I \frac{(N-(I_2+\cdots+I_r)) \cdots
(N-I_r)}{(\f-(I_2+\cdots+I_r)) \cdots (\f-I_r)} \M_{2I}(1) \\[-3mm] =
\left(\f\!-\!N\right) (N\!-\!1)! N! 2^{2N} w_{2N}.
\end{multline}
\end{thm}

For even $n$ and $N = \f$, the relation \eqref{sum} is to be
understood as
$$
\sum_{|I|=N} n_I \M_{2I}(1) = 0.
$$
This identity is an immediate consequence of the inversion formula
(Theorem \ref{duality}) and $P_n (1) = 0$.

\begin{proof} We write any quantity $\M_{2I}(1)$ on the left-hand
side as a linear combination of GJMS-operators acting on
$Q$-curvatures by using the definition \eqref{M-def}. In particular,
we have the contribution
$$
\M_{2N}(1) = \sum_{|I|=N} m_I P_{2I}(1).
$$
Among other terms, this sum contains the term
$$
P_{2N}(1) = (-1)^N \left(\f\!-\!N\right) Q_{2N}.
$$
We apply the universal recursive formula \eqref{URQ} for
$Q$-curvature to this term and simplify the resulting formula for
$\M_{2N}(1)$. We find
\begin{equation}\label{first}
\M_{2N}(1) = D_{2N} + \left(\f\!-\!N\right) N!(N\!-\!1)! 2^{2N}
w_{2N},
\end{equation}
where
$$
D_{2N} \st \sum_{|I|+a=N} m_{(I,a)} (-1)^a (N\!-\!a) P_{2I}(Q_{2a}).
$$
The relation \eqref{first} generalizes the obvious relation
$$
\M_n(1) = \sum_{|I|+a=\f} m_{(I,a)} (-1)^a \left(\f\!-\!a\right)
P_{2I}(Q_{2a}).
$$
The formula \eqref{first} yields the right-hand side of \eqref{sum}
together with additional terms given by GJMS-operators acting on
$Q$-curvatures. It remains to prove that, in the sum on the
left-hand side of \eqref{sum}, all these terms vanish. For this
purpose, we define the polynomial
\begin{equation}\label{S-2N}
S_{2N}(T) \st \prod_{k=1}^{N-1} (T\!-\!k) \sum_{|I|=N} n_I
\frac{(N-(I_2+\cdots+I_r)) \cdots (N-I_r)}{(T-(I_2+\cdots+I_r))
\cdots (T-I_r)} \M_{2I}(1)
\end{equation}
of degree $N\!-\!1$. Then we have
$$
S_{2N}\left(\f\right) = \prod_{k=1}^{N-1} \left(\f\!-\!k\right)
\times \mbox{left-hand side of \eqref{sum}.}
$$
We also let $\tilde{S}_{2N}(T)$ be the polynomial of degree $N$
which arises from $S_{2N}(T)$ by replacing in each term the most
right factor $P_{2M}(1)$ by $(-1)^M (T\!-\!M) Q_{2M}$. Then
$\tilde{S}_{2N}(\f) = S_{2N}(\f)$. In these terms, Eq.~\eqref{first}
implies that
\begin{multline*}
S_{2N}\left(\f\right) = N!(N\!-\!1)!2^{2N} \prod_{k=1}^N
\left(\f\!-\!k\right) w_{2N} \\ + \prod_{k=1}^{N-1}
\left(\f\!-\!k\right) (D_{2N} - \M_{2N}(1)) +
\tilde{S}_{2N}\left(\f\right),
\end{multline*}
i.e.,
\begin{multline}\label{sum-simp}
S_{2N}\left(\f\right) = N!(N\!-\!1)!2^{2N} \prod_{k=1}^N
\left(\f\!-\!k\right) w_{2N} \\
+ \left(\prod_{k=1}^{N-1} \left(T\!-\!k\right) (D_{2N} -
\M_{2N}(T,1)) + \tilde{S}_{2N}(T) \right) \Big|_{T=\f},
\end{multline}
where
\begin{equation}\label{M-pol-gen}
\M_{2N}(T,1) \st \sum_{|I|+a=N} m_{(I,a)} (-1)^a (T\!-\!a)
P_{2I}(Q_{2a}).
\end{equation}
Now we consider the leading coefficient of the degree $N$ polynomial
\begin{equation}\label{poly-red}
T \mapsto \prod_{k=1}^{N-1} \left(T\!-\!k\right) (D_{2N} -
\M_{2N}(T,1)) + \tilde{S}_{2N}(T).
\end{equation}
By \eqref{S-2N} and \eqref{M-pol-gen}, this coefficient equals
$$
-\sum_{|I|+a=N} m_{(I,a)} (-1)^a P_{2I}(Q_{2a}) + \sum_{|I|+a=N}
m_{(I,a)} (-1)^a P_{2I}(Q_{2a}) = 0;
$$
note that the sum $\tilde{S}_{2N}(T)$ contributes to this
coefficient only through the trivial composition $I=(N)$, i.e.,
through $\M_{2N}(1)$. Thus, the polynomial \eqref{poly-red} has
degree $N\!-\!1$. By definition of $D_{2N}$, we have
\begin{equation}\label{zero-N}
D_{2N} - \M_{2N}(N,1) = 0.
\end{equation}
Combining this property with
\begin{equation}\label{zero-all}
\tilde{S}_{2N}(1) = \cdots = \tilde{S}_{2N}(N) = 0,
\end{equation}
we conclude that the polynomial \eqref{poly-red} has zeros at
$T=1,\dots,N$. Thus, it vanishes identically. In particular, it
vanishes at $T=\f$. Hence \eqref{sum-simp} simplifies to
$$
S_{2N}\left(\f\right) = N!(N\!-\!1)!2^{2N} \prod_{k=1}^N
\left(\f\!-\!k\right) w_{2N}.
$$
Therefore, in order to complete the proof, it only remains to prove
\eqref{zero-all}. We shall establish \eqref{zero-all} by induction
on $N$. The main observation is that the polynomials $S_{2N}(T)$
satisfy the factorization relations
\begin{equation}\label{factor}
S_{2N}(N\!-\!a) = \left( \sum_{|J|=a} b_J \, n_{(J,N-a)} \, \M_{2J}
\right) S_{2N-2a}(N\!-\!a)
\end{equation}
for $a=1,\dots,N-1$, where
$$
b_J \st (-1)^{r-a} a! \left( \frac{J_1 (J_1+J_2) \cdots
(J_1+\cdots+J_{r-1})}{(J_2+\cdots+J_r)(J_3+\cdots+J_r)\cdots J_r}
\right)
$$
for $J = (J_1,\dots,J_r)$ with $|J|=a$.

The simplest of these relations states that
$$
S_{2N}(N\!-\!1) = (N\!-\!1) \M_2 S_{2N-2}(N\!-\!1).
$$

In order to prove \eqref{factor}, it suffices to prove that, for all
compositions $J$ and $K$ with $|J| = a$ and $|K|=N-a$, the
coefficients of $\M_{2J} \M_{2K}(1)$ on both sides of \eqref{factor}
coincide. On the left-hand side, the coefficient for
$J=(J_1,\dots,J_r)$ and $K=(K_1,\dots,K_s)$ is given by the product
of
\begin{equation*}
\left(\frac{(N\!-\!(J_2+\dots+K_s)) \cdots (N\!-\!(K_1+\cdots+K_s))
\cdots (N\!-\!K_s)}{(T\!-\!(J_2+\dots+K_s)) \cdots
(T\!-\!(K_1+\cdots+K_s)) \cdots (T\!-\!K_s)} \prod_{k=1}^{N-1}
(T\!-\!k) \right)\Big|_{T=N-a}
\end{equation*}
and $n_{(J,K)}$. By $|J|=a$ and $|K|=N-a$, this product equals
\begin{multline*}
J_1 \cdots (J_1+\cdots+J_{r-1}) \, a (a+K_1) \cdots (a+K_1+\cdots+K_{s-1}) \\
\times (-1)^{r-a} \frac {(N\!-\!a\!-\!1)!(a\!-\!1)!}
{(J_2+\cdots+J_r) \cdots J_r \, K_1 \cdots (K_1 + \cdots + K_{s-1})}
\, n_{(J,K)}.
\end{multline*}
On the other hand, on the right-hand side of \eqref{factor}, the
quantity $\M_{2J} \M_{2K}(1)$ contributes with the coefficient
$$
(-1)^{r-a} a! (N\!-\!1\!-\!a)! \left( \frac{J_1 \cdots
(J_1+\cdots+J_{r-1})} {(J_2+\cdots+J_r) \cdots J_r} \right)
n_{(J,N-a)} n_K.
$$
It follows that the equality of both expressions is equivalent to
the identity
\begin{equation}\label{c-reduced}
\left(\frac{(a+K_1) \cdots (a+K_1+\cdots+K_{s-1})}{K_1 \cdots
(K_1+\cdots+K_{s-1})} \right) n_{(J,K)} = n_{(J,N-a)} n_K.
\end{equation}
But the latter relation actually follows by a straightforward
calculation using \eqref{eq:n_I}. We omit the details. This proves
\eqref{factor}. An analogous proof yields the factorization
relations
\begin{equation}\label{tilde-factor}
\tilde{S}_{2N}(N\!-\!a) = \left( \sum_{|J|=a} b_J \, n_{(J,N-a)} \,
\M_{2J} \right) \tilde{S}_{2N-2a}(N\!-\!a)
\end{equation}
for $a=1,\dots,N\!-\!1$. Now assume that we have proved
\eqref{zero-all} for the polynomials $\tilde{S}_2, \dots,
\tilde{S}_{2N-2}$. Thus, in particular, we assume that
$$
\tilde{S}_2(1) = \cdots = \tilde{S}_{2N-2}(N\!-\!1) = 0.
$$
Then the identities \eqref{tilde-factor} (for $a=1,\dots,N\!-\!1$)
imply that
$$
\tilde{S}_{2N}(1) = \cdots = \tilde{S}_{2N}(N\!-\!1) = 0.
$$
Finally, the inversion theorem (and its proof) show that in the sum
$\tilde{S}_{2N}(N)$ the only non-trivial contribution can come from
the operator $P_{2N}$. But this contribution equals
$$
(N\!-\!1)! (-1)^N (T\!-\!N) Q_{2N} |_{T=N} = 0.
$$
Hence $\tilde{S}_{2N}(N) = 0$. This proves \eqref{zero-all} for
$\tilde{S}_{2N}$. This completes the induction and the proof.
\end{proof}

\begin{rem}\label{tilde-6} The proof of Theorem \ref{basic} utilizes
the vanishing results \eqref{zero-all}. We illustrate these
results by a direct verification of the relations
\begin{equation}\label{zero-6}
\tilde{S}_6(1) = \tilde{S}_6(2) = \tilde{S}_6(3) = 0.
\end{equation}
By \eqref{S-2N}, the quadratic polynomial $S_6(T)$ is given by
\begin{multline*}
(T\!-\!1)(T\!-\!2) \\ \times \left(\M_6 (1) + \frac{4}{T\!-\!1} \M_4
\M_2 (1) + \frac{2}{T\!-\!2} \M_2 \M_4 (1) +
\frac{2}{(T\!-\!1)(T\!-\!2)} \M_2^3(1)\right).
\end{multline*}
Now the definitions of the operators $\M_2$, $\M_4$ and $\M_6$ yield
\begin{multline*}
S_6 (T) = (T\!-\!1)(T\!-\!2) \big[ P_6(1) - 2 P_4 P_2 (1) - 2 P_2
P_4 (1) + 3 P_2^3(1) \\ + \frac{4}{T\!-\!1} (P_4 \!-\! P_2^2) P_2(1)
+ \frac{2}{T\!-\!2} P_2 (P_4 \!-\! P_2^2) (1) +
\frac{2}{(T\!-\!1)(T\!-\!2)} P_2^3(1) \big].
\end{multline*}
Thus, for the cubic polynomial $\tilde{S}_6(T)$ we obtain
\begin{multline*}
\tilde{S}_6(T) = (T\!-\!1)(T\!-\!2) \big[-(T\!-\!3) Q_6 + 2
(T\!-\!1) P_4 (Q_2) - 2 (T\!-\!2) P_2 (Q_4) - 3 (T\!-\!1) P_2^2 (Q_2) \\
- 4 (P_4 \!-\! P_2^2) (Q_2) + \frac{2}{T\!-\!2} P_2 ((T\!-\!2) Q_4
\!+\! (T\!-\!1) P_2 (Q_2)) - \frac{2}{T\!-\!2} P_2^2 (Q_2) \big].
\end{multline*}
The claim \eqref{zero-6} is a direct consequence of this formula.
\end{rem}

Now Theorem \ref{basic} yields the following important formula for
the quantities $\mu_{2N} = \M_{2N}(1)$.

\begin{thm}\label{basic-2} For $1 \le N \le \f$ (if $n$ is even) and
all $N \ge 1$ (if $n$ is odd),
\begin{equation}\label{reduced-w}
\sum_{k=0}^{N-1} \left(2^k \frac{(N\!-\!1)!}{(N\!-\!1\!-\!k)!}
\right)^2 \M_{2N-2k} (w_{2k}) = \left(\f\!-\!N\right) (N\!-\!1)! N!
2^{2N} w_{2N}.
\end{equation}
\end{thm}

\begin{proof} The left-hand side of \eqref{reduced-w} equals
$$
\sum_{a=1}^N \left(2^{N-a} \frac{(N\!-\!1)!}{(a\!-\!1)!} \right)^2
\M_{2a} (w_{2N-2a}).
$$
Now we use Theorem \ref{basic} (for $w_2,\dots,w_{2N-2}$) to rewrite
this sum as
\begin{multline}\label{start}
\M_{2N}(1) + \sum_{a=1}^{N-1} \left(2^{N-a}
\frac{(N\!-\!1)!}{(a\!-\!1)!} \right)^2
\frac{1}{(\f\!-\!(N\!-\!a)) (N\!-\!a\!-\!1)!(N\!-\!a)! 2^{2N-2a}} \\
\times \sum_{|J|=N-a} n_J
\frac{(N\!-\!a\!-\!(J_2\!+\!\cdots\!+\!J_{r-1}))\cdots(N\!-\!a\!-\!J_{r-1})}
{(\f\!-\!(J_2\!+\!\cdots\!+\!J_{r-1}))\cdots(\f\!-\!J_{r-1})}
\M_{2a} \M_{2J}(1).
\end{multline}
On the other hand, Theorem \ref{basic} (for $w_{2N}$) implies that
the right-hand side of \eqref{reduced-w} equals
\begin{multline*}
\sum_{|I|=N} n_I \frac{(N\!-\!(I_2\!+\!\cdots\!+\!I_r))\cdots
(N\!-\!I_r)} {(\f\!-\!(I_2\!+\!\cdots\!+\!I_r))\cdots(\f\!-\!I_r)}
\M_{2I}(1) \\
= \M_{2N}(1) + \sum_{a=1}^{N-1} \sum_{|J|=N-a} n_{(a,J)}
\frac{(N\!-\!(J_1\!+\!\cdots\!+\!J_{r-1})) \cdots (N\!-\!J_{r-1})}
{(\f\!-\!(J_1\!+\!\cdots\!+\!J_{r-1})) \cdots (\f\!-\!J_{r-1})}
\M_{2a} \M_{2J}(1).
\end{multline*}
Thus, the assertion follows from the relation
\begin{multline}\label{red-m}
\left( \frac{(N\!-\!1)!}{(a\!-\!1)!} \right)^2
\frac{1}{(\f\!-\!(N\!-\!a))(N\!-\!a\!-\!1)!(N\!-\!a)!} \\
\times \frac{(N\!-\!a\!-\!(J_2\!+\!\cdots\!+\!J_{r-1})) \cdots
(N\!-\!a\!-\!J_{r-1})}{(\f\!-\!(J_2\!+\!\cdots\!+\!J_{r-1})) \cdots
(\f\!-\!J_{r-1})} n_J \\
= \frac{(N\!-\!(J_1\!+\!\cdots\!+\!J_{r-1})) \cdots (N\!-\!J_{r-1})}
{(\f\!-\!(J_1\!+\!\cdots\!+\!J_{r-1})) \cdots (\f\!-\!J_{r-1})}
n_{(a,J)}
\end{multline}
for all $a=1,\dots,N\!-\!1$ and all compositions $J$ of size
$N\!-\!a$. For the proof of \eqref{red-m}, we make the identity
fully explicit by using \eqref{eq:n_I}. Indeed, for $J =
(J_1,\dots,J_{r-1})$ with $|J|= N-a$, Eq.~\eqref{red-m} is
equivalent to
\begin{multline}\label{m-final}
\left(\frac{(N\!-\!a\!-\!1)! (N\!-\!1)!}{(a\!-\!1)!} \right)^2 \,
\prod_{k=1}^{r-1} \frac{1}{(J_k\!-\!1)!^2} \\
\times \left(\frac{1}{J_1(J_1+J_2)\cdots(J_1+\cdots+J_{r-2})}\right)
\left(\frac{1}{(J_2+\cdots+J_{r-1})\cdots J_{r-1}}\right) \\
\times \frac{1}{(\f\!-\!(N\!-\!a)) (N\!-\!a\!-\!1)!(N\!-\!a)!}
\frac{(N\!-\!a\!-\!(J_2\!+\!\cdots\!+\!J_{r-1}))\cdots(N\!-\!a\!-\!J_{r-1})}
{(\f\!-\!(J_2\!+\!\cdots\!+\!J_{r-1}))\cdots(\f\!-\!J_{r-1})} \\
= (N\!-\!1)!^2 \frac{1}{(a\!-\!1)!^2} \, \prod_{k=1}^{r-1}
\frac{1}{(J_k\!-\!1)!^2} \\
\times \frac{1}{a(a+J_1)\cdots(a+(J_1+\cdots+J_{r-2}))}
\frac{1}{(J_1+\cdots+J_{r-1})(J_2+\cdots+J_{r-1})\cdots J_{r-1}} \\
\times \frac{(N\!-\!(J_1\!+\!\cdots\!+\!J_{r-1})) \cdots
(N\!-\!J_{r-1})} {(\f\!-\!(J_1\!+\!\cdots\!+\!J_{r-1}))
(\f\!-\!(J_2\!+\!\cdots\!+\!J_{r-1})) \cdots (\f\!-\!J_{r-1})}.
\end{multline}
Now using $\f\!-\!(J_1\!+\!\cdots\!+\!J_{r-1}) = \f\!-\!(N\!-\!a)$,
$$
\frac{(N\!-\!a\!-\!(J_2\!+\!\cdots\!+\!J_{r-1}))\cdots(N\!-\!a\!-\!J_{r-1})}
{J_1(J_1+J_2)\cdots(J_1+\cdots+J_{r-2})} = 1
$$
and
$$
\frac{(N\!-\!(J_1\!+\!\cdots\!+\!J_{r-1})) \cdots (N\!-\!J_{r-1})}
{a(a+J_1)\cdots(a+(J_1+\cdots+J_{r-2}))} = 1,
$$
it is easy to verify \eqref{m-final}. The proof is complete.
\end{proof}

Now we are ready to prove Theorem \ref{GF}. \label{basic-argument}

\begin{proof}[Proof of Theorem \ref{GF}] By definition of $\H(r)$,
we have
\begin{align*}
\H(r)(w(r)) & = \left( \sum_{N \ge 0} \M_{2N+2} \frac{1}{N!^2}
\left( \frac{r^2}{4} \right)^N  \right) \left( \sum_{M \ge 0}
w_{2M} r^{2M} \right) \\
& = \sum_{N,M \ge 0}  2^{2M} \frac{1}{N!^2} \M_{2N+2}(w_{2M})
\left(\frac{r^2}{4}\right)^{N+M} \\
& = \sum_{T \ge 1} \left( \sum_{S=0}^{T-1} 2^{2S}
\frac{1}{(T\!-\!1\!-\!S)!^2} \M_{2T-2S}(w_{2S}) \right)
\left(\frac{r^2}{4}\right)^{T-1} \\
& = \sum_{T \ge 1} 2T (n\!-\!2T) \, w_{2T} \, r^{2T-2}.
\end{align*}
In the last step we have used Theorem \ref{basic-2}. But the latter
sum equals
$$
\left(-\frac{\partial^2}{\partial r^2} + (n\!-\!1) \frac{1}{r}
\frac{\partial}{\partial r}\right) (w(r)).
$$
This proves \eqref{key}. \end{proof}

\begin{rem} For the round sphere $\S^n$, the summation formula
(Theorem 6.1 in \cite{juhl-power})
$$
\M_{2N} = N!(N\!-\!1)! P_2
$$
and
$$
w(r) = \left(1\!-\!\frac{r^2}{4}\right)^\f
$$
show that Theorem \ref{basic-2} is equivalent to
\begin{multline}\label{sum-round}
-\sum_{k=0}^{N-1} \left(2^k \binom{N\!-\!1}{k} k!\right)^2
\f\left(\f\!-\!1\right) (N\!-\!k)!(N\!-\!k\!-\!1)! \\
\times ((-1)^k \binom{\f}{k} 2^{-2k}) = (-1)^N
\left(\f\!-\!N\right)(N\!-\!1)!N! \binom{\f}{N}.
\end{multline}
The latter relation is a consequence of the well-known summation
formula
$$
\sum_{k=0}^{N-1} \binom{\f}{k} (-1)^k = (-1)^{N-1} \binom{\f-1}{N-1}
$$
(\cite{cm}, Eq.~(5.16)). Similarly, for an Einstein metric $g$ with
$\Ric(g) = \lambda g$, we have
$$
\Rho = \frac{\lambda}{2(n\!-\!1)} g = \frac{c}{2} g \quad
\mbox{with} \quad c = \frac{\scal}{n(n\!-\!1)}
$$
and
$$
w(r) = \det \left(1\!-\!\frac{r^2}{2} \Rho\right) =
\left(1\!-\!\frac{r^2}{4} c\right)^\f.
$$
Moreover, by a rescaling argument, we find
$$
\M_{2N} = N! (N\!-\!1)! c^{N-1} P_2,
$$
and it follows that the identity \eqref{sum-round} also confirms
Theorem \ref{basic-2} for Einstein metrics.
\end{rem}

For locally conformally flat metrics, Theorem \ref{main-P} extends
to super-critical GJMS-operators in even dimensions. For the proof
of this claim it would be enough to extend the relations
\eqref{reduced-w}. However, the above proof of \eqref{reduced-w}
does not extend since the formulation of Theorem \ref{basic} does
not make sense if $2N > n$. Therefore, in the following subsection,
we shall provide an alternative proof of \eqref{reduced-w} for
locally conformally flat metrics which extends to $2N > n$.

%%%%%%%%%%%%%%%%%%%%%%%%%%%%%%%%%%%%%%%%%%%%%%%%%%%%%%%%%%%%%%%%%
\subsection{Conformally flat metrics}\label{proof-2}

In the present section, we extend Theorem \ref{basic-2} for locally
conformally flat metrics to all $N \ge 1$. We shall derive this fact
from the restriction property and the commutator relations of the
operators $\M_{2N}$. We recall that these relations are not
obstructed for large $N$.

We start by proving the following result.

\begin{lemm}\label{comm-simple} For $N \ge 3$, we have
\begin{equation}\label{M-rec}
\M_{2N} i^* = (2N\!-\!2) \left( i^* \left[\frac{\partial^2}{\partial
r^2},\bar{\M}_{2N-2} \right] - 2 [\M_{2N-2},w_2]i^* \right).
\end{equation}
\end{lemm}

\begin{proof} The restriction property (see Theorem \ref{rest-prop})
\begin{equation}\label{RP-case}
i^* \bar{\M}_{2N-2} = \M_{2N-2} i^*, \; N \ge 3
\end{equation}
shows that
\begin{align}
& i^* [\bar{P}_2,\bar{\M}_{2N-2}] - [P_2,\M_{2N-2}]i^*  \nonumber \\
& = i^* \bar{P}_2 \bar{\M}_{2N-2} - i^* \bar{\M}_{2N-2} \bar{P}_2 -
P_2 \M_{2N-2} i^* + \M_{2N-2} P_2 i^* \nonumber \\
& = i^* \bar{P}_2 \bar{\M}_{2N-2} - \M_{2N-2} i^* \bar{P}_2 - P_2
i^* \bar{\M}_{2N-2} + \M_{2N-2} P_2 i^* \nonumber \\
& = (i^* \bar{P}_2 - P_2 i^*) \bar{\M}_{2N-2} - \M_{2N-2} (i^*
\bar{P}_2 - P_2 i^*). \label{simp}
\end{align}
Moreover, the formulas
$$
i^* \bar{P}_2 = i^* \frac{\partial^2}{\partial r^2} + \Delta i^* -
\frac{n\!-\!1}{2} i^* \bar{\J} \quad \mbox{and} \quad P_2 = \Delta -
\frac{n\!-\!2}{2} \J
$$
and the relation $i^* \bar{\J} = \J$ imply that
$$
i^* \bar{P}_2 - P_2 i^* = i^* \left( \frac{\partial^2}{\partial r^2}
- \frac{1}{2} \bar{\J}\right).
$$
Another application of \eqref{RP-case} shows that \eqref{simp}
implies
\begin{equation*}
i^* [\bar{P}_2,\bar{\M}_{2N-2}] - [P_2,\M_{2N-2}]i^* = i^* \left(
\frac{\partial^2}{\partial r^2} - \frac{1}{2} \bar{\J}\right)
\bar{\M}_{2N-2} - i^* \bar{\M}_{2N-2} \left(
\frac{\partial^2}{\partial r^2} - \frac{1}{2} \bar{\J}\right).
\end{equation*}
Now the assertion follows by combining this identity with the
commutator relation \eqref{CR} using \eqref{RP-case} and $i^*
\bar{\J} = \J$.
\end{proof}

By applying \eqref{M-rec} to $u=1$, we find the following recursive
formula.

\begin{lemm}\label{mu-rec-bar} We have
\begin{equation}\label{mu-rec-0}
\mu_{2N} = (2N\!-\!2) \left( i^* \frac{\partial^2}{\partial r^2}
(\bar{\mu}_{2N-2}) - 2 \M_{2N-2}^0 (w_2) \right)
\end{equation}
for $N \ge 3$.
\end{lemm}

\begin{ex} Let $N=3$. Then Lemma \ref{mu-rec-bar} and the explicit formula
\eqref{mu4-ex} for $\mu_4$ show that
\begin{equation}\label{mu6-4}
\mu_6 = 4 i^* \frac{\partial^2}{\partial r^2} \left( -\bar{\J}^2
\!-\! (n\!-\!3) |\bar{\Rho}|^2 \!+\! \bar{\Delta} (\bar{\J}) \right)
- 8 \delta (\Rho d\J).
\end{equation}
In fact, using formulas in Sections \ref{g-4}--\ref{g-24}, one can
verify that \eqref{mu6-4} is equivalent to \eqref{mu6-ex}.
\end{ex}

These preparations enable us to prove the main result of the present
section.

\begin{thm}\label{mu-second} For any locally conformally flat metric $g$,
we have
\begin{equation}\label{base}
\mu_{2N} + \sum_{k=1}^{N-1} \left( 2^k
\frac{(N\!-\!1)!}{(N\!-\!1\!-\!k)!} \right)^2 \M_{2N-2k}(w_{2k}) =
\left(\f\!-\!N\right) (N\!-\!1)! N! 2^{2N} w_{2N}
\end{equation}
for all $N \ge 1$.
\end{thm}

\begin{proof} We use induction on $N$. For $N=1$ and $N=2$, the assertions
can be easily verified by direct calculations. In fact, for $N=1$,
the assertion is the obvious relation
$$
\mu_2 = 4 \left(\f-1\right) w_2
$$
(see Example \ref{E1}). Similarly, for $N=2$ the assertion coincides
with Example \ref{E2}. Now assume that \eqref{base} has been proved
up to $\mu_{2N-2}$ (for conformally flat metrics). The metric
$\bar{g}$ is locally conformally flat, too. Hence, by
\eqref{mu-rec-0}, the assertion for $\mu_{2N}$ is equivalent to the
relation
\begin{align}\label{a2}
\left(\f\!-\!N\right) & (N\!-\!1)! N! 2^{2N} w_{2N}  -
\sum_{k=1}^{N-1} \left( 2^k \frac{(N\!-\!1)!}{(N\!-\!1\!-\!k)!}
\right)^2 \M_{2N-2k}(w_{2k}) \nonumber \\
& = (2N\!-\!2) \Big[ \frac{\partial^2}{\partial r^2}\Big|_{r=0}
\Big( - \sum_{k=1}^{N-2} \left( 2^k
\frac{(N\!-\!2)!}{(N\!-\!2\!-\!k)!}
\right)^2 \bar{\M}_{2N-2-2k}(\bar{w}_{2k}) \nonumber \\
& + \left(\frac{n\!+\!3}{2}\!-\!N\right) (N\!-\!2)! (N\!-\!1)!
2^{2N-2} \bar{w}_{2N-2} \Big) \nonumber \\
& - 2 \M_{2N-2} (w_2) \nonumber \\
& + 2 \Big( - \sum_{k=1}^{N-2} \left( 2^k
\frac{(N\!-\!2)!}{(N\!-\!2\!-\!k)!} \right)^2 \M_{2N-2-2k}(w_{2k}) \nonumber \\
& + \left(\frac{n\!+\!2}{2}\!-\!N\right) (N\!-\!2)! (N\!-\!1)!
2^{2N-2} w_{2N-2} \Big) w_2 \Big].
\end{align}
Here we use the notation $\bar{w}_{2N}$ for the Taylor coefficients
of $\bar{w}(r) = w(dr^2\!+\!g(r))$. Now Lemma \ref{comm-simple} and
the restriction property (Theorem \ref{rest-prop}) imply
\begin{multline}\label{a3}
(2N\!-\!2k\!-\!2) \frac{\partial^2}{\partial r^2}\Big|_{r=0}
(\bar{\M}_{2N-2k-2}(\bar{w}_{2k})) = (2N\!-\!2k\!-\!2) \M_{2N-2k-2}
\left(\frac{\partial^2}{\partial r^2}\Big|_{r=0} (\bar{w}_{2k})\right) \\
+ \M_{2N-2k} (w_{2k}) - 2(2N\!-\!2k\!-\!2)
[w_2,\M_{2N-2k-2}](w_{2k})
\end{multline}
if $2N\!-\!2k\!-\!2 \ge 4$. Next, we observe that
\begin{equation}\label{reduced-3}
\frac{\partial^2}{\partial r^2}\Big|_{r=0} (w(r,s)) = \frac{1}{s}
\frac{\partial w(s)}{\partial s} - 2 w_2 w(s)
\end{equation}
(by the first part of Lemma \ref{l3}). By comparing coefficients of
powers of $s$ in \eqref{reduced-3}, we obtain
\begin{equation}\label{reduced-2}
\frac{\partial^2}{\partial r^2}\Big|_{r=0} (\bar{w}_{2k}) =
(2k\!+\!2) w_{2k+2} - 2 w_2 w_{2k}, \; k \ge 0.
\end{equation}
Hence \eqref{a3} implies
\begin{multline}\label{commu}
(2N\!-\!2k\!-\!2) \frac{\partial^2}{\partial r^2}\Big|_{r=0}
(\bar{\M}_{2N-2k-2}(\bar{w}_{2k})) = \M_{2N-2k}(w_{2k}) \\ +
(2N\!-\!2k\!-\!2) (2k\!+\!2) \M_{2N-2k-2} (w_{2k+2}) - 2
(2N\!-\!2k\!-\!2) \M_{2N-2k-2}(w_{2k}) w_2
\end{multline}
if $2N\!-\!2k\!-\!2 \ge 4$. The relations \eqref{reduced-2} for $k =
N-1$ and \eqref{commu} for $k=1,\dots,N-2$ show that the right-hand
side of \eqref{a2} equals the product of $(2N\!-\!2)$ and the sum
\begin{align*}
& \Big[\!-\frac{2^2 (N\!-\!2)^2}{(2N\!-\!4)}
\left( \M_{2N-2}(w_2) + (2N\!-\!4) 4 \M_{2N-4} (w_4)
- 2(2N\!-\!4) \M_{2N-4}(w_2) w_2 \right) \\ & \vdots \\
& - \frac{2^{2(N-3)}(N\!-\!2)!^2}{4} \left(\M_6 (w_{2N-6}) + 4
(2N\!-\!4) \M_4(w_{2N-4}) - 8 \M_4 (w_{2N-6}) w_2 \right) \\
& - 2^{2(N-2)} (N\!-\!2)!^2 \frac{\partial^2}{\partial r^2}
\Big|_{r=0} \bar{\M}_2 (\bar{w}_{2N-4}) \\
& + \left(\frac{n\!+\!3}{2}\!-\!N\right) (N\!-\!2)!(N\!-\!1)!
2^{2N-2} (2N w_{2N} - 2 w_{2N-2} w_2) \\
& - 2 \M_{2N-2}(w_2) \\
& + 2 \Big( - \sum_{k=1}^{N-2} \left( 2^k
\frac{(N\!-\!2)!}{(N\!-\!2\!-\!k)!} \right)^2 \M_{2N-2-2k}(w_{2k}) \\
& + \left(\frac{n\!+\!2}{2}\!-\!N\right) (N\!-\!2)! (N\!-\!1)!
2^{2N-2} w_{2N-2} \Big) w_2 \Big].
\end{align*}
The latter formula shows that the assertion \eqref{a2} is equivalent
to a formula for
$$
2^{2N-3} (N\!-\!1)! (N\!-\!2)! \frac{\partial^2}{\partial r^2}
\Big|_{r=0} \bar{\M}_2 (\bar{w}_{2N-4}).
$$
in terms of the quantities
$$
w_{2N}, \;\; \M_2(w_{2N-2}), \;\; \M_4(w_{2N-4})
$$
and
$$
w_{2N-2} w_2, \;\; \M_2(w_{2N-4}) w_2;
$$
all other contributions cancel. More precisely, the respective
coefficients of these five terms are given by
$$
2^{2N-1} 3 N! (N\!-\!1)!, \;\; 2^{2N-2} (N\!-\!1)!^2, \;\; 2^{2N-4}
(N\!-\!1)! (N\!-\!2)!
$$
and
$$
-2^{2N-1} (N\!-\!1)!^2, \;\; -2^{2N-2} (N\!-\!1)! (N\!-\!2)!.
$$
This reduces the proof of \eqref{a2} to the following identity.

\begin{lemm}\label{l1} For $N \ge 3$, we have
\begin{multline}\label{iden}
\frac{\partial^2}{\partial r^2} \Big|_{r=0}
(\bar{\M_2}(\bar{w}_{2N-4})) = 12 N(N\!-\!1) w_{2N}  + (2N\!-\!2)
\M_2(w_{2N-2}) + \frac{1}{2} \M_4 (w_{2N-4}) \\ - 2(2N\!-\!2)
w_{2N-2} w_2 - 2 \M_2 (w_{2N-4}) w_2.
\end{multline}
\end{lemm}

Next we use explicit formulas for $\M_2$ and $\M_4$ to reduce Lemma
\ref{l1} to the following result.

\begin{lemm}\label{l2} For $N \ge 2$, we have
\begin{equation}\label{id2}
\frac{\partial^4}{\partial r^4}\Big|_{r=0} (\bar{w}_{2N-4}) = 12
N(N\!-\!1) w_{2N} + 24 (w_2^2-w_4) w_{2N-4} - 12 (2N\!-\!2) w_2
w_{2N-2}.
\end{equation}
\end{lemm}

In order to prove that Lemma \ref{l2} implies Lemma \ref{l1}, we use
the formula
$$
\bar{\M}_2 = \frac{\partial^2}{\partial r^2} + \frac{\dot{v}}{v}(r)
\frac{\partial}{\partial r} + \Delta_{g_r} - \frac{n\!-\!1}{2}
\bar{\J}
$$
(see Eq.~\eqref{Laplace-bar}). It follows that the left-hand side of
\eqref{iden} equals the sum
\begin{multline*}
\frac{\partial^4}{\partial r^4} \Big|_{r=0} (\bar{w}_{2N-4}) + 2
\frac{\partial}{\partial r}\Big|_{r=0}
\left(\frac{\dot{v}}{v}\right) \frac{\partial^2}{\partial r^2}\Big|_{r=0}(\bar{w}_{2N-4}) \\
-2 \delta (\Rho d) (w_{2N-4}) - (d\J,dw_{2N-4}) + \Delta_g
\frac{\partial^2}{\partial r^2}\Big|_{r=0} (\bar{w}_{2N-4}) -
\frac{n\!-\!1}{2} \frac{\partial^2}{\partial r^2}\Big|_{r=0} (
\bar{\J} \bar{w}_{2N-4}).
\end{multline*}
Here we made use of the variational formula
\begin{equation}\label{first-var}
(d/dt)|_0 (\Delta_{g-t\Rho}) = -\delta(\Rho d) - \frac{1}{2}
(d\J,d).
\end{equation}
(see Eq.~(6.9.33) in \cite{juhl-book}). Thus by
$$
(\partial/\partial r)^2 |_{r=0} (\bar{\J}) = |\Rho|^2
$$
(see Lemma 6.11.1 in \cite{juhl-book}) the left-hand side of
\eqref{iden} equals
\begin{multline*}
\frac{\partial^4}{\partial r^4} \Big|_{r=0} (\bar{w}_{2N-4}) + 8 w_2
\frac{\partial^2}{\partial r^2}\Big|_{r=0}(\bar{w}_{2N-4}) \\
-2\delta (\Rho d) (w_{2N-4}) + 4(d w_2,d w_{2N-4}) + \Delta_g
\frac{\partial^2}{\partial r^2}\Big|_{r=0} (\bar{w}_{2N-4}) \\
-\frac{n\!-\!1}{2} \left(|\Rho|^2 w_{2N-4} + \J
\frac{\partial^2}{\partial r^2}\Big|_{r=0} (\bar{w}_{2N-4})\right).
\end{multline*}
We simplify the latter identity using the relation
$$
\frac{\partial^2}{\partial r^2}\Big|_{r=0}(\bar{w}_{2N-4}) =
(2N\!-\!2) w_{2N-2} - 2 w_2 w_{2N-4}
$$
(see Eq.~\eqref{reduced-2} for $k=N\!-\!2$). Now a calculation of
the right-hand side of \eqref{iden} using
$$
\M_2 = -\delta d + 4 \left(\f\!-\!1\right) w_2 \quad \mbox{and}
\quad \M_4 = -4 \delta (\Rho d) - 4 \M_2(w_2) + 32
\left(\f\!-\!2\right) w_4
$$
(see Example \ref{Q-6}) and
$$
\J = -4 w_2 \quad \mbox{and} \quad |\Rho|^2 = -16 w_4 + 8 w_2^2
$$
proves the asserted implication. We omit the details.

Finally, we turn to the proof of Lemma \ref{l2} (for locally
conformally flat metrics).

Note that the identity \eqref{id2} is trivial for $N=2$. We rewrite
the assertion in terms of generating functions. Then
\begin{equation}\label{rd}
\frac{\partial^4}{\partial r^4}\Big|_{r=0}(w(r,s)) = 3
\left(\frac{1}{s} \frac{\partial}{\partial s}\right)^2 (w(s)) - 12
w_2 \left (\frac{1}{s} \frac{\partial}{\partial s}\right) (w(s)) +
24 (w_2^2-w_4) w(s).
\end{equation}

Now for locally conformally flat metrics, we have
\begin{align*}
w(s) & = \det (1-s^2/2 \Rho )^{1/2},  \\
w(r,s) & = \det (1-(r^2+s^2)/2 \Rho)^{1/2}/\det (1-r^2/2\Rho)^{1/2}.
\end{align*}
We set
\begin{align*}
\tilde{w}(s) & \st \det(1-s/2 \Rho )^{1/2}, \\
\tilde{w}(r,s) & \st
\det(1-(r+s)/2\Rho)^{1/2}/\det(1-r/2\Rho)^{1/2}.
\end{align*}
Hence
$$
\frac{\partial^4}{\partial r^4}\Big|_{r=0}(w(r,s)) = 12
\frac{\partial^2}{\partial r^2}\Big|_{r=0}(\tilde{w}(r,s^2))
$$
and
$$
\frac{1}{s} \frac{\partial}{\partial s} (w(s)) = 2 \tilde{w}'(s^2),
\quad \left(\frac{1}{s} \frac{\partial}{\partial s}\right)^2 (w(s))
= 4 \tilde{w}''(s^2).
$$
Here $'$ denotes the derivative with respect to $s$. Hence the
assertion \eqref{rd} is equivalent to the second part of the
following result.

\begin{lemm}\label{l3} We have
\begin{equation}\label{rd-3}
\frac{\partial}{\partial r}\Big|_{r=0} (\tilde{w}(r,s)) =
\tilde{w}'(s) - \tilde{w}_2 \tilde{w}(s)
\end{equation}
and
\begin{equation}\label{rd-2}
\frac{\partial^2}{\partial r^2}\Big|_{r=0}(\tilde{w}(r,s)) =
\tilde{w}''(s) - 2 \tilde{w}_2 \tilde{w}'(s) - 2 (\tilde{w}_4 -
\tilde{w}_2^2) \tilde{w}(s),
\end{equation}
where
$$
\tilde{w}(s) = 1 + \tilde{w}_2 s + \tilde{w}_4 s^2 + \cdots.
$$
\end{lemm}

\begin{proof} The relation
$$
\tilde{w}(r,s) = \tilde{w}(r+s)/\tilde{w}(r)
$$
implies
\begin{align*}
\frac{\partial}{\partial r}\Big|_{r=0}(\tilde{w}(r,s)) & =
\frac{\partial}{\partial r}\Big|_{r=0} (\tilde{w}(r+s)
(1-\tilde{w_2} r + \cdots)) \\
& = \tilde{w}'(s) - \tilde{w}_2 \tilde{w}(s)
\end{align*}
and
\begin{align*}
\frac{\partial^2}{\partial r^2}\Big|_{r=0}(\tilde{w}(r,s)) & =
\frac{\partial^2}{\partial r^2}\Big|_{r=0} (\tilde{w}(r+s)
(1-\tilde{w_2} r - (\tilde{w}_4 - \tilde{w}_2^2) r^2 + \cdots)) \\
& = \tilde{w}''(s) - 2 \tilde{w}_2 \tilde{w}'(s) - 2 (\tilde{w}_4 -
\tilde{w}_2^2) \tilde{w}(s)).
\end{align*}
The proof is complete.
\end{proof}

This completes the proof Theorem \ref{mu-second}.
\end{proof}

Now, in even dimensions $n$ and locally conformally flat metrics,
Theorem \ref{mu-second} extends Theorem \ref{basic-2} to
super-critical orders $2N > n$. The same arguments as in Subsection
\ref{proof-1} then extend Theorem \ref{GF} for such metrics to all
$N \ge 1$.

\begin{rem} For even dimensions $n$ and sub-critical orders $2N \le
n$, the proof of Theorem \ref{mu-second} extends to general metrics.
In fact, it suffices to replace the usage of the explicit formulas
for $w(r)$ and $w(r,s)$ by Theorem \ref{double}.
\end{rem}

Finally, for locally conformally flat metrics, we use Theorem
\ref{GF} to derive the contributions \eqref{CT-flat} to the
zeroth-order term of $\M_{2N}$.

\begin{thm}\label{CT-f} For locally conformally flat $g$ and all $N \ge
1$, the contribution of
\begin{equation}\label{CT-main}
-\left( \frac{\partial^2 w}{\partial r^2} - (n\!-\!1) \frac{1}{r}
\frac{\partial w}{\partial r} \right) \big/ w
\end{equation}
to the zeroth-order term of $\M_{2N}(1)$ equals
\begin{equation*}
-(N\!-\!1)!^2 2^{N-1} \left( \left(\f\!-\!N\right) \tr (\Rho^N) +
\frac{1}{2} \sum_{a=1}^{N-1} \tr (\Rho^a) \tr (\Rho^{N-a})\right).
\end{equation*}
\end{thm}

\begin{proof} By the assumption on $g$, we have $w(r) = \sqrt{\det(1-r^2/2 \Rho)}$.
Now we calculate
\begin{align*}
\frac{1}{r} \frac{1}{w} \frac{\partial w}{\partial r} & =
\frac{1}{r} \frac{\partial}{\partial r} (\log \sqrt{\det(1-r^2/2 \Rho)}) \\
& = \frac{1}{2r} \frac{\partial}{\partial r} (\log \det(1-r^2/2\Rho)) \\
& = \frac{1}{2r} \frac{\partial}{\partial r} (\tr \log
(1-r^2/2\Rho)) \\
& = -\frac{1}{2} \tr \left( \frac{\Rho}{1-r^2/2\Rho}\right) \\
& = - \frac{1}{2} \sum_{N \ge 0} \tr (\Rho^{N+1})
\left(\frac{r^2}{2}\right)^N
\end{align*}
and
\begin{align*}
\frac{\partial^2}{\partial r^2} (\log w) & = \frac{1}{2}
\frac{\partial^2}{\partial r^2} (\tr \log (1-r^2/2\Rho)) \\
& = - \frac{1}{2} \frac{\partial}{\partial r} \left( \tr \left(
\frac{r\Rho}{1-r^2/2\Rho}\right)\right) \\
& = - \frac{1}{2}
\tr\left(\Rho\frac{1+r^2/2\Rho}{(1-r^2/2\Rho)^2}\right)
\\
& = -\frac{1}{2} \sum_{N \ge 1} (2N\!-\!1) \tr (\Rho^N)
\left(\frac{r^2}{2}\right)^{N-1}.
\end{align*}
These relations and
$$
\frac{1}{w} \frac{\partial^2 w}{\partial r^2} = \left(\frac{1}{w}
\frac{\partial w}{\partial r}\right)^2 + \frac{\partial^2}{\partial
r^2} (\log w)
$$
imply that \eqref{CT-main} equals
\begin{multline*}
-\frac{1}{4} r^2 \left( \sum_{N \ge 0} \tr(\Rho^{N+1}) \left(
\frac{r^2}{2} \right)^N \right)^2 \\ + \frac{1}{2} \sum_{N \ge 1}
(2N\!-\!1) \tr(\Rho^N) \left(\frac{r^2}{2}\right)^{N-1} -
\frac{n\!-\!1}{2} \sum_{N \ge 0} \tr(\Rho^{N+1})
\left(\frac{r^2}{2}\right)^N.
\end{multline*}
These sums can be simplified to
\begin{equation*}
-\frac{1}{2} \sum_{N \ge 1} \left( \sum_{a=1}^{N-1} \tr (\Rho^a)
\tr(\Rho^{N-a})\right) \left(\frac{r^2}{2}\right)^{N-1}  +
\frac{2N\!-\!n}{2} \sum_{N \ge 1} \tr (\Rho^N)
\left(\frac{r^2}{2}\right)^{N-1}.
\end{equation*}
Now the assertion follows by taking the coefficient of
$(r^2/4)^{N-1}$.
\end{proof}

For the round sphere $\S^n$, Theorem \ref{CT-f} reproves the fact
that the zeroth-order term of $\M_{2N}$ is given by
$$
-(N\!-\!1)!N! \f\left(\f\!-\!1\right).
$$
(see Eq.~\eqref{sphere-sum}).

%%%%%%%%%%%%%%%%%%%%%%%%%%%%%%%%%%%%%%%%%%%%%%%%%%%%%%%%%%%%%%%%%%%%
\section{The explicit formula for $Q$-curvature}\label{Q-curv-ex}

As a consequence of Theorem \ref{duality} and Theorem \ref{basic-2},
we find an {\em explicit} formula for $Q$-curvature. The following
result restates Theorem \ref{main-Q}.

\begin{thm}{\bf (Explicit formula for $Q$-curvature)}\label{Q-explicit}
For even $n$ and $2N \le n$ and for odd $n$ and all $N \ge 1$,
\begin{equation}\label{Q-ex}
(-1)^N Q_{2N} = \sum_{|I|+a=N} n_{(I,a)} a!(a\!-\!1)!2^{2a}
\M_{2I}(w_{2a}).
\end{equation}
\end{thm}

\begin{proof} By combining
$$
P_{2N}(1) = (-1)^N \left(\f\!-\!N\right) Q_{2N}
$$
with Theorem \ref{duality}, we find
$$
(-1)^N \left (\f\!-\!N\right) Q_{2N} = \sum_{|I|=N} n_I \M_{2I}(1) =
\sum_{|I|+a=N} n_{(I,a)} \M_{2I}(\mu_{2a}).
$$
By Theorem \ref{basic-2}, the latter sum equals
\begin{multline*}
\sum_{|I|+a=N} n_{(I,a)} \M_{2I} \Big( \left(\f\!-\!a\right)
a!(a\!-\!1)!2^{2a} w_{2a} \\[-3mm] - (2a\!-\!2)^2 \M_{2a-2} (w_2) - \cdots
- ((2a\!-\!2)\cdots2) \M_2(w_{2a-2}) \Big).
\end{multline*}
This sum is a linear combination of quantities of the form
$$
\M_{2J} \M_{2b} (w_{2a}), \quad |J|+a+b=N
$$
with respective coefficients
$$
n_{(J,b,a)} \left(\f\!-\!a\right) a!(a\!-\!1)! 2^{2a} - n_{(J,a+b)}
((2(a\!+\!b)\!-\!2)\cdots 2b)^2.
$$
Hence for $2N \ne n$ it suffices to prove that these differences
coincide with
$$
n_{(J,b,a)} \left(\f\!-\!N\right)  a!(a\!-\!1)! 2^{2a}.
$$
In turn, this claim is equivalent to the identity
\begin{equation*}
n_{(J,b,a)} (N\!-\!a) a!(a\!-\!1)! = n_{(J,a+b)} \left(
\frac{(a\!+\!b\!-\!1)!}{(b\!-\!1)!}\right)^2.
\end{equation*}
But, according to \eqref{n-form}, the latter identity is equivalent
to
$$
\binom{N\!-\!1}{a\!+\!b\!-\!1}
\left(\frac{(a\!+\!b\!-\!1)!}{(b\!-\!1)!}\right)^2 =
\binom{N\!-\!a\!-\!1}{b\!-\!1} \binom{b\!+\!a\!-\!1}{b\!-\!1}
\binom{N\!-\!1}{a\!-\!1} (N\!-\!a) a! (a\!-\!1)!.
$$
This relation can be confirmed by a straightforward calculation. The
critical case follows by continuation.
\end{proof}

Some comments are in order. Of course, the actual power of Theorem
\ref{Q-explicit} comes from the identification of the operators
$\M_{2N}$ as certain operators of only {\em second} order (Theorem
\ref{main-P}). In connection with further evaluations of Theorem
\ref{Q-explicit}, it is useful to recall that Graham \cite{G-ext}
gave an algorithm which generates formulas for the
Taylor coefficients $v_{2j}$ of $v(r)$ in terms of the
Poincar\'e-Einstein (or ambient) metric. In turn, the
Taylor coefficients of the square root $w(r)$ easily can be
described in terms of the holographic coefficients $v_{2j}$.

In the locally conformally flat case, Theorem \ref{Q-explicit}
extends to {\em all} $Q$-curvatures $Q_{2N}$. In fact, in the proof
it suffices to replace Theorem \ref{basic-2} by Theorem
\ref{mu-second}. Moreover, in that case, the evaluation of Theorem
\ref{Q-explicit} is greatly simplified by the closed formula
$$
w(r) = \sqrt{\det(1-r^2/2 \Rho)}.
$$

Some information on special contributions to $Q_{2N}$ can be easily
deduced from Theorem \ref{Q-explicit}. In fact, using
$$
n_{(N)} = n_{(1,\dots,1)} = 1,
$$
Eq.~\eqref{Q-ex} shows that $Q_{2N}$ is of the form
$$
(-1)^N N!(N\!-\!1)! 2^{2N} w_{2N} + \cdots + 4 (-1)^N
\M_2^{N-1}(w_2).
$$
Hence, by $4 w_2 = -\J$, we recover the well-known fact
(\cite{branson}, Corollary 1.5) that the quantity $Q_{2N}$ contains
the term $(-1)^N \Delta^{N-1} \J$. This is the contribution with the
most derivatives. Moreover, using
$$
n_{(N-1,1)} = N\!-\!1,
$$
Theorem \ref{Q-explicit} implies that for $N \ge 3$ the full tensor
$g_{(2N-4)}$ contributes to $Q_{2N}$ through
\begin{equation}\label{full-max}
(-1)^{N-1} (N\!-\!1)!(N\!-\!2)! 2^{2N-4} \delta ( g_{(2N-4)} d\J).
\end{equation}
We also reproduce the result (\cite{branson}, Corollary 1.6) that
the contribution with the most derivatives to the {\em total}
$Q$-curvature
$$
\int_M Q_{2N} dvol
$$
on the closed manifold $M$ has the form
\begin{equation}\label{total-most}
\left(\f\!-\!N\right) \int_M (\nabla^{N-2} \J,\nabla^{N-2} \J) dvol.
\end{equation}
In fact, Theorem \ref{Q-explicit} implies that these contributions
only come from
$$
n_{(1,\dots,1)} 2^2 \mu_2 \M_2^{N-2}(w_2) + n_{(2,1,\dots,1)} 2^2
\mu_4 \M_2^{N-3}(w_2).
$$
But using \eqref{mu4} and
$$
n_{(1,\dots,1)} = 1 \quad \mbox{and} \quad n_{(2,1,\dots,1)} =
N\!-\!1,
$$
these terms yield
$$
\left(\f\!-\!1\right) \int_M \J \Delta^{N-2} (\J) dvol - (N\!-\!1) \int_M
\Delta (\J) \Delta^{N-3} (\J) dvol.
$$
Now partial integration gives \eqref{total-most}.

Finally, we illustrate Theorem \ref{Q-explicit} by the following
three low-order special cases.

\begin{ex}\label{Q-4} For $N=2$, Theorem \ref{Q-explicit} states
that
$$
Q_4 = 32 w_4 + 4 \M_2 (w_2).
$$
In view of
$$
4 w_2 = - \J \quad \mbox{and} \quad 32 w_4 = \J^2 - 2 |\Rho|^2,
$$
this formula is equivalent to the recursive formula
$$
Q_4 = 32 w_4 - P_2(Q_2)
$$
and the common explicit formula
$$
Q_4 = \f \J^2 - 2 |\Rho|^2 - \Delta \J.
$$
\end{ex}

\begin{ex}\label{Q-6} The following two formulas are the universal recursive
formula of \cite{Q-recursive} and the explicit formula of Theorem
\ref{Q-explicit} for $Q_6$:
\begin{equation}\label{Q6-rec}
Q_6 + 2 P_2(Q_4) - 2 P_4 (Q_2) + 3 P_2^2(Q_2) = -3!2!2^6 w_6
\end{equation}
and
\begin{equation}\label{Q6-ex}
Q_6 = -3!2!2^6 w_6 - 64 \M_2 (w_4) - 8 \M_4(w_2) - 4 \M_2^2(w_2).
\end{equation}
In both formulas, we have collected all terms of the same nature on
the same side. In order to illustrate the effectiveness of
\eqref{Q6-ex}, we make the terms fully explicit. We have
$$
\M_2 = P_2 = \Delta \!-\! \left(\f\!-\!1\right) \J \quad \mbox{and}
\quad \M_4 = - 4 \delta(\Rho d) + \M_4(1)
$$
with
$$
\M_4(1) + 4 \M_2 (w_2) = 32 \left(\f\!-\!2\right) w_4,
$$
i.e.,
\begin{align}
\M_4(1) & = \left(\f\!-\!2\right)(\J^2\!-\!2|\Rho|^2) +
\left(\Delta\!-\!\left(\f\!-\!1\right)\J\right)(\J) \nonumber \\
& = -\J^2 \!-\!(n\!-\!4)|\Rho|^2 \!+\! \Delta \J. \label{mu4}
\end{align}
Thus, by a straightforward calculation, \eqref{Q6-ex} gives
\begin{multline*}
Q_6 = -3!2!2^6 w_6 - \left(\f\!+\!1\right) \Delta (\J^2)  + 4 \Delta
(|\Rho|^2) - 8 \delta (\Rho d \J) + \Delta^2 \J \\ -
\left(\f\!-\!3\right) \J \Delta \J - 4(n\!-\!3) \J |\Rho|^2 +
\frac{n^2\!-\!12}{4} \J^3
\end{multline*}
or, equivalently,
\begin{multline}\label{Q6-full}
Q_6 = -3!2!2^5 v_6 - \left(\f\!+\!1\right) \Delta (\J^2) + 4 \Delta
(|\Rho|^2) - 8 \delta (\Rho d \J) + \Delta^2 \J \\
- \left(\f\!-\!3\right) \J \Delta \J - 4(n\!-\!6) \J |\Rho|^2 +
\frac{(n\!-\!6)(n\!+\!6)}{4} \J^3.
\end{multline}
In the last step we have used \eqref{v6-gen} and
$$
16 w_6 + 8 v_2 w_4 = 8 v_6.
$$
The identity \eqref{Q6-full} is equivalent to Theorem 6.10.3 in
\cite{juhl-book}. The same formula for $Q_6$ follows from
\eqref{Q6-rec} by a somewhat longer calculation. Note that
\eqref{Q6-full} also reproduces the familiar fact that, in dimension
$n=6$, $Q_6$ differs from a multiple of $v_6$ by a divergence. In
that case, we actually find that
$$
Q_6 = -3!2!2^5 v_6 - 64 \M_2^0(w_4) - 8 \M_4^0(w_2) - 4
\M_2^0\M_2(w_2).
$$
\end{ex}

\begin{ex}\label{Q-8} For even $n \ge 8$ and odd $n \ge 3$ and any
metric,
\begin{multline}\label{Q8-ex}
Q_8 = 3!4!2^8 w_8 + 2^2 3 \M_6 (w_2) + 2^4 18 \M_4 (w_4) + 2^6 36
\M_2 (w_6) \\ + 2^2 3 \M_4 \M_2 (w_2) + 2^2 4 \M_2\M_4(w_2) + 2^4 6
\M_2^2(w_4) + 2^2 \M_2^3(w_2).
\end{multline}
The operators $\M_2,\M_4,\M_6$ are second-order (see
\eqref{full-M24}, \eqref{full-M6}). The quantity $w_8$ is given by
\eqref{w8-v} and \eqref{v8}. This formula should be compared with
the universal recursive formula
\begin{multline}\label{Q8-rec}
Q_8 + 3 P_2(Q_6) + 3 P_6(Q_2) - 9 P_4(Q_4) \\
- 8 P_2 P_4(Q_2) + 12 P_2^2 (Q_4) - 12 P_4 P_2 (Q_2) + 18 P_2^3(Q_2)
= 3! 4! 2^8 w_8
\end{multline}
proved in \cite{juhl-Q8}, \cite{Q-recursive}.
\end{ex}

Another formula for $Q_8$ was displayed in \cite{G-P} (see Fig.~5).
Unfortunately, it is a rather non-trivial task to compare that
formula with those given in Example \ref{Q-8}.\footnote{In this
direction, we only note that all formulas contain the contribution
$-96 (\Omega^{(1)},\Omega^{(1)})$ with a second-order pole at
$n=4$.}

%%%%%%%%%%%%%%%%%%%%%%%%%%%%%%%%%%%%%%%%%%%%%%%%%%%%%%%%%%%%%%%%%%%%%%%%%%%
\section{Some low-order cases}\label{low}

In the present section, we make explicit the consequences of the
general theory for the third and the fourth GJMS-operator $P_6$ and
$P_8$.

%%%%%%%%%%%%%%%%%%%%%%%%%%%%%%%%%%%%%%%%%%%%%%%%%%%%%%%%%%%%%%%%%%%%%%%%%
\subsection{The conformal cube $P_6$ of the Laplacian}\label{cube}

We discuss two formulas for the conformally covariant cube $P_6$ of
the Laplacian. In this case, explicit formulas in terms of the
metric are still not too complicated. Although a number of
alternative approaches have been used to understand this
higher-order generalization of the Paneitz-operator, it seems that
the formula for $P_6$ given in Theorem \ref{P6-full-gen} has not
been derived before. For alternative methods we refer to
\cite{branson-scand}, \cite{wue1}, \cite{wue2}, \cite{G-P} and
\cite{Casimir}.

We recall that one of the key observations which leads to the
following results is the commutator relation
$$
\M_6 i^* = 4 (i^* [ \bar{P}_2,\bar{\M}_4] - [P_2,\M_4] i^*).
$$
By $\M_4 = P_4 - P_2^2$ (for any metric), the latter relation is
equivalent to the identity
$$
\M_6 i^* = 4 (i^* [ \bar{P}_2,\bar{P}_4] - [P_2,P_4] i^*).
$$
In the critical dimension $n=6$ and for locally conformally flat
metrics, this relation first appeared in \cite{juhl-book} (Theorem
6.11.17).

%%%%%%%%%%%%%%%%%%%%%%%%%%%%%%%%%%%%%%%%%%%%%%%%%%%%%%%%%%%%%%%%%%%%%%%%%%
\subsubsection{The recursive formula}\label{P6-recursive}

The following formula describes $P_6$ in terms of the lower-order
GJMS-operators $P_2$, $P_4$ and the $Q$-curvatures $Q_6$. For all
metrics and in dimension $n \ge 3$ ($n \ne 4$), we have
\begin{equation}\label{P6-rec}
P_6 = (2 P_2 P_4 + 2 P_4 P_2 - 3P_2^3)^0 - 48 \delta (\Rho^2 d) -
\frac{16}{n\!-\!4} \delta (\B d) - \left(\f\!-\!3\right) Q_6.
\end{equation}
Here the tensor
\begin{equation}\label{Bach}
\B_{ij} \st \Delta (\Rho)_{ij} - \nabla^k \nabla_j (\Rho)_{ik} +
\Rho^{kl} W_{kijl}
\end{equation}
being defined by the Schouten tensor $\Rho$ and the Weyl tensor $W$
is known as the Bach tensor. As usual, the superscript $^0$ removes
the zeroth-order term. Eq.~\eqref{P6-rec} is equivalent to the
formula
\begin{equation}\label{M6-ex}
\M_6^0 = - 48 \delta(\Rho^2 d) - \frac{16}{n\!-\!4} \delta(\B d)
\end{equation}
for the principal part of $\M_6$ (see also the discussion in Section
\ref{principal-part}). Finally, the $Q$-curvature $Q_6$ is given by
the universal recursive formula
\begin{equation*}
Q_6 + 2 P_2(Q_4) - 2 P_4(Q_2) + 3 P_2^2(Q_2) = - 3!2!2^6 w_6
\end{equation*}
(see Eq.~\eqref{Q6-rec}), where the quantity $w_6$ is related
through
$$
16 w_6 = 8 v_6 - 4 v_4 v_2 + v_2^3
$$
to the holographic coefficients
\begin{equation}\label{v6-gen}
v_2 = - \frac{1}{2} \J, \quad v_4 = \frac{1}{8} (\J^2 \!-\!
|\Rho|^2) \quad \mbox{and} \quad  v_6 = - \frac{1}{8} \tr
(\wedge^3\Rho) \!-\! \frac{1}{24(n\!-\!4)}(\B,\Rho).
\end{equation}
The resulting fully explicit formula for $Q_6$ is displayed in
\eqref{Q6-full}.

Alternatively, by combining the definition of $\M_6$ (see Eq.~
\eqref{M6-def}) with \eqref{M6-ex}, we obtain the recursive formula
\begin{equation}\label{P6-rec2}
P_6 = (2 P_2 P_4 + 2 P_4 P_2 - 3 P_2^3) - 48 \delta (\Rho^2 d) -
\frac{16}{n\!-\!4} \delta (\B d) + \mu_6
\end{equation}
with the scalar curvature quantity $\mu_6$ as in \eqref{mu6-ex}.

%%%%%%%%%%%%%%%%%%%%%%%%%%%%%%%%%%%%%%%%%%%%%%%%%%%%%%%%%%%%%%%%%%%%%%%%%%
\subsubsection{The explicit formula}\label{P6-explicit}

Under the same assumptions on the dimension $n$ as in Section
\ref{P6-recursive}, the following formula describes $P_6$ for
general metrics in terms of the second-order operators $\M_2$,
$\M_4$ and $\M_6$. We have
\begin{equation}\label{P6-ex}
P_6 = \M_6 + 2 \M_2 \M_4 + 2 \M_4 \M_2 + \M_2^3
\end{equation}
(see Example \ref{P6-form}), where
\begin{equation}\label{full-M24}
\M_2 = P_2 = - \delta d + \mu_2, \quad \M_4 = -4\delta(\Rho d)+\mu_4
\end{equation}
and
\begin{equation}\label{full-M6}
\M_6 = - 48 \delta(\Rho^2 d) - \frac{16}{n\!-\!4} \delta(\B d) +
\mu_6.
\end{equation}
The curvature quantities $\mu_2$, $\mu_4$ and $\mu_6$ are determined
by the identities
\begin{equation}\label{mu-24}
\mu_2 = 2^2 \left(\f\!-\!1\right) w_2, \quad \mu_4 = 2! 2^4
\left(\f\!-\!2\right) w_4 - 4 \M_2 (w_2)
\end{equation}
and
\begin{equation}\label{full-mu6}
\mu_6 = 3! 2! 2^6 \left(\f\!-\!3\right) w_6 - 4^2 \M_4 (w_2) - 8^2
\M_2(w_4).
\end{equation}
These are special cases of Theorem \ref{basic-2} (see also
\eqref{mu4-case} and \eqref{mu6-case}). Note also that
\begin{equation}\label{w24}
2 w_2 = v_2 \quad \mbox{and} \quad 8 w_4 = 4 v_4 - v_2^2.
\end{equation}
Some calculations yield the explicit formulas
\begin{equation}\label{mu4-ex}
\mu_4 = -\J^2 \!-\! (n\!-\!4)|\Rho|^2 \!+\! \Delta \J \quad
\mbox{(see Eq.~\eqref{mu4})}
\end{equation}
and
\begin{multline}\label{mu6-ex}
\mu_6 = - 8 \frac{n\!-\!6}{n\!-\!4} (\B,\Rho) - 8 (n\!-\!6) \tr
(\Rho^3) \\ - 16 \J |\Rho|^2 - 4 |d\J|^2 + 4 \Delta (|\Rho|^2) - 16
\delta (\Rho d\J).
\end{multline}
The formulas \eqref{mu4-ex} and \eqref{mu6-ex} should also be
compared with Theorem \ref{CT-f}.

Now we use these results to derive a fully explicit formula for
$P_6$ in terms of the metric. In the critical case, we find the
following result.

\begin{thm}\label{P6-full} On manifolds of dimension $n=6$, the
GJMS-operator $P_6$ is given by the explicit formula
\begin{equation}\label{P6-fully}
P_6 = \Delta^3 + 4 \Delta \delta  T_2 d + 4 \delta T_2 d \Delta + 2
\Delta (\J \Delta) + 8 \delta T_4 d
\end{equation}
with
$$
T_2= \J - 2 \Rho \quad \mbox{and} \quad T_4 = - \J^2 + |\Rho|^2 + 4
\J \Rho - 6 \Rho^2 - \B.
$$
\end{thm}

Theorem \ref{P6-full} extends as follows to general dimensions.

\begin{thm}\label{P6-full-gen} On manifolds of dimension $n \ge 3$
($n \ne 4$), the GJMS-operator $P_6$ is given by the explicit
formula
\begin{equation}\label{P6-fully-gen}
P_6 = \Delta^3 + \Delta \delta T_2 d + \delta T_2 d \Delta +
\left(\f\!-\!1\right) \Delta (\J \Delta) + \delta T_4 d -
\left(\f\!-\!3\right) Q_6
\end{equation}
with
$$
T_2 = (n\!-\!2) \J  - 8 \Rho,
$$
\begin{multline*}
T_4 = - \left(3 \left(\f\!-\!1\right)^2\!-\!4\right) \J^2 +
4(n\!-\!4) |\Rho|^2 + 16 \left(\f\!-\!1\right) \J \Rho \\[-2mm]
+ (n\!-\!6) \Delta \J  - 48 \Rho^2  - \frac{16}{n\!-\!4} \B
\end{multline*}
and $Q_6$ as in \eqref{Q6-full}.
\end{thm}

\begin{proof} We make \eqref{P6-ex} explicit. We recall that
$$
\mu_2 = - \left(\f\!-\!1\right) \J \quad \mbox{and} \quad \mu_4 =
-\J^2 \!-\! (n\!-\!4)|\Rho|^2 \!+\! \Delta \J
$$
(see Eq.~\eqref{mu4-ex}). In order to prove the assertion, it
suffices to determine the non-zero order contributions. Let $\lambda
= \f-1$. Then we have to determine the sum of the non-zero order
terms of the operators
\begin{align}
& \Delta^3 u - \lambda \J \Delta^2 u - \lambda \Delta^2 (\J u) -
\lambda \Delta (\J \Delta u) + \lambda^2 \J \Delta (\J u) +
\lambda^2 \J^2 \Delta u + \lambda^2 \Delta (\J^2 u),
\label{co-1} \\
& - 2 \Delta (\J^2 u) - 2 \J^2 \Delta u + 2 \Delta (u \Delta \J) + 2
\Delta (\J) \Delta (u), \label{co-2}\\
& -8\Delta \delta (\Rho du) - 8 \delta (\Rho d) \Delta u +
4(n\!-\!4) \delta (|\Rho|^2 du) - 48 \delta (\Rho^2 du) - 8
\delta(\B du) \nonumber
\end{align}
and
\begin{equation}\label{co-3}
8 \lambda \J \delta (\Rho d u) + 8 \lambda \delta (\Rho d (\J u)).
\end{equation}
In the following calculations, we omit the zeroth-order terms. Now
\begin{align*}
\J \Delta (\J u) & = \J (\Delta (\J) u + 2 (d\J,du) + \J \Delta(u)) \\
& = (d \J^2,du) + \J^2 \Delta(u) \\
& = -\delta(\J^2 du)
\end{align*}
and
$$
\J^2 \Delta (u) + \Delta(\J^2 u) = -2 \delta(\J^2 du)
$$
imply that the last three terms of \eqref{co-1} yield the
contribution
$$
-3\lambda^2 \delta (\J^2 du).
$$
Similarly, the first two terms of \eqref{co-2} yield the
contribution
$$
4 \delta (\J^2 du).
$$
Next, we apply the relation $\Delta(d\J,du)= - \Delta \delta(\J du)
- \Delta (\J \Delta u)$ to find that
\begin{align*}
& \J \Delta^2 u + \Delta^2 (\J u) + \Delta (\J \Delta u) \\
& = \J \Delta^2 u + \Delta ( \Delta(\J) u + 2 (d\J,du) + \J \Delta(
u)) + \Delta (\J \Delta u) \\
& = \J \Delta^2 u + 2(d \Delta \J,du) + \Delta (\J) \Delta (u) - 2
\Delta \delta (\J du).
\end{align*}
It follows that the first four terms in \eqref{co-1} and the last
two terms of \eqref{co-2} give rise to the contribution
\begin{multline*}
\Delta^3 u - \left(\f\!-\!1\right) \J \Delta^2 u - (n\!-\!2)(d
\Delta \J,du) - \left(\f\!-\!1\right) \Delta (\J) \Delta (u) +
(n\!-\!2) \Delta \delta (\J du) \\ + 4 \Delta (u) \Delta (\J) + 4
(du,d \Delta \J).
\end{multline*}
We rewrite this sum in the form
\begin{multline*}
\Delta^3 u - \left(\f\!-\!1\right) (\J \Delta^2 u - \Delta (\J)
\Delta(u)) + (n\!-\!2) \Delta \delta (\J du) \\ - (n\!-\!6) ((d
\Delta \J,du) + \Delta (\J) \Delta (u)).
\end{multline*}
Now the identity
$$
-\left(\f\!-\!1\right) \J \Delta^2 u + \left(\f\!-\!1\right)
\Delta(\J) \Delta(u) = (n\!-\!2) \delta (\J d \Delta u) +
\left(\f\!-\!1\right) \Delta(\J \Delta u)
$$
implies that the sum equals
$$
\Delta^3 u + (n\!-\!2) \delta (\J d \Delta u) + (n\!-\!2) \Delta
\delta (\J du) +\left(\f\!-\!1\right) \Delta(\J \Delta u) +(n\!-\!6)
\delta (\Delta \J,du).
$$
Finally, we observe that \eqref{co-3} equals $16 \lambda \delta (\J
\Rho du)$ (modulo a zeroth-order term). This completes the proof.
\end{proof}

Some comments are in order. First of all, an easy calculation
confirms that for the round sphere $\S^n$ the formula
\eqref{P6-fully-gen} actually simplifies to the product
$$
(\Delta \!-\! m(m\!-\!1))(\Delta \!-\! (m\!+\!1)(m\!-\!2)) (\Delta
\!-\! (m\!+\!2)(m\!-\!3)), \; m = n/2.
$$
The first four terms in \eqref{P6-fully-gen} cover all contributions
to $P_6$ which contain more than four derivatives. In Subsection
\ref{many} we shall recognize this fact as a special case of a
general result.

Next, it is worth noting that for general metrics the formula
\eqref{P6-fully} coincides with a formula of Branson for a
conformally covariant cube of the Laplacian in dimension $n=6$
(\cite{branson-scand}, Theorem 2.8).\footnote{Branson's tensor $Y$
actually coincides with the sum $-\B-(n-4)\Rho^2$.} Another formula
for a conformally covariant cube of the Laplacian in general
dimensions $n \ge 6$ containing the second-order operator
$-\delta(\B d)/(n\!-\!4)$ was given in Proposition 3.1 of
\cite{wue1}. However, a full comparison is difficult since W\"{u}nsch's
formula is stated in a less explicit form. Finally, we note that
Theorem \ref{P6-full} reproduces the formula for the critical $P_6$
contained in Proposition 8.3 of \cite{AG}.\footnote{In \cite{AG},
$\Rho$ denotes the double of the Schouten tensor. Thus, in terms of
the notation of the present paper, the tensors $E_1$ and $E_2$ in
\cite{AG} are given by $2\Rho-\J$ and $4\Rho^2+2\B-2|\Rho|^2$.}

An independent proof of the conformal covariance of the operator on
the right-hand side of \eqref{P6-rec} in the critical dimension
$n=6$ can be found in the Appendix 13.2 of \cite{juhl-power}. In
general dimensions, the formula \eqref{P6-rec} first appeared in
Section 6.12 of \cite{juhl-book}. There it was derived by
(infinitesimal) conformal variation of (a version of) the universal
recursive formula \eqref{Q6-rec} for $Q_6$.

%%%%%%%%%%%%%%%%%%%%%%%%%%%%%%%%%%%%%%%%%%%%%%%%%%%%%%%%%%%%%%%%%%%%%%%%%%%%%%
\subsection{The conformal fourth power $P_8$ of the  Laplacian}\label{quartic}

Explicit formulas for the principal parts of $\M_4$ and $\M_6$ are
displayed in Section \ref{cube}. The following result follows from
Theorem \ref{base-c}. It leads to an explicit formula for $P_8$.

\begin{thm}\label{M8-case} For even $n \ge 8$ and odd $n \ge 3$,
\begin{equation}\label{M8-principal}
\M_8^0 = - 96 \delta ([\Omega^{(2)} - 4 (\Rho \Omega^{(1)} +
\Omega^{(1)} \Rho) + 12 \Rho^3] d),
\end{equation}
where $\Omega^{(1)}, \Omega^{(2)}$ are the first two extended
obstruction tensors.
\end{thm}

For the definition of extended obstruction tensors we refer to
\cite{G-ext} and \eqref{e-obst}. Section 2 of \cite{G-ext} also
gives an explicit formula for $\Omega^{(2)}$ in terms of $g$. The
full details of the derivation of \eqref{M8-principal} from the
general formula for the principal part can be found in
\cite{juhl-power}.

Similarly as in Subsection \ref{P6-recursive}, the equation
\eqref{M8-principal} can be viewed as a recursive formula for $P_8$.
We omit the details and continue with the description of $P_8$ in
terms of the operators $\M_2,\dots,\M_8$. By combining Theorem
\ref{M8-case} with Example \ref{P8-form}, we obtain the following
result.

\begin{thm}\label{P8-final} For all even $n \ge 8$ and all odd $n \ge 3$,
the GJMS-operator $P_8$ is given by the formula
\begin{multline}\label{P8-complete}
P_8 = \M_2^4 + 3 (\M_2^2 \M_4 + \M_4 \M_2^2) + 4 \M_2\M_4 \M_2 + 9
\M_4^2 + 3 (\M_2 \M_6 + \M_6 \M_2) \\ - 96 \delta ([\Omega^{(2)} - 4
(\Rho \Omega^{(1)} + \Omega^{(1)} \Rho) + 12 \Rho^3] d) + \mu_8,
\end{multline}
where
\begin{equation}\label{w8}
\mu_8 = \left(\f\!-\!4\right) 3!4!2^8 w_8 - 48^2 \M_2 (w_6) - 24^2
\M_4(w_4) - 6^2 \M_6 (w_2).
\end{equation}
\end{thm}

The formula \eqref{w8} is a consequence of Theorem \ref{basic-2}. In
terms of holographic coefficients, the function $w_8$ can be
described as
\begin{equation}\label{w8-v}
128 w_8 = 64 v_8 - 32 v_6 v_2 - 16 v_4^2 + 24 v_2^2 v_4 - 5 v_2^4,
\end{equation}
where $v_2,v_4,v_6$ are as in Section \ref{cube}. Finally, $v_8$ is
given by
\begin{multline}\label{v8}
2^4 v_8 = \tr (\wedge^4 \Rho) \\ + \frac{1}{3} \tr
(\Rho^2\Omega^{(1)}) - \frac{1}{3} \tr (\Rho) \tr (\Rho
\Omega^{(1)}) - \frac{1}{12} \tr (\Rho \Omega^{(2)}) - \frac{1}{12}
\tr (\Omega^{(1)} \Omega^{(1)})
\end{multline}
(see \cite{G-ext}, Eq.~(2.23)).

The above results actually suffice to derive from
\eqref{P8-complete} a fully explicit formula for $P_8$ in terms of
the metric. The details will be presented elsewhere. Here we only
note that the contributions with more than {\em four} derivatives
are contained in the sum
$$
\sum_{k=0}^2 \Delta^k \delta((n\!-\!2)\J - 4 (k\!+\!1)(3\!-\!k)
\Rho)d \Delta^{2-k} + \left(\f\!-\!1\right) \sum_{k=1}^2 \Delta^k
(\J \Delta^{3-k}).
$$
This is a special case of a result discussed in Subsection
\ref{many}.

We finish the present section with two brief comments on alternative
approaches. As special cases of their approach using tractor
calculus, Gover and Peterson \cite{G-P} derived formulas for $P_6$
and $P_8$ in terms of the metric (and in general dimensions).
Unfortunately, the conventions used in the formulation of their
final results hinder a direct comparison with the present results.
W\"{u}nsch \cite{wue2} (Theorem 4.3) gave a formula for a conformally
covariant {\em fourth} power of the Laplacian in general dimensions.
That formula is far from being explicit in terms of the metric and
the relation to the above description of the operator $P_8$ (see
Theorem \ref{P8-final}) is not clear.\footnote{Note only that the
contributions of $|\Omega^{(1)}|^2$ to both operators differ.}

%%%%%%%%%%%%%%%%%%%%%%%%%%%%%%%%%%%%%%%%%%%%%%%%%%%%%%%%%%%%%%%%%%%%%%%%%
\section{Further results and comments}\label{open}

In the present section, we describe a few additional results and
collect a number of comments.

%%%%%%%%%%%%%%%%%%%%%%%%%%%%%%%%%%%%%%%%%%%%%%%%%%%%%%%%%%%%%%%%%%%%%%%%%%%%
\subsection{General signature}\label{sign}

Although the original construction of GJMS-operators works for
metrics of arbitrary signature $(p,q)$, we restricted all
discussions in the present paper to the framework of Riemannian
metrics. This is motivated by our usage of the identification of
GJMS-operators in the asymptotic expansions of eigenfunctions of the
Laplacian of Poincar\'e-Einstein metrics \cite{GZ} and of the theory
of residue families \cite{juhl-book}. However, it is natural to
expect that the main results {\em literally} extend to metrics of
{\em arbitrary} signature.

For the discussion of the special cases $\S^{(p,q)}$ we refer to
\cite{juhl-power} and \cite{JK} (see also Example \ref{pseudo}).
Another interesting special case in this direction are $pp$-waves
$$
g = \sum_{i=1}^{n-2} dx_i^2 + 2 du dr + 2 h du^2
$$
with $h \in C^\infty(\r^n)$ being independent of $r$. For these
metrics, Leistner and Nurowski \cite{pp-waves} found a {\em closed}
formula for the corresponding ambient metrics. In particular, one
finds that $v(r)=1$. Therefore, the extension of \eqref{GF-form}
reduces to the simple formula $\H(r) = - \delta(g_r^{-1}d)$. It
follows that (for $N \ge 2)$ the operators $\M_{2N}$ are respective
constant multiples of $\Delta^{N-1}(h)(\partial^2/\partial r\partial
u)$.

We also note that the validity of all factorizations of residue
families for {\em general} metrics (Theorem \ref{Factor-A} and
Theorem \ref{Factor-B}) has the consequence that the results and
conjectures in \cite{FJ} literally extend from locally conformally
flat metrics to general metrics.

%%%%%%%%%%%%%%%%%%%%%%%%%%%%%%%%%%%%%%%%%%%%%%%%%%%%%%%%%%%%%%%%%%%%%%%%%%%%
\subsection{Einstein metrics}\label{ME}

If $g$ is Einstein, Theorem \ref{main-P} implies Gover's product
formula \cite{G-ein}
\begin{equation}\label{Einstein}
P_{2N}(g) = \prod_{j=\f}^{\f+N-1} (\Delta_g \!-\! \lambda_g j
(n\!-\!1\!-\!j)),
\end{equation}
where $\lambda_g$ is the constant $\scal(g)/n(n\!-\!1)$.
Eq.~\eqref{Einstein} extends the product formula
$$
P_{2N} = \prod_{j=\f}^{\f+N-1} (\Delta \!-\! j (n\!-\!1\!-\!j))
$$
for the round sphere $\S^n$. We sketch the argument. First, we
recall that \cite{FG-final}
$$
g_+ = r^{-2} (dr^2 + (1 \!-\! \lambda r^2/4)^2 g).
$$
Hence $w(r) = (1-\lambda r^2/4)^\f$ and a calculation shows that
$$
\D(g)(r) = \left(1 \!-\! \lambda \frac{r^2}{4}\right)^{-2}
\left(\Delta \!-\! \lambda \f\left(\f\!-\!1\right)\right).
$$
Now Theorem \ref{main-P} yields
$$
\H(g)(r) = (1\!-\!\lambda r^2/4)^{-2} P_2(g).
$$
Hence
$$
\M_{2N} = N!(N\!-\!1)! \lambda^{N-1} P_2.
$$
It follows that in the sum
\begin{equation}\label{S}
P_{2N} = \sum_{|I|=N} n_I \M_{2I}
\end{equation}
the terms with exactly $r$ factors are multiples of $\lambda^{N-r}
P_2^r$. Now on the sphere $\S^n$ the sum on the right-hand side of
\eqref{S} simplifies to the product
$$
\prod_{k=0}^{N-1} (P_2 + k(k\!+\!1)).
$$
This follows by combining Theorem \ref{duality} with the summation formula
Theorem 6.1 in \cite{juhl-power}. But an easy modification of the proof of
the latter result yields
$$
\prod_{k=0}^{N-1} (P_2 + \lambda k(k\!+\!1))
$$
in the present situation. This proves \eqref{Einstein}.

%%%%%%%%%%%%%%%%%%%%%%%%%%%%%%%%%%%%%%%%%%%%%%%%%%%%%%%%%%%%%%%%%%%%%%%%%%%%
\subsection{Singular cases}\label{sing}

We note that the formula \eqref{P6-ex} implies the conformal
covariance of the right-hand side if $4 \ne n \ge 3$. Moreover, it
shows that in dimension $n=4$ the residue
$$
R_6 = \Res_{n=4} (\M_6)
$$
is conformally covariant in the sense that $e^{5\varphi} \hat{R}_6 =
R_6 e^{-\varphi}$.\footnote{Here the dimension $n$ is treated as a
formal variable.} A calculation using \eqref{full-M6},
\eqref{full-mu6} and \eqref{v6-gen} actually yields the explicit
formula
$$
R_6 = -16 \delta (\B d) + \Res_{n=4} (\mu_6) = -16 (\delta (\B d) -
(\B,\Rho)).
$$
Similarly, the formula for $P_8$ in Example \ref{P8-form} implies the
conformal covariance of the right-hand side if $n \ge 3$ differs from
$4$ and $6$. Moreover, it shows that in dimension $n=6$ the residue
$$
R_8 = \Res_{n=6} (\M_8)
$$
is conformally covariant, i.e., $e^{7\varphi} \hat{R}_8 = R_8
e^{-\varphi}$. In order to find an explicit formula for $R_8$, we
first note that
$$
\M_8^0 = - 96 \delta (\Omega^{(2)} d) + \cdots
$$
implies
$$
\Res_{n=6} (\M_8) = -48 \delta (\OB_6 d) + \Res_{n=6}(\mu_8),
$$
where
$$
\OB_6 = 2 \Res_{n=6}(\Omega^{(2)})
$$
(see \cite{G-ext}, Proposition 2.8) is the Fefferman-Graham
obstruction tensor in dimension $n=6$. An explicit formula for
$\OB_6$ can be found, for instance, in \cite{GrH}. Now Theorem
\ref{basic-2} and
$$
v_8 = - \frac{1}{2^4} \frac{1}{12} \tr (\Rho \Omega^{(2)}) + \cdots
$$
(see Eq.~\eqref{v8}) yield the explicit formula
$$
R_8 = -48 (\delta (\OB_6 d) - (\OB_6,\Rho)).
$$
As the Bach tensor $\B$ in dimension $4$, $\OB_6$ is trace-free and
divergence-free. Moreover, its transformation law $e^{4\varphi} \hat
{\OB}_6 = \OB_6$ in dimension six generalizes $e^{2\varphi} \hat
{\B} = \B$ in dimension four. A direct calculation, using these
properties, confirms that the operator $R_8$ is conformally
covariant (in dimension six).

We also note that $\M_8$ has a {\em second-order} pole at $n=4$. It
is caused by the contribution $|\Omega^{(1)}|^2$ to the holographic
coefficient $v_8$ which in turn contributes to the zeroth-order term
$\mu_8 = \M_8 (1)$ of $\M_8$. The corresponding scalar quantity $I =
|\B|^2$ is conformally covariant (in dimension $n=4$) in the sense
that $e^{8\varphi} \hat{I} = I$.

%%%%%%%%%%%%%%%%%%%%%%%%%%%%%%%%%%%%%%%%%%%%%%%%%%%%%%%%%%%%%%%%%%%%%%%%%%%%
\subsection{More universal relations for GJMS-operators}\label{URR-more}

We recall that the degree of the polynomial $\lambda \mapsto
D_{2N}^{res}(\lambda)$ is $N$. Thus, Theorem \ref{D-rep} implies the
vanishing results
\begin{equation}\label{van}
d_{2N}^{(1)} = \cdots = d_{2N}^{(N-1)} = 0
\end{equation}
for the coefficients of $D_{2N}(\lambda)$ (see Eq.~\eqref{def-d}).
These are $N\!-\!1$ universal identities which involve the
GJMS-operators $P_2,\dots,P_{2N}$ and
$\bar{P}_2,\dots,\bar{P}_{2N}$. Only the first two identities
\begin{equation}\label{van-2}
d_{2N}^{(1)} = 0  \quad  \mbox{and}  \quad d_{2N}^{(2)} = 0
\end{equation}
were utilized in the proofs of the main results of the present
paper. In that connection, it was important to recognize the {\em
equivalence} of the identities in \eqref{van-2} to the restriction
property and the commutator relations for the operators $\M_{2N}$
(using the restriction property of all $\M_{2M}$ for $M < N$.) A
similar description of the identities $d_{2N}^{(k)}=0$ for
$k=3,\dots,N\!-\!1$ (for $N \ge 4$) and an understanding of their
significance remain to be found. Along these lines, one can prove
for instance that the three relations in \eqref{van} for $N=4$
contain the additional identity
\begin{equation*}
(P_2 \M_6 + 3 \M_6 P_2 - 4 [P_2,[P_2,\M_4]]) i^* = i^*
\left(\bar{\M}_6 \bar{P_2} + 3 \bar{P}_2 \bar{\M}_6 - 4
[\bar{P}_2,[\bar{P}_2,\bar{\M}_4]]\right).
\end{equation*}

%%%%%%%%%%%%%%%%%%%%%%%%%%%%%%%%%%%%%%%%%%%%%%%%%%%%%%%%%%%%%%%%%%%%%%%%%%%%%%%
\subsection{The structure of the polynomials $\pi_{2N}(\lambda)$}\label{Q-pol}

In \cite{Q-recursive}, we introduced the polynomials
$\pi_{2N}(\lambda)$ of degree $N\!-\!1$.\footnote{The
operator-valued polynomials $\pi_{2N}(\lambda)$ are not to be
confused with the polynomials defined in \eqref{poly}.} These
polynomials satisfy systems of respective $N$ factorization
relations which allow to describe them in terms of compositions of
GJMS-operators. In fact, the following formula should be compared
with Theorem \ref{D-rep}. We have
\begin{equation}\label{pi-ex}
\pi_{2N}(\lambda) = \frac{1}{(N\!-\!1)!} \sum_{|I|=N}
\frac{b_N(\lambda\!+\!\f\!-\!N)}{(\lambda\!+\!\f\!-\!2N\!+\!I_l)}
m_I P_{2I},
\end{equation}
where $b_N(\lambda)= \prod_{k=0}^{N-1} (\lambda-k)$. In particular,
the constant term $\pi_{2N}(-n/2\!+\!N)$ equals $P_{2N}$, up to a
sign, and the operator $\M_{2N}$ is a constant multiple of the {\em
leading} coefficient of $\pi_{2N}(\lambda)$. Moreover, it turns out
that the {\em sub-leading} coefficient $\M_{2N}^{(2)}$ of
$\pi_{2N}(\lambda)$ can be expressed in terms of the leading
coefficients according to
\begin{equation}\label{comm-2}
\M_{2N}^{(2)} = - \frac{N}{2} \M_{2N} - \sum_{k=1}^{N-1}
\binom{N\!-\!2}{N\!-\!1\!-\!k} \binom{N\!-\!1}{N\!-\!1\!-\!k}
\M_{2k} \M_{2N-2k}.
\end{equation}
In fact, the identity \eqref{comm-2} is a consequence of the
quadratic relation \eqref{comb-first} for the coefficients $m_I$.
Since the operators $\M_2, \M_4, \cdots$ are of second order, the
identity \eqref{comm-2} implies that the sub-leading coefficients
are of {\em fourth} order. In other words, the sub-leading
coefficient of $\pi_{2N}(\lambda)$ defines a universal linear
combination of order $2N$ compositions of GJMS-operators which
actually is of order four. This cancellation result confirms an
expectation formulated in \cite{Q-recursive} (see the comments after
Theorem 4.2). Similar results are expected for {\em all}
coefficients of the polynomials $\pi_{2N}(\lambda)$. In the extreme
case, the inversion formula (Theorem \ref{duality}) can be regarded
as a description of its constant term $P_{2N}$ along these lines.
The full structure of $\pi_{2N}(\lambda)$ as a function of
$\M_2,\dots,\M_{2N}$ remains to be determined, however.

%%%%%%%%%%%%%%%%%%%%%%%%%%%%%%%%%%%%%%%%%%%%%%%%%%%%%%%%%%%%%%%%%%%%%%%%%
\subsection{$Q_{2N}^{res}(\lambda)$ and the universal recursive
formula for $Q_{2N}$}

The representation formula \eqref{pi-ex} is closely related to the
following formula for the so-called $Q$-curvature polynomial
\cite{juhl-book}, \cite{BJ}
\begin{equation}\label{Q-poly}
Q_{2N}^{res}(\lambda) \st -(-1)^N D_{2N}^{res}(\lambda)(1).
\end{equation}
We have
\begin{multline}\label{Q-pol-ex}
Q_{2N}^{res}(\lambda) = - \frac{1}{(N\!-\!1)!} \lambda \Big(
\prod_{j=1}^{N-1} \left(\lambda\!+\!\f\!-\!N\!-\!j\right) (-1)^N Q_{2N} \\
+ \sum_{|I|+a=N} \frac{b_N(\lambda\!+\!\f\!-\!N)}
{(\lambda\!+\!\f\!-\!2N\!+\!I_l)} m_{(I,a)} (-1)^a P_{2I}(Q_{2a})
\Big).
\end{multline}
Similarly as Theorem \ref{D-rep} follows from the {\em full} systems
of respective $2N$ factorization identities for residue families of
order $2N$, the formula \eqref{Q-pol-ex} follows from (the
consequences for the $Q$-curvature polynomials of) the respective
systems of the first $N$ factorization identities for residue
families of order $2N$ (and the vanishing result
$Q_{2N}^{res}(0)=0$\footnote{see Theorem 1.6.6 in \cite{BJ}}). Note
that the equality of the two expressions for the leading coefficient
of $Q_{2N}^{res}(\lambda)$ which follow from the definition
\eqref{Q-pol} and from \eqref{Q-pol-ex} is equivalent to the
universal recursive formula \eqref{URQ} for $Q_{2N}$. For more
details we refer to \cite{Q-recursive}.

%%%%%%%%%%%%%%%%%%%%%%%%%%%%%%%%%%%%%%%%%%%%%%%%%%%%%%%%%%%%%%%%%%%%%%%%%%%%
\subsection{Contributions with many derivatives}\label{many}

Theorem \ref{main-P} easily yields a description of the
contributions to $P_{2N}$ with strictly {\em more} than $2N-4$
derivatives. In fact, these are contained in the contributions to
\eqref{main-deco} which are defined by the compositions
$I=(1,\dots,1)$ and
$$
I_k = (\underbrace{1,\dots,1}_k,2,\underbrace{1,\dots,1}_{N-k-2}),
\; k=0,\dots,N\!-\!2
$$
of size $N$. Now $I=(1,\dots,1)$, i.e., $\M_2^N$ yields the
contributions $\Delta^N$ and
$$
-\left(\f\!-\!1\right) \sum_{k=0}^{N-1} \Delta^k (\J \Delta^{N-k-1}).
$$
Moreover, using
$$
n_{I_k} = \binom{k\!+\!1}{1} \binom{N\!-\!k\!-\!1}{1} =
(k\!+\!1)(N\!-\!k\!-\!1),
$$
we obtain the contribution
$$
-\sum_{k=0}^{N-2} 4 (k\!+\!1)(N\!-\!k\!-\!1) \Delta^k \delta (\Rho
d) \Delta^{N-k-2}.
$$
Thus, by $\delta (\Rho d) = -(\Rho,\Hess)-(d\J,d)$, the desired
contributions are covered by
\begin{multline}\label{many-A}
\Delta^N + 4 \sum_{k=0}^{N-2} (k\!+\!1)(N\!-\!k\!-\!1) \Delta^k [
(\Rho, \Hess) + (d\J,d)] \Delta^{N-k-2} \\[-3mm]
- \left(\f\!-\!1\right) \sum_{k=0}^{N-1} \Delta^k (\J
\Delta^{N-k-1}).
\end{multline}
We observe that \eqref{many-A} coincides with the sum
\begin{multline}\label{many-B}
\Delta^N + 2 \sum_{k=0}^{N-2} (k\!+\!1)(N\!-\!k\!-\!1) \Delta^k
\left[ 2 (\Rho,\Hess) + (d\J,d) \right] \Delta^{N-k-2} \\[-3mm]
+ \sum_{k=0}^{N-1} \left(2k\!+\!2\!-\!N\!-\!\f\right) \Delta^k (\J
\Delta^{N-k-1}),
\end{multline}
up to terms with at most $2N-4$ derivatives.\footnote{\label{F} The
operator $2(\Rho,\Hess) + (d\J,d) = - 2 \delta(\Rho d) - (d\J,d)$ is
the metric variation $2 (d/dt)|_{t=0} (\Delta_{g-t\Rho})$ of the
Laplacian.} In fact, the difference of \eqref{many-A} and
\eqref{many-B} equals
\begin{equation*}
2 \sum_{k=0}^{N-2} (k\!+\!1)(N\!-\!k\!-\!1) \Delta^k (d\J,d)
\Delta^{N-k-2} - \sum_{k=0}^{N-1} \left(2k\!+\!1\!-\!N\right)
\Delta^k (\J \Delta^{N-k-1}).
\end{equation*}
Using $2(d\J,du) = \Delta (\J u) - \J \Delta(u) - \Delta(\J)u$, we
find that (modulo terms with at most $2N-4$ derivatives) the latter
sum equals
\begin{multline*}
\sum_{k=0}^{N-2} (k\!+\!1)(N\!-\!k\!-\!1) \Delta^{k+1} (\J
\Delta^{N-k-2}) - \sum_{k=0}^{N-2} (k\!+\!1)(N\!-\!k\!-\!1) \Delta^k
(\J \Delta^{N-k-1}) \\[-2mm] - \sum_{k=0}^{N-1} \left(2k\!+\!1\!-\!N
\right) \Delta^k (\J \Delta^{N-k-1}) = 0.
\end{multline*}
An alternative direct proof of \eqref{many-B} can be found in
\cite{michel} (Proposition 2.1). We also observe that \eqref{many-A}
can be written in the form
$$
\Delta^N + \sum_{k=0}^{N-2} \Delta^k \delta((n\!-\!2)\J - 4
(k\!+\!1)(N\!-\!k\!+\!1) \Rho) d \Delta^{N-k-2} +
\left(\f\!-\!1\right) \sum_{k=1}^{N-2} \Delta^k (\J \Delta^{N-1-k}).
$$
This formula explains the structure of the first four terms in
\eqref{P6-fully-gen}. In order to prove the assertion, we consider
the difference
$$
-\left(\f\!-\!1\right) \sum_{k=0}^{N-1} \Delta^k (\J \Delta^{N-k-1})
- \sum_{k=0}^{N-2} \Delta^k \delta((n\!-\!2)\J) d \Delta^{N-k-2}.
$$
In view of
$$
\Delta^{N-1} (\J u) = \Delta^{N-2} ( \Delta(\J)u + 2 (d\J,du) + \J
\Delta(u))
$$
and $\delta (\J d) = -\J \Delta - (d\J,d)$, the last two terms of
the first sum cancel against the last term of the second sum, up to
a term with at most $2N-4$ derivatives. The remaining differences
are given by
\begin{equation*}
-\left(\f\!-\!1\right) \Delta^k (\J \Delta^{N-k-1}) -
(n\!-\!2)\Delta^k \delta (\J d) \Delta^{N-k-2} =
\left(\f\!-\!1\right) \Delta^{k+1} ( \J \Delta^{N-k-2})
\end{equation*}
for $k=0,\dots,N-3$, up to terms of order $\le 2N-4$. This proves
the claim.

The direct determination of the above contributions from the ambient
metric (as in \cite{michel}) involves the {\em first} metric
variation of the Laplacian (see footnote \ref{F}). The terms with
fewer derivatives also contain contributions which come from
higher-order metric variations of the Laplacian. From this point of
view, the method of residue families can be regarded as a way to
avoid dealing with these metric variations.

%%%%%%%%%%%%%%%%%%%%%%%%%%%%%%%%%%%%%%%%%%%%%%%%%%%%%%%%%%%%%%%%%
\subsection{The restriction property revisited}\label{restriction}

In Section \ref{RP-and-CR}, we derived the restriction property
\begin{equation}\label{M-rest}
\M_{2N} i^* = i^* \bar{\M}_{2N} \quad \mbox{for $N \ge 2$}
\end{equation}
from the factorization relations of residue families.\footnote{Of
course, the parameter $N$ is also subject to the restriction $2N \le
n$ for even $n$.} But Theorem \ref{main-P} identifies the generating
function $\H(r)$ of the sequence $\M_2,\M_4, \dots$ with a natural
second-order differential operator. We show that the restriction
property directly can be derived from this identification.

For this purpose, we {\em define} a sequence of natural second-order
operators $\vartheta_{2N}(g)$ by the identity
\begin{equation}\label{theta-gen}
\sum_{N \ge 1} \vartheta_{2N}(g) \frac{1}{(N\!-\!1)!^2}
\left(\frac{r^2}{4}\right)^{N-1} = \D(g)(r),
\end{equation}
where $\D(g)(r)$ is the operator on the right-hand side of
\eqref{GF-form}. A simple calculation shows that $\D(0) =
\vartheta_2 = P_2$.

\begin{prop}\label{rest-geom} For $N \ge 2$, the operators
$\vartheta_{2N}$ satisfy the restriction property
\begin{equation}\label{theta-rest}
i^* \bar{\vartheta}_{2N} = \vartheta_{2N} i^*.
\end{equation}
Here $\vartheta_{2N}$ and $\bar{\vartheta}_{2N}$ are defined with
respect to the metrics $g$ and $\bar{g}$.
\end{prop}

\begin{proof} Theorem \ref{double} implies that
\begin{equation}\label{g-rest}
g(0,s) = dr^2 + g(s).
\end{equation}
Hence, by the structure of the principal part of the right-hand side
of \eqref{theta-gen}, the relation \eqref{theta-rest} is equivalent
to the restriction property
\begin{equation}\label{theta-rest-c}
i^* \bar{\vartheta}_{2N}(1) = \vartheta_{2N}(1) \quad \mbox{for $N
\ge 2$}
\end{equation}
of the zeroth-order terms. A second consequence of
Eq.~\eqref{g-rest} is that the curvature quantities
$$
v(r,s) = \frac{vol(g(r,s))}{vol(g(r,0))} =
\frac{vol(g(r,s))}{vol(dr^2\!+\!g(r))} \quad \mbox{and} \quad v(r) =
\frac{vol(g(r))}{vol(g)}
$$
satisfy the relation
\begin{equation*}
v(0,s) = v(s).
\end{equation*}
Hence we also have
\begin{equation}\label{w-rest}
w(0,s) = w(s).
\end{equation}
Now the arguments in the proof of Theorem \ref{GF} given in Section
\ref{proof-1} (page \pageref{basic-argument}) prove that the
operators $\vartheta_{2N}$ satisfy the recursive relations
\begin{equation}\label{rec-theta}
\sum_{k=0}^{N-1} \left(2^k \frac{(N\!-\!1)!}{(N\!-\!1\!-\!k)!}
\right)^2 \vartheta_{2N-2k} (w_{2k}) = \left(\f\!-\!N\right)
(N\!-\!1)! N! 2^{2N} w_{2N}
\end{equation}
(see \eqref{reduced-w}). We apply these relations to prove
\eqref{theta-rest-c} using induction on $N$. First of all, Lemma
\ref{theta-4-direct} gives $i^* \bar{\vartheta}_4(1) =
\vartheta_4(1)$. Now assume that we have proved
$$
i^* \bar{\vartheta}_{2k}(1) = \vartheta_{2k}(1) \quad \mbox{for
$k=1,\dots,N$}.
$$
Then, by arguments as above, we also have
\begin{equation}\label{ind-assume}
i^* \bar{\vartheta}_{2k} = \vartheta_{2k}i^* \quad \mbox{for
$k=1,\dots,N$}.
\end{equation}
But these properties together with the relation
\begin{equation}\label{w-rec-gen}
i^* \bar{P}_2 (\bar{w}_{2N}) - P_2 (w_{2N}) = (2N\!+\!2) w_{2N+2}
\end{equation}
(see Lemma \ref{w-rec}) imply
$$
i^* \bar{\vartheta}_{2N+2}(1) = \vartheta_{2N+2}(1).
$$
In fact, by taking the difference of \eqref{rec-theta} for $g$ and
$\bar{g}$ (and $N$ replaced by $N\!+\!1$) and using $i^*
\bar{w}_{2N+2} = w_{2N+2}$ (see \eqref{w-rest}) and
\eqref{ind-assume}, we obtain
\begin{multline*}
i^* \bar{\vartheta}_{2N+2}(1) - \vartheta_{2N+2} (1) + (2^N N!)^2
(i^* \bar{P}_2(w_{2N}) - P_2(w_{2N})) = N! (N\!+\!1)! 2^{2N+1}
w_{2N+2}.
\end{multline*}
Now \eqref{w-rec-gen} shows that
$$
i^* \bar{\vartheta}_{2N+2}(1) = \vartheta_{2N+2}(1).
$$
This completes the induction.
\end{proof}

In order to complete the proof of Proposition \ref{rest-geom}, it
remains to prove the following two results.

\begin{lemm}\label{theta-4-direct} $i^* \bar{\vartheta}_4 = \vartheta_4$
\end{lemm}

\begin{proof} A direct calculation of the coefficient of $r^2$ on the
right-hand side of \eqref{theta-gen} shows that
$$
\vartheta_4 = (n\!-\!4) |\Rho|^2 + \J^2 - \Delta \J
$$
(see \eqref{mu4-ex}). Thus, the assertion is equivalent to
\begin{equation}\label{rp-mu4}
(n\!-\!3) i^* |\bar{\Rho}|^2 + i^* \bar{\J}^2 - i^* \bar{\Delta}
\bar{\J} = (n\!-\!4) |\Rho|^2 + \J^2 - \Delta \J.
\end{equation}
But
$$
i^* \bar{\J} = \J, \quad \bar{\Rho}|_{r=0} = \begin{pmatrix} 0 & 0 \\
0 & \Rho \end{pmatrix} \quad \mbox{and} \quad i^*
(\partial^2/\partial r^2)(\bar{\J}) = |\Rho|^2
$$
(for detailed proofs see \cite{juhl-book}, Section 6.11) imply that
$$
i^* |\bar{\Rho}|^2 = |\Rho|^2 \quad \mbox{and} \quad i^*
\bar{\Delta} \bar{\J} - \Delta \J = |\Rho|^2.
$$
This proves \eqref{rp-mu4}.
\end{proof}

\begin{lemm}\label{w-rec} For $N \ge 1$, we have
\begin{equation}\label{w-rec-flat}
i^* \bar{P}_2 (\bar{w}_{2N-2}) - P_2 (w_{2N-2}) = 2N w_{2N},
\end{equation}
where $P_2$ and $\bar{P}_2$ denote the respective Yamabe-operators
of $g$ and $\bar{g}$.
\end{lemm}

\begin{proof} We reformulate the identities \eqref{w-rec-flat}
in the form
$$
i^* \bar{P}_2(w(r,s)) - P_2(w(s)) = s^{-1} \partial w(s)/\partial s.
$$
Now an easy calculation using $i^* \bar{\J} = \J$ and $w(0,s) =
w(s)$ (see \eqref{w-rest}) shows that this identity is equivalent to
\eqref{reduced-3}. The proof is complete.
\end{proof}

%%%%%%%%%%%%%%%%%%%%%%%%%%%%%%%%%%%%%%%%%%%%%%%%%%%%%%%%%%%%%%%%%%%%%%%%
\subsection{$\M_{2N}$ and $\bar{\M}_{2N}$ for round spheres}

We determine explicit summation formulas for the operators $\M_{2N}$
and $\bar{\M}_{2N}$ for round spheres $\S^n$. First, we recall the
summation formula (\cite{juhl-power}, Theorem 6.1)
$$
\M_{2N} = N!(N\!-\!1)! P_2, \; N \ge 1.
$$
Next, we have
$$
\bar{g} = dr^2 + \left(1\!-\!r^2/4\right)^2 g_{\S^n},
$$
and Theorem \ref{main-P} implies the following result.

\begin{prop} We have
\begin{equation}\label{sum-form}
\bar{\M}_{2N} = (N\!-\!1)! N! (1\!-\!r^2/4)^{-N-1} P_2(\S^n)
\end{equation}
for $N \ge 2$.
\end{prop}

\begin{proof} Theorem \ref{double} gives
$$
g(r,s) = dr^2 + (1\!-\!(r^2\!+\!s^2)/4)^2 g_{\S^n}.
$$
Now Theorem \ref{main-P} and the identity
\begin{equation}\label{series}
\left(1\!-\!\frac{r^2\!+\!s^2}{4}\right)^{-2} = \sum_{N \ge 1} N
\left(1\!-\!\frac{r^2}{4}\right)^{-N-1}\left(\frac{s^2}{4}\right)^{N-1}
\end{equation}
show that for $N \ge 2$ the principal part of $\bar{\M}_{2N}$ equals
$$
(N\!-\!1)! N! \left(1\!-\!r^2/4\right)^{-N-1} \Delta_{\S^n}.
$$
Moreover, using
$$
w(r,s) = (1\!-\!(r^2\!+\!s^2)/4)^\f (1\!-\!r^2/4)^{-\f}
$$
and
$$
-\delta_{\bar{g}} (g(r,s)^{-1}d) = \frac{\partial^2}{\partial r^2}
-\f r \left(1\!-\!\frac{r^2}{4}\right)^{-1}\frac{\partial}{\partial
r} + \left(1\!-\!\frac{r^2\!+\!s^2}{4}\right)^{-2} \Delta_{\S^n},
$$
a straightforward calculation shows that the potential of the
second-order operator on the right-hand side of \eqref{GF-form}
equals
\begin{equation}\label{pot}
-\left(1\!-\!\frac{r^2\!+\!s^2}{4}\right)^{-2} \f
\left(\f\!-\!1\right) + \left(1\!-\!\frac{r^2}{4}\right)^{-2} \f
\left( \frac{n\!-\!1}{2} \frac{r^2}{4} - \frac{1}{2}\right).
\end{equation}
Now Eq.~\eqref{series} and Eq.~\eqref{pot} prove the assertion for
the zeroth-order terms in \eqref{sum-form}. \end{proof}

Note that the potential \eqref{pot} restricts to
$$
-\f \frac{n\!-\!1}{2} \left(1\!-\!\frac{r^2}{4}\right)^{-1} = -
\frac{n\!-\!1}{2} \bar{\J}
$$
for $s=0$. Here we use that
$$
\bar{\J} = \J (\bar{g}) = - \frac{1}{2r} \tr
\left(\frac{\dot{g}_r}{g_r}\right) = \f \left( 1\!-\!\frac{r^2}{4}
\right)^{-1}.
$$
Note also that the formula \eqref{sum-form} confirms the restriction
property and the commutator relations in that special case.

%%%%%%%%%%%%%%%%%%%%%%%%%%%%%%%%%%%%%%%%%%%%%%%%%%%%%%%%%%%%%%%%%%%%%%%%%
\section{Appendix}\label{app}

In this appendix, we collect a number of important formulas and give
the details of some central calculations. In particular, we give
direct proofs of some low-order special cases of the symmetry
relation $g_{(2a,2b)} = g_{(2b,2a)}$ (see Eq.~\eqref{sy}), and of
the relation between $g_{(4,4)}$ and $g_{(8)}$. Moreover, we extract
from these considerations an improved version of an algorithm of
\cite{G-ext}.

%%%%%%%%%%%%%%%%%%%%%%%%%%%%%%%%%%%%%%%%%%%%%%%%%%%%%%%%%%%%%%%%%%%%%%%%
\subsection{The operators $\M_{2N}$ for $N \le 5$}\label{averages}

In the present section, we display the explicit definitions of the
low-order cases $N \le 5$ of the operators $\M_{2N}$. These
operators are given by the formulas
\begin{equation}\label{M2-def}
\M_2 = P_2,
\end{equation}
\begin{equation}\label{M4-def}
\M_4 = P_4 - P_2^2,
\end{equation}
\begin{equation}\label{M6-def}
\M_6 = P_6 - (2 P_2 P_4 + 2 P_4 P_2) + 3 P_2^3,
\end{equation}
\begin{equation}\label{M8-def}
\M_8 = P_8 - (3 P_2 P_6 + 3 P_6 P_2 + 9 P_4^2) + (12 P_2^2 P_4 + 12
P_4 P_2^2 + 8 P_2 P_4 P_2) - 18 P_2^4
\end{equation}
and
\begin{multline}\label{M10-def}
\M_{10} = P_{10} - (4 P_2 P_8 + 4 P_8 P_2 + 24 P_4 P_6 + 24 P_6 P_4)
\\ + (60 P_2 P_4^2 + 60 P_4^2 P_2 + 30 P_2^2 P_6 + 30 P_6 P_2^2 + 15
P_2 P_6 P_2 + 80 P_4 P_2 P_4) \\- (120 P_2^3 P_4 + 120 P_4 P_2^3 +
80 P_2^2 P_4 P_2 + 80 P_2 P_4 P_2^2) + 180 P_2^5.
\end{multline}
The sums in \eqref{M6-def}, \eqref{M8-def} and \eqref{M10-def}
involve $4$, $8$ and $16$ terms, respectively. In these formulas,
the parentheses contain the respective contributions with the same
number of factors.

%%%%%%%%%%%%%%%%%%%%%%%%%%%%%%%%%%%%%%%%%%%%%%%%%%%%%%%%%%%%%%%%%%%%%%%%%%
\subsection{The restriction property of $g_{(4)}$}\label{g-4}

Here we prove that
$$
g_{(4)} = \frac{1}{4} \left(\Rho^2 \!-\! \frac{\B}{n\!-\!4}\right)
$$
satisfies the restriction property
\begin{equation}\label{rest-g4}
\bar{g}_{(4)}|_{r=0} = g_4.
\end{equation}
In other words, we prove the symmetry relation $g_{(4,0)} =
g_{(0,4)}$. First of all,
$$
\bar{\Rho}|_{r=0} = \Rho,
$$
i.e., $g_{(2,0)} = g_{(0,2)}$ (see Lemma \ref{schouten-bar}) shows
that the relation \eqref{rest-g4} is equivalent to the restriction
property
\begin{equation}\label{bach-rest}
(n\!-\!4) \B(dr^2\!+\!g_r)|_{r=0} = (n\!-\!3) \B(g)
\end{equation}
of the Bach tensor. In order to prove \eqref{bach-rest}, we use the
transformation rule
\begin{equation}\label{bach-conform}
e^{2\varphi} \hat{\B}_{ij} = \B_{ij} - (n\!-\!4) (C_{ikj} + C_{jki})
\varphi^k + (n\!-\!4) W_{irsj} \varphi^r \varphi^s,
\end{equation}
where $C$ and $W$ denote the Cotton and Weyl tensor,
respectively.\footnote{Here we use the conventions of
\cite{juhl-book}. In particular, the curvature tensor $R$ and the
Weyl tensor $W$ have opposite signs as in \cite{FG-final}. Moreover,
we define the Cotton tensor by $C(X,Y,Z) = \nabla_X(\Rho)(Y,Z) -
\nabla_Y(\Rho)(X,Z)$.} Since $r \partial/\partial r$ is dual to
$d\log r = dr/r$, \eqref{bach-conform} implies that
\begin{multline*}
r^2 \B(dr^2\!+\!g_r)(X,Y) = \B(g_+)(X,Y) \\ - (n\!-\!3) (C(g_+)(X,r
\partial/\partial r,Y) + C(g_+)(Y,r\partial/\partial r,X)) \\ +
(n\!-\!3) W(g_+) (X,r\partial/\partial r,r\partial/\partial r,Y).
\end{multline*}
But the Bach tensor and the Cotton tensor of the Einstein metric
$g_+$ vanish. Thus, we find
$$
\B(dr^2\!+\!g_r)(X,Y) = (n\!-\!3) W(g_+)(X,\partial/\partial
r,\partial/\partial r,Y).
$$
Hence the assertion
$$
i^* \B(dr^2\!+\!g_r) = (n\!-\!3)/(n\!-\!4) \B(g)
$$
follows from the fact that (on $r=0$) the Weyl-curvature term yields
$\B(g)/(n-4)$. In order to prove this, we apply the decomposition $R
= W - \Rho \owedge g$ for $g_+$. Using $\Rho(g_+) = -\frac{1}{2}
g_+$, we find
$$
W(g_+)(X, \partial/\partial r, \partial/\partial r,Y) =
R(g_+)(X,\partial/\partial r, \partial/\partial r,Y) + \frac{1}{r^2}
g_+(X,Y)
$$
for $X,Y \in \mathfrak{X}(M)$. Now a routine calculation
\cite{diplom-fisch} shows that
$$
R(g_+)(\cdot,\partial/\partial r, \partial/\partial r,\cdot) = -
\frac{1}{r^2} g_+(\cdot,\cdot) + \tilde{R}(\cdot,\partial/\partial
\rho,\partial/\partial \rho,\cdot) |_{t=1,\rho=-r^2/2}.
$$
Here $\tilde{R}$ is the curvature of the ambient metric
$$
\tilde{g} = 2 t dt d\rho + 2\rho dt^2 + t^2 g(\rho), \quad g(0) = g
$$
of $g$, and the dots indicate vector fields in $\mathfrak{X}(M)$.
Hence
$$
W(g_+)(\cdot, \partial/\partial r, \partial/\partial r,\cdot) =
\tilde{R}(\cdot,\partial/\partial \rho,\partial/\partial
\rho,\cdot)|_{t=1,\rho=-r^2/2},
$$
i.e.,
\begin{equation}\label{Omega-R}
- \bar{\Omega}^{(1)}(X,Y) = \tilde{R}(X,\partial/\partial
\rho,\partial/\partial \rho,Y) |_{t=1,\rho=-r^2/2}, \; X,Y \in
\mathfrak{X}(M),
\end{equation}
where $\Omega^{(1)} = -\B/(n\!-\!4)$ is the first extended
obstruction tensor \cite{G-ext}. Therefore, the identity
$$
\tilde{R}(\cdot,\partial/\partial \rho,\partial/\partial
\rho,\cdot)|_{t=1,\rho=0} = \B/(n\!-\!4) = - \Omega^{(1)}
$$
(see \cite{FG-final}) proves the assertion. Finally,
$$
\B(dr^2\!+\!g_r)(X,\partial/\partial r) = (n\!-\!3)
W(g_+)(X,\partial/\partial r,\partial/\partial r,\partial/\partial
r) = 0
$$
shows that the restriction of $\B(dr^2+g_r)$ to $r=0$ has no
non-trivial normal components. This completes the proof of
\eqref{bach-rest}.

%%%%%%%%%%%%%%%%%%%%%%%%%%%%%%%%%%%%%%%%%%%%%%%%%%%%%%%%%%%%%%%%%%%%%%
\subsection{The symmetry relation $g_{(4,2)} = g_{(2,4)}$}\label{g-24}

In the present section, we prove the symmetry relation
\begin{equation}\label{partial-symm}
g_{(4,2)} = g_{(2,4)}.
\end{equation}
We consider the tensor
\begin{equation}\label{bar-c4}
\frac{1}{4} \left(\bar{\Rho}^2 \!-\! \frac{\bar{\B}}{n\!-\!3}\right)
\end{equation}
as a function of $r$. By the above results, the constant term of its
Taylor series is
$$
\frac{1}{4} \left(\Rho^2 \!-\! \frac{\B}{n\!-\!4}\right).
$$
We determine its Taylor coefficient $g_{(2,4)}$ of $r^2$. Using
\begin{equation}\label{P-full}
\bar{\Rho} = - \frac{1}{2r} \dot{g_r} = \Rho - \frac{r^2}{2}
\left(\Rho^2 \!-\! \frac{\B}{n\!-\!4}\right) + \cdots
\end{equation}
(see Lemma \ref{schouten-bar}), we find that $\bar{\Rho}^2$ is given
by
\begin{multline*}
\bar{\Rho}_{ir} \bar{g}^{rk} \bar{\Rho}_{kj} = \left(\Rho_{ir} -
\frac{r^2}{2} \left(\Rho^2_{ir} \!-\! \frac{\B_{ir}}{n\!-\!4}
\right) + \cdots \right)
\left( g^{rk} + r^2 \Rho^{rk} + \cdots \right) \\
\times \left(\Rho_{kj} - \frac{r^2}{2} \left(\Rho^2_{kj} \!-\!
\frac{\B_{kj}}{n\!-\!4} \right) + \cdots\right).
\end{multline*}
The Taylor coefficient of $r^2$ of the latter product is the sum of
$$
\Rho_{ir} \Rho^{rk} \Rho_{kj} = (\Rho^3)_{ij}
$$
and
$$
- (\Rho^3)_{ij} + \frac{1}{2(n\!-\!4)} (\B_{ir} \Rho_j^r + \Rho_i^r
\B_{rj}).
$$
Hence, in terms of $\Omega^{(1)} = - \B/(n\!-\!4)$, the Taylor
coefficient of $r^2$ of $\bar{\Rho}^2/4$ is given by
\begin{equation}\label{part-1}
-\frac{1}{8} \left( \Rho^r_i \Omega_{jr}^{(1)} + \Rho^r_j
\Omega_{ir}^{(1)}\right).
\end{equation}
Next, we use the relation
\begin{equation}\label{bach-bar}
\frac{\bar{\B}}{n\!-\!3} =
\tilde{R}(\cdot,e_\infty,e_\infty,\cdot)|_{t=1,\rho=-r^2/2} =
\tilde{R}_{\cdot\infty\infty\cdot}|_{t=1,\rho=-r^2/2}, \quad
e_\infty = \partial/\partial \rho
\end{equation}
to determine the Taylor coefficient of $r^2$ of
$-\bar{\B}/(4(n\!-\!3))$. Now
$$
\tilde{R}_{i \infty \infty j} |_{t=1,\rho=-r^2/2} = \tilde{R}_{i
\infty \infty j} |_{t=1,\rho=0} -
\partial \tilde{R}_{i \infty \infty j}/\partial \rho
|_{t=1,\rho=0} \frac{r^2}{2} + \cdots.
$$
But $\tilde{\nabla}_{\rho}(\partial/\partial \rho) = 0$ implies
\begin{multline*}
\partial \tilde{R}_{i \infty \infty j}/\partial \rho =
\tilde{\nabla}_{\rho}(\tilde{R})(e_i,e_\infty,e_\infty,e_j) \\
+ \tilde{R}(\tilde{\nabla}_{\rho}(e_i),e_\infty,e_\infty,e_j) +
\tilde{R}(e_i,e_\infty,e_\infty,\tilde{\nabla}_{\rho}(e_j)).
\end{multline*}
A calculation of Christoffel symbols (\cite{FG-final}, Eq.~(3.16) or
\cite{diplom-fisch}) shows that
\begin{equation}\label{christ}
\tilde{\nabla}_{\rho}(e_i)|_{t=1,\rho=0} = \Rho_i^k e_k.
\end{equation}
Hence
$$
\partial \tilde{R}_{i \infty \infty j}/\partial \rho|_{t=1,\rho=0} =
-\Omega^{(2)}_{ij} - \Rho_i^k \Omega^{(1)}_{kj} - \Rho_j^k
\Omega^{(1)}_{ik}.
$$
(see Eq.~\eqref{e-obst}). Therefore, the Taylor coefficient of $r^2$
of $-\bar{\B}_{ij}/(4(n\!-\!3))$ is
\begin{equation}\label{part-2}
-\frac{1}{8} \left( \Omega_{ij}^{(2)} + \Rho_i^k \Omega^{(1)}_{kj} +
\Rho_j^k \Omega^{(1)}_{ik} \right).
\end{equation}
Summarizing \eqref{part-1} and \eqref{part-2}, we find the
coefficient
\begin{equation}\label{both-parts}
-\frac{1}{8} \left( \Omega_{ij}^{(2)} + 2 \Rho_i^k \Omega^{(1)}_{kj}
+ 2\Rho_j^k \Omega^{(1)}_{ik} \right).
\end{equation}
But the quantity \eqref{both-parts} coincides with the coefficient
$g_{(4,2)}$ of $r^4$ in the Taylor expansion of $-\bar{\Rho}$. In
fact, the latter assertion follows by combining the formula (see
\cite{G-ext}, Eq.~(2.22))
$$
g_r = g - \Rho r^2 + \left(\Rho^2\!+\!\Omega^{(1)}\right)
\frac{r^4}{4} - \left(\Omega^{(2)} \!+\! 2 (\Rho \Omega^{(1)} \!+\!
\Omega^{(1)} \Rho)\right) \frac{r^6}{24} + \cdots
$$
with Lemma \ref{schouten-bar}. This proves the symmetry
\eqref{partial-symm}.

%%%%%%%%%%%%%%%%%%%%%%%%%%%%%%%%%%%%%%%%%%%%%%%%%%%%%%%%%%%%%%%%%%%%%%%%
\subsection{The coefficient $g_{(4,4)}$}\label{diagonal-4}

Theorem \ref{double} implies that
\begin{equation}\label{eight}
g_{(4,4)} = \binom{4}{2} g_{(8)}.
\end{equation}
Here we present a direct proof of this relation. The following
arguments also serve as an illustration of an algorithm described in
Section \ref{more-double}.

In order to make the left-hand side of \eqref{eight} explicit, we
have to determine the coefficient of $r^4$ of
\begin{equation}\label{44}
\bar{\Rho}^2 + \bar{\Omega}^{(1)},
\end{equation}
First, using
$$
\bar{\Rho} = \Rho - \frac{r^2}{2} (\Rho^2 + \Omega^{(1)}) +
\frac{r^4}{8} \left(\Omega^{(2)} + 2 (\Rho \Omega^{(1)} \!+\!
\Omega^{(1)} \Rho) \right) + \cdots
$$
(Lemma \ref{schouten-bar}), we find that $\bar{\Rho}^2$ yields the
contribution
$$
r^4 \left( \frac{1}{8} (\Omega^{(2)} \Rho \!+\! \Rho \Omega^{(2)}) +
\frac{1}{4} \Rho \Omega^{(1)} \Rho + \frac{1}{4} \Omega^{(1)}
\Omega^{(1)} \right).
$$
Next, using
$$
\bar{\Omega}^{(1)}_{ij} = \Omega^{(1)}_{ij}(dr^2\!+\!g_r) = -
\tilde{R}_{i \infty \infty j}|_{t=1, \rho=-r^2/2}
$$
(see Eq.~\eqref{Omega-R}), we find that the second term in
\eqref{44} yields the contribution
$$
- \frac{1}{8} r^4 (\partial/\partial \rho)^2 (\tilde{R}_{i \infty
\infty j}) |_{t=1,\rho=0}.
$$
But
\begin{align*}
(\partial/\partial \rho)^2 (\tilde{R}_{i \infty \infty j}) & =
\tilde {\nabla}^2_{\rho}(\tilde{R})_{i \infty \infty j} \\
& + 2 \tilde{\nabla}_{\rho}
(\tilde{R})(\tilde{\nabla}_{\rho}(\partial_i),\partial_\infty,
\partial_\infty,\partial_j) + 2 \tilde{\nabla}_{\rho}
(\tilde{R})(\partial_i,\partial_\infty,\partial_\infty,
\tilde{\nabla}_{\rho}(\partial_j)) \\
& + \tilde{R}(\tilde{\nabla}_\rho^2(\partial_i),\partial_\infty,
\partial_\infty,\partial_j) + \tilde{R}(\partial_i,\partial_\infty,\partial_\infty,
\tilde{\nabla}_\rho^2(\partial_j)) \\
& + 2 \tilde{R}(\tilde{\nabla}_\rho(\partial_i),\partial_\infty,
\partial_\infty, \tilde{\nabla}_\rho(\partial_j)).
\end{align*}
By definition (see Eq.~\eqref{e-obst}), we have
$$
-\tilde{\nabla}_\rho (\tilde{R})_{i \infty \infty j} |_{t=1,\rho=0}
= \Omega^{(2)}_{ij} \; \mbox{and} \; -\tilde{\nabla}^2_\rho
(\tilde{R})_{i \infty \infty j} |_{t=1,\rho=0} = \Omega^{(3)}_{ij};
$$
recall that our sign convention for $\tilde{R}$ is opposite to that
of \cite{G-ext}. Moreover, using
$$
\tilde{\nabla}_\rho (\partial_i) = \frac{1}{2} g^{kr} \dot{g}_{ir}
\partial_k, \quad g = g(\rho)
$$
(\cite{FG-final}, Eq.~(3.16) or \cite{diplom-fisch}), we obtain the
relations
\begin{equation*}
\tilde{\nabla}_{\rho}(\partial_i)|_{t=1,\rho=0} = \Rho_i^k
\partial_k
\end{equation*}
(see Eq.~\eqref{christ}) and
$$
\tilde{\nabla}^2_\rho (\partial_i)|_{t=1,\rho=0} =
(\Omega^{(1)})_i^k \partial_k.
$$
In particular,
$$
(\tilde{R}(\tilde{\nabla}_\rho^2 (\partial_i),\partial_\infty,
\partial_\infty,\partial_j) + \tilde{R}(\partial_i,\partial_\infty,\partial_\infty,
\tilde{\nabla}_\rho^2 (\partial_j)))|_{t=1,\rho=0} = - 2
(\Omega^{(1)} \Omega^{(1)})_{ij}.
$$
These results show that the coefficient of $r^4$ of
$\bar{\Omega}^{(1)}$ is given by the sum
$$
\frac{1}{8} \Omega^{(3)} + \frac{1}{4} (\Omega^{(2)} \Rho + \Rho
\Omega^{(2)}) + \frac{1}{4} \Omega^{(1)} \Omega^{(1)} + \frac{1}{4}
\Rho \Omega^{(1)} \Rho.
$$
Summarizing the contributions of $\bar{\Rho}^2$ and
$\bar{\Omega}^{(1)}$, we obtain the formula
$$
g_{(4,4)} = \frac{1}{32} \left( \Omega^{(3)} + 3 (\Omega^{(2)} \Rho
\!+\! \Rho \Omega^{(2)}) + 4 \Omega^{(1)} \Omega^{(1)} + 4 \Rho
\Omega^{(1)} \Rho \right).
$$
But, according to \cite{G-ext}, Eq.~(2.22), we have
$$
g_{(8)} = \frac{1}{4! 2^3} \left( \Omega^{(3)} + 3 (\Omega^{(2)}
\Rho \!+\! \Rho \Omega^{(2)}) + 4 \Omega^{(1)} \Omega^{(1)} + 4 \Rho
\Omega^{(1)} \Rho \right).
$$
This proves the relation \eqref{eight}.

%%%%%%%%%%%%%%%%%%%%%%%%%%%%%%%%%%%%%%%%%%%%%%%%%%%%%%%%%%%%%%%%%%%%%%%%
\subsection{Some further consequences of Theorem \ref{double}}
\label{more-double}

First, we prove a formula for the Schouten tensor of $\bar{g} =
dr^2+g_r$. This result (with an alternative proof) can be found also
in \cite{juhl-book}, Lemma 6.11.2.

\begin{lemm}\label{schouten-bar}
$$
\bar{\Rho} = -\frac{1}{2r} (\partial/ \partial r) (g_r).
$$
In particular,
$$
\bar{\Rho}|_{r=0} = \Rho.
$$
\end{lemm}

\begin{proof} We use the fact that $-\bar{\Rho}$ is the coefficient of
$s^2$ in the expansion of $g_{++}$. Now, Theorem \ref{double}
implies that $-\bar{\Rho}$ has the expansion
$$
\sum_{N \ge 1} N r^{2N-2} g_{(2N)}.
$$
But the latter sum equals
$$
\frac{1}{2r} \, \partial \left( \sum_{N \ge 0} r^{2N} g_{(2N)}
\right)/\partial r.
$$
This proves the first assertion. The second assertion follows
by using $g_{(2)} = -\Rho$.
\end{proof}

Next, we recall that by \cite{G-ext} the extended obstruction
tensors\footnote{We recall that the sign conventions for the
curvature tensor differ from those in \cite{G-ext}. By the
additional sign in \eqref{e-obst}, the extended obstruction tensors
coincide, however.}
\begin{equation}\label{e-obst}
\Omega^{(k)}_{ij} =
-\tilde{\nabla}^{(k-1)}_\infty(\tilde{R})_{\infty i j \infty}
|_{t=1,\rho=0}, \; k \ge 1
\end{equation}
together with the Schouten tensor $\Rho$, are the building blocks of
$g_+$ in the sense that all Taylor coefficients $g_{(2k)}$ are given
by universal expressions involving $\Rho$ and extended obstruction
tensors.

\begin{lemm}\label{obstruction-rest} All extended obstruction
tensors $\Omega^{(k)}$ have the restriction property
$$
\bar{\Omega}^{(k)}|_{r=0} = \Omega^{(k)}.
$$
In particular, $\bar{\Omega}^{(k)}|_{r=0}(\cdot,\partial/\partial r)
= 0$.
\end{lemm}

\begin{proof} Let $\bar{g}(r) = dr^2\!+\!g(r)$. By Lemma \ref{PE-rest},
the Taylor series
$$
E(\bar{g}(r))(s)|_{r=0} = \bar{g}(0) + s^2 \bar{g}_{(2)}|_{r=0} +
s^4 \bar{g}_{(4)}|_{r=0} + \cdots
$$
coincides with
$$
(dr^2\!+\!g) + s^2 g_{(2)} + s^4 g_{(4)} + \cdots.
$$
Hence we have the restriction relation
\begin{equation}\label{PET-rest}
\bar{g}_{(2N)} |_{r=0} = g_{(2N)}.
\end{equation}
But, for any metric $g$, the coefficient $g_{(2N)}$ can be written as
a universal expression in terms of the Schouten tensor $\Rho$ and the
extended obstruction tensors $\Omega^{(1)},\dots,\Omega^{(N-1)}$ (see
\cite{G-ext}, Eq.~(2.22)). Thus, the assertion follows from
\eqref{PET-rest} by induction.
\end{proof}

Finally, we sketch an algorithm which generates formulas for the
Taylor coefficients $g_{(2k)}$ in terms of $\Rho$ and extended
obstruction tensors. This algorithm can be considered as a refined
version of the algorithm in the proof of Theorem 1.2 in
\cite{G-ext}. On the one hand, Theorem \ref{double} implies
\begin{equation}\label{alg-1}
\frac{1}{4} (\bar{\Rho}^2 + \bar{\Omega}^{(1)}) = \sum_{N \ge 2}
\binom{N}{2} r^{2N-4} g_{(2N)}.
\end{equation}
On the other hand, Lemma \ref{schouten-bar} and the relation
\begin{equation}\label{alg-2}
\bar{\Omega}^{(1)}(\cdot,\cdot) = - \tilde{R}(\cdot,e_\infty,
e_\infty,\cdot) |_{t=1,\rho=-r^2/2}
\end{equation}
on tangential vectors of $M$ (see Eq.~\eqref{Omega-R}) can be used
to express the $(2k)^{\text{th}}$ Taylor coefficient of the
left-hand side in terms of $\Omega^{(k+1)}$, lower-order extended
obstruction tensors and $\Rho$. This yields the desired formulas for
the coefficients $g_{(2N)}$. For $g_{(8)}$, the details of this
algorithm are given in Section \ref{diagonal-4}.

%%%%%%%%%%%%%%%%%%%%%%%%%%%%%%%%%%%%%%%%%%%%%%%%%%%%%%%%%%%%%%%%%%%%%%%%%%%%%%

\end{document}